%% file: 000_main.tex
\documentclass[times,final]{elsarticle}

\usepackage{framed,multirow}

\usepackage{amssymb}
\usepackage{latexsym}

\usepackage{url}
\usepackage{xcolor}
\definecolor{newcolor}{rgb}{.8,.349,.1}

\usepackage{amsmath}
\usepackage{amsthm}
\input{001_commands}
\input{002_variables}

\input{003_acronyms}

\begin{document}


\begin{frontmatter}

\title{A Greedy Sensor Selection Algorithm for Hyperparameterized Linear Bayesian Inverse Problems}


\author[1]{Nicole {Aretz}}
\author[2]{Peng {Chen}}
\author[3]{Denise D. {Degen}}
\author[4]{Karen {Veroy}}

\address[1]{Oden Institute for Computational Engineering and Sciences, University of Texas at Austin, 201 E 24th St, Austin, TX 78712, USA}
\address[2]{School of Computational Science and Engineering, Georgia Institute of Technology, 756 W Peachtree St NW, Atlanta, GA 30308, USA}
\address[3]{Computational Geoscience, Geothermics, and Reservoir Geophysics, RWTH Aachen University, Mathieustr. 30, 52074 Aachen, Germany}
\address[4]{Center for Analysis, Scientific Computing and Applications, Department of Mathematics and Computer Science, Eindhoven University of Technology, 5612 AZ Eindhoven, The Netherlands}


\begin{abstract}
We consider optimal sensor placement for a family of linear Bayesian inverse problems characterized by a deterministic hyper-parameter.
The hyper-parameter describes distinct configurations in which measurements can be taken of the observed physical system.
To optimally reduce the uncertainty in the system's model with a single set of sensors, the initial sensor placement needs to account for the non-linear state changes of all admissible configurations.
We address this requirement through an observability coefficient which links the posteriors' uncertainties directly to the choice of sensors.
We propose a greedy sensor selection algorithm to iteratively improve the observability coefficient for all configurations through orthogonal matching pursuit.
The algorithm allows explicitly correlated noise models even for large sets of candidate sensors, and remains computationally efficient for high-dimensional forward models through model order reduction.
We demonstrate our approach on a large-scale geophysical model of the Perth Basin, and provide numerical studies regarding optimality and scalability with regard to classic optimal experimental design utility functions.
\end{abstract}


\end{frontmatter}



\input{10_introduction}

\input{20_problem}

\input{620_observability}

\input{623_eigenvalues}

\input{625_parameterreduction}

\input{624_surrogate}

\input{630_sensors}
\input{631_cholesky}

\input{632_obsComp}

\input{633_algorithm}

\input{640_results}

\input{641_model}
\input{642_restricted}

\input{643_seedComparison}

\input{643_unrestricted}

\input{650_conclusion}

\section*{Acknowledgments}
We would like to thank Tan Bui-Thanh, Youssef Marzouk, Francesco Silva, Andrew Stuart, Dariusz Ucinski, and Keyi Wu for very helpful discussions, and Florian Wellmann at the Institute for Computational Geoscience, Geothermics and Reservoir Geophysics at RWTH Aachen University for providing the Perth Basin Model.
This work was supported by the Excellence Initiative of the German federal and state governments and the German Research Foundation through Grants GSC 111 and 33849990/GRK2379 (IRTG Modern Inverse Problems).  This project has also received funding from the European Research Council (ERC) under the European Union's Horizon 2020 research and innovation programme (grant agreement n° 818473), the US Department of Energy (grant DE-SC0021239), and the US Air Force Office of Scientific Research (grant FA9550-21-1-0084). Peng Chen is partially supported by the NSF grant DMS \#2245674.


\bibliographystyle{model1-num-names}
\bibliography{BayesOED-Paper}



\end{document}

%% file: 001_commands.tex
\usepackage{acro}
\usepackage{showlabels} 
\usepackage{subcaption}

\usepackage[ruled,vlined]{algorithm2e}
\usepackage{algorithmic}

\newtheorem{remark}{Remark}
\newtheorem{proposition}{Proposition}

\usepackage{showlabels}
\usepackage[colorinlistoftodos,prependcaption,textsize=tiny]{todonotes}

\newcommand{\myred}[1]{{\color[rgb]{0.65,0.0,0.0} #1}}

\newcommand{\myorange}[1]{{\color[rgb]{.76,.39,.13} #1}}

\SetKwComment{tcp}{// }{}
\SetKwInput{KwData}{Data} 
\SetKwInput{KwResult}{Result} 
\SetKwInput{KwIn}{Input} 
\SetKwInput{KwOut}{Output}

\newcommand{\vertiii}[1]{{\left\vert\kern-0.25ex\left\vert\kern-0.25ex\left\vert #1 
    \right\vert\kern-0.25ex\right\vert\kern-0.25ex\right\vert}}

\newcommand{\mycurlyO}[1]{\mathcal{O}( #1 )}

\newcommand{\myveceval}[2]{\left[ #1 \right]_{#2}}
\newcommand{\mytranspose}[1]{\left[ #1 \right]^T}
\newcommand{\mymatrixinv}[1]{\left[ #1 \right]^{-1}}

\newcommand{\mymatrixsum}[1]{\left[ #1\right]}

\newcommand{\myzeromatrix}[1][]{\mathbf{0}_{#1}}

\newcommand{\myspan}{\text{span}}

\newcommand{\myabs}[1]{\lvert #1 \rvert}
\newcommand{\diag}{\text{diag}}

\newcommand{\myCov}{\text{\textbf{cov}}}

\newcommand{\trace}{\text{trace}}

\DeclareMathSymbol{\shortminus}{\mathbin}{AMSa}{"39}

%% file: 002_variables.tex
\newcommand{\myX}{\mathcal{X}}
\newcommand{\myXdual}{\myX'}

\newcommand{\myfedim}{N}

\newcommand{\myW}{\mathcal{W}_{\mypara}}

\newcommand{\myUfebasis}[1][\myUfebasisIndex]{\myutestfct_{#1}}

\newcommand{\mygeninner}[3]{\left< #1, #2 \right>_{#3}}
\newcommand{\myXinner}[2]{\mygeninner{#1}{#2}{\myX}}

\newcommand{\mygennorm}[2]{\left\| #1 \right\|_{#2}}

\newcommand{\myXnorm}[1]{\mygennorm{#1}{\myX}}

\newcommand{\myx}{x}
\newcommand{\myxtrue}{\myx_{\rm{true}}}
\newcommand{\mygenstate}{x}

\newcommand{\myutrue}{\myuvec_{\rm{true}}}
\newcommand{\myutestfct}{\varphi}
\newcommand{\myuvec}{\mathbf{u}}

\newcommand{\myek}[1][k]{\mathbf{e}_{#1}}

\newcommand{\myparadomain}{\mathcal{P}}
\newcommand{\myparadomainspace}{\mathbb{R}^{\myparadomainspacedim}}
\newcommand{\myparaspace}{\myparadomainspace}
\newcommand{\myparadomainspacedim}{p}

\newcommand{\myparadomaintrain}{\Xi_{\rm{train}}}

\newcommand{\mypararef}{\mypara_{\rm{ref}}}

\newcommand{\mydatadim}{K}
\newcommand{\myd}{\mathbf{d}}
\newcommand{\mydata}{\myd}

\newcommand{\myobs}[1]{L #1}
\newcommand{\myobsmap}{L}
\newcommand{\myextendedobsmap}[1][\ell]{[\myobsmap, #1]}

\newcommand{\myobscont}[1][\myobsmap]{\gamma_{#1}}
\newcommand{\myobsinfsup}[1][\mypara]{\beta_{\myobsmap|\mathcal{W}}(#1)}

\newcommand{\myobsrbinfsup}[1][\mypara]{\tilde{\beta}_{\myobsmap | \mathcal{W}}(#1)}

\newcommand{\myparatoobsmap}{G_{\myobsmap, \mypara}}
\newcommand{\myparatoobsmatrix}{\mathbf{G}_{\myobsmap, \mypara}}
\newcommand{\myparatoobs}[1]{G_{\myobsmap, \mypara}\left( #1 \right)}

\newcommand{\myparatoobsinfsup}[1][\mypara]{\beta_{G}(#1)}

\newcommand{\myparatoobsrbinfsup}[1][\mypara]{\tilde{\beta}_{G}(#1)}

\newcommand{\mynoise}[1][]{\varepsilon_{#1}}

\newcommand{\mynoiseinner}[2]{\mygeninner{#1}{#2}{\mynoiseCovInv}}
\newcommand{\mynoisenorm}[2][\mynoiseCovInv]{\mygennorm{#2}{#1}}
\newcommand{\mynoiseCovInv}{\mynoiseCov^{-1}}

\newcommand{\myetainf}{\underline{\eta}(\mypara)}
\newcommand{\myetainfBND}{\underline{\eta}}
\newcommand{\myetasup}{\overline{\eta}(\mypara)}

\newcommand{\mylibrary}{\mathcal{L}}
\newcommand{\mylibrarydim}{K_{\mathcal{L}}}

\newcommand{\mydatadimmax}{K_{\max}}

\newcommand{\myeigval}{\lambda}
\newcommand{\myeigvalmin}{\myeigval^{\min}}

\newcommand{\myiter}{K}

\newcommand{\myuiter}[1][\myiter]{\myu_{#1}}

\newcommand{\myxiter}[1][\myiter]{\myx_{#1}}

\newcommand{\mysensorset}{\mathcal{S}}

\newcommand{\myplaceholderprior}{\rm{pr}}
\newcommand{\myplaceholderpost}{\rm{post}}

\newcommand{\mymodel}[1][\mypara]{\mathcal{M}_{#1}}
\newcommand{\myRBmodel}{\tilde{\mathcal{M}}_{\mypara}}

\newcommand{\mypara}{\theta}

\newcommand{\mystate}[1][\myu]{\myx_{\mypara}(#1)}
\newcommand{\mystatepara}[2]{\myx_{#1}(#2)}
\newcommand{\myRBstate}[1][\myu]{\tilde{x}_{\mypara}(#1)}

\newcommand{\myu}{\mathbf{u}}
\newcommand{\myum}[1][m]{\psi_{#1}}
\newcommand{\myv}{\mathbf{v}}

\newcommand{\myU}{\mathbb{R}^{\myfedimU}}
\newcommand{\myfedimU}{M}

\newcommand{\myobskmap}[1][k]{\ell_{#1}}
\newcommand{\myobsk}[2][k]{\myobskmap[#1](#2)}

\newcommand{\myreg}{\sigma^2}

\newcommand{\mynoiseCov}[1][\myobsmap]{\Sigma_{#1}}
\newcommand{\mynoiseCovobs}[1][\myobsmap]{\mynoiseCov[#1]}

\newcommand{\mynoiseCovinv}[1][\myobsmap]{\Sigma_{#1}^{-1}}
\newcommand{\mynoiseCovinvobs}[1][\myobsmap]{\mynoiseCovinv[#1]}

\newcommand{\myprior}{\pi_{\myplaceholderprior}}
\newcommand{\mypriormeasure}{\mu_{\myplaceholderprior}}
\newcommand{\mypriormean}{\myu_{\myplaceholderprior}}
\newcommand{\mypriorCov}{\Sigma_{\myplaceholderprior}}
\newcommand{\mypriorCovInv}{\Sigma_{\myplaceholderprior}^{-1}}
\newcommand{\mypriorCovEig}[1][m]{\lambda_{\myplaceholderprior}^{#1}}
\newcommand{\mypriorinner}[2]{\mygeninner{#1}{#2}{\mypriorCovInv}}
\newcommand{\mypriornorm}[1]{\| #1 \|_{\mypriorCovInv}}
\newcommand{\mypriorpdf}{\myprior}

\newcommand{\myposterior}{\pi_{\myplaceholderpost}^{\myobsmap,\mypara}}
\newcommand{\myposteriormeasure}{\mu_{\myplaceholderpost}^{\myobsmap,\mypara}}
\newcommand{\myposteriormean}{\myu_{\myplaceholderpost}^{\myobsmap,\mypara}}
\newcommand{\mypostCov}[1][\myobsmap, \mypara]{\Sigma_{\myplaceholderpost}^{#1}}
\newcommand{\mypostCovInv}{\mymatrixinv{\mypostCov}}
\newcommand{\mypostCovEig}[1][m]{\lambda_{\myobsmap}^{\mypara, #1}}

\newcommand{\mylikelyhood}{\Phi_{\myobsmap}}

\newcommand{\mySVDU}{\mathbf{U}}
\newcommand{\mySVDSigma}{\mathbf{D}_{\rm{pr}}}
\newcommand{\mySVDSigmaEntry}[1][i]{\myeigval^{#1}_{\rm{pr}}}

\newcommand{\mynoiseChol}[1][\myobsmap]{\mathbf{C}_{#1}}
\newcommand{\mynoiseCholinv}[1][\myobsmap]{\mathbf{C}_{#1}^{-1}}
\newcommand{\mynoiseCholadj}[1][\myobsmap]{\mathbf{C}_{#1}^T}
\newcommand{\mynoiseCholinvadj}[1][\myobsmap]{\mathbf{C}_{#1}^{-T}}
\newcommand{\myCholvector}[1][\ell]{\mathbf{v}_{\myobsmap, #1}}
\newcommand{\myCholvectorw}[1][\ell]{\mathbf{c}_{\myobsmap, #1}}
\newcommand{\myCholvectorx}[1][\ell]{\mathbf{r}_{\myobsmap, #1}}
\newcommand{\myCholvar}[1][\ell]{v_{#1,#1}}
\newcommand{\myCholbeta}{c_{\ell, \ell}}

\newcommand{\myhypereps}{\varepsilon_{\mypara}}

\newcommand{\myparaiter}[1][\myiter]{\mypara_{#1}}

\newcommand{\myOmega}{\Omega}

\newcommand{\myOmegaPoint}{\chi}

\newcommand{\myInflow}{\Gamma_{\rm{In}}}

\newcommand{\myXtestfct}{\phi}

\newcommand{\myerrorboundPrimal}{\Delta(\mypara)}

%% file: 003_acronyms.tex
\DeclareAcronym{3dvar}{
short = 3D-VAR,
long = 3D-VAR,
tag = method,
}

\DeclareAcronym{4dvar}{
short = 4D-VAR,
long = 4D-VAR,
tag = method
}

\DeclareAcronym{ae}{
short = a.e.,
long = \text{almost every},
tag = abbrev
}

\DeclareAcronym{aoed}{
short = A-OED,
long = A-optimal experimental design,
tag = abbrev
}

\DeclareAcronym{aposteriori}{
short = \textit{a posteriori},
long = \myred{\textit{a posteriori}},
tag = shortcut
}

\DeclareAcronym{apriori}{
short = \textit{a priori},
long = \myred{\textit{a priori}},
tag = shortcut
}

\DeclareAcronym{bk}{
short = bk,
long = best knowledge,
tag = abbrev
}

\DeclareAcronym{bnb}{
short = BNB,
long = Banach-Ne\v{c}as-Babu\v{s}ca,
tag = abbrev
}

\DeclareAcronym{cg}{
short = CG,
long = Continuous Galerkin,
tag = abbrev
}

\DeclareAcronym{CN}{
short = CN,
long = Crank-Nicolson,
tag = abbrev
}

\DeclareAcronym{comp}{
short = \texttt{cOMP},
long = collective orthogonal matching pursuit,
tag = abbrev
}

\DeclareAcronym{doed}{
short = D-OED,
long = D-optimal experimental design,
tag = abbrev
}

\DeclareAcronym{eim}{
short = EIM,
long = empirical interpolation method,
tag = abbrev
}

\DeclareAcronym{eoed}{
short = E-OED,
long = E-optimal experimental design,
tag = abbrev
}

\DeclareAcronym{fe}{
short = FE,
long = finite element,
tag = abbrev,
long-plural = s,
short-plural = s
}

\DeclareAcronym{fe4dvar}{
short = \myred{FE-4D-VAR},
long = \myred{FE-4D-VAR},
tag = method
}

\DeclareAcronym{geim}{
short = GEIM,
long = generalized empirical interpolation method,
tag = abbrev
}

\DeclareAcronym{greedyOMP}{
short = \myred{\texttt{greedyOMP}},
long = \myred{greedy orthogonal matching pursuit},
tag = abbrev
}

\DeclareAcronym{greedyPOD}{
short = \texttt{POD-greedy},
long = \myred{greedy-POD},
tag = method
}

\DeclareAcronym{iff}{
short = \textit{iff},
long = if and only if,
tag = abbrev
}

\DeclareAcronym{iid}{
short = i.i.d.,
long = independent identically distributed,
tag = abbrev
}

\DeclareAcronym{infsup}{
short = \textit{inf-sup},
long = \myred{\textit{inf-sup}},
tag = shortcut
}

\DeclareAcronym{kkt}{
short = KKT,
long = Karush-Kuhn-Tucker,
tag = abbrev
}

\DeclareAcronym{map}{
short = MAP,
long = maximum a posteriori probability,
tag = abbrev
}

\DeclareAcronym{mgs}{
short = MGS,
long = modified Gram Schmidt,
tag = abbrev
}

\DeclareAcronym{mor}{
short = MOR,
long = model order reduction,
tag = abbrev
}

\DeclareAcronym{moose}{
short = MOOSE,
long = \myred{MOOSE},
tag = method
}

\DeclareAcronym{oed}{
short = OED,
long = optimal experimental design,
tag = abbrev
}

\DeclareAcronym{omp}{
short = OMP,
long = orthogonal matching pursuit,
tag = abbrev
}

\DeclareAcronym{pbdw}{
short = PBDW,
long = parameterized-background data-weak,
tag = abbrev
}

\DeclareAcronym{pde}{
short = PDE,
long = partial differential equation,
tag = abbrev,
long-plural = s,
short-plural = s
}

\DeclareAcronym{pdf}{
short = pdf,
long = probability density function,
tag = abbrev
}

\DeclareAcronym{pg}{
short = PG,
long = Petrov Galerkin,
tag = abbrev
}

\DeclareAcronym{pod}{
short = POD,
long = proper orthogonal decomposition,
tag = abbrev
}

\DeclareAcronym{qoi}{
short = QoI,
long = quantity of interest,
tag = abbrev
}

\DeclareAcronym{rb}{
short = RB,
long = reduced basis,
tag = abbrev,
long-plural-form = reduced bases,
short-plural = ,
}

\DeclareAcronym{rb3dvar}{
short = RB-3D-VAR,
long = \myred{RB-3D-VAR},
tag = method
}

\DeclareAcronym{rb4dvar}{
short = RB-4D-VAR,
long = \myred{RB-4D-VAR},
tag = method
}

\DeclareAcronym{scm}{
short = SCM,
long = successive constraint method,
tag = abbrev
}

\DeclareAcronym{sGreedy}{
short = \texttt{sGreedy},
long = \myred{\texttt{sGreedy}},
tag = method
}

\DeclareAcronym{spd}{
short = s.p.d.,
long = symmetric positive definite,
tag = abbrev
}

\DeclareAcronym{supg}{
short = SUPG,
long = Streamline Upwind Petrov-Galerkin,
tag = abbrev
}

\DeclareAcronym{svd}{
short = SVD,
long = singular value decomposition,
tag = abbrev
}

\DeclareAcronym{tg}{
short = \rm{tg},
long = time-gradient,
tag = abbrev
}

\DeclareAcronym{uq}{
short = UQ,
long = uncertainty quantification,
tag = abbrev
}

\DeclareAcronym{wcomp}{
short = \texttt{\myorange{wOMP}},
long = worst-case orthogonal matching pursuit,
tag = abbrev
}

\DeclareAcronym{wlog}{
short = w.l.o.g.,
long = without loss of generality,
tag = abbrev
}

%% file: 10_introduction.tex
\section{Introduction}
\label{sec:introduction}

In the Bayesian approach to inverse problems (c.f. \cite{Stuart2010}), the uncertainty in a parameter is described via a probability distribution.
With Bayes' Theorem, the prior belief in a parameter is updated when new information is revealed such that the posterior distribution describes the parameter with improved certainty.
Bayes' posterior is optimal in the sense that it is the unique minimizer of the sum of the relative entropy between the posterior and the prior, and the mean squared error between the model prediction and the experimental data.
The noise model drives, along with the measurements, how the posterior's uncertainty is reduced in comparison to the prior.
A critical aspect -- especially for expensive experimental data\footnote{For instance, for projects harvesting geothermal energy, the development costs (e.g., drilling, stimulation, and tests) take up $50-70\%$ of the total budget (\cite{Stober2012}).
As each borehole can cost several million dollars, it is essential to plan their location carefully.} -- is how to select the measurements to improve the posterior's credibility best.
The selection of adequate sensors meeting individual applications' needs is, therefore, a big goal of the \ac{oed} research field and its surrounding community.
We refer to the literature (e.g., \cite{Ucinski2004,Melas2006,Pronzato2008}) for introductions.

The analysis and algorithm presented in this work significantly extend our initial ideas presented in \cite{Aretz2020} in which we seek to generalize the \acs{3dvar} stability results from \cite{Aretz-Nellesen2019a} to the probabilistic Bayesian setting.
Our proposed algorithm is directly related to the \ac{omp} algorithm \cite{Binev2018b, Maday2015e} for the \ac{pbdw} method and the \ac{eim} (\cite{Barrault2004a, Maday2013a}).
Closely related \ac{oed} methods for linear Bayesian inverse problems over \acp{pde} include \cite{Alexanderian2016, Alexanderian2014a, Attia2018, Alexanderian2018, Alexanderian2021, WuChenGhattas21}, mostly for A- and D-\ac{oed} and uncorrelated noise.
In recent years, these methods have also been extended to non-linear Bayesian inverse problems, e.g., \cite{Alexanderian2021a, Alexanderian2016a,Huan2013,WuChenGhattas20,WuOLearyRoseberryChenEtAl22}, while an advance to correlated noise has been made in \cite{Attia2020}.
In particular, \cite{WuChenGhattas20,WuOLearyRoseberryChenEtAl22} use similar algorithmic approaches to this work by applying a greedy algorithm to maximize the expected information gain.
Common strategies for dealing with the high dimensions imposed by the \ac{pde} model use the framework in \cite{Bui-ThanhGhattasMartinEtAl13} for discretization, combined with parameter reduction methods (e.g., \cite{Cui2016,Parente2020a, Lieberman2010, Chen2019a, ChenWuChenEtAl19b, ChenGhattas20, zahm2022certified}) and \ac{mor} methods for \ac{uq} problems (e.g., \cite{Qian2017, Chen2014a, Chen2017, OLearyRoseberryVillaChenEtAl22, OLearyRoseberryChenVillaEtAl22}).

In this paper, we consider inverse problem settings, in which a deterministic hyper-parameter describes anticipated system configurations such as material properties or loading conditions.
Each configuration changes the model non-linearly, so we obtain a \textit{family} of possible posterior distributions for any measurement data.
Supposing data can only be obtained with a single set of sensors regardless of the system's configuration, the \ac{oed} task becomes to reduce the uncertainty in each posterior uniformly over all hyper-parameters.
This task is challenging for high-dimensional models since 1) each configuration requires its own computationally expensive model solve, and 2) for large sets of admissible measurements, the comparison between sensors requires the inversion of the associated, possibly dense noise covariance matrix.
By building upon \cite{Aretz2020}, this paper addresses both challenges and proposes in detail a sensor selection algorithm that remains efficient even for correlated noise models.

The main contributions are as follows:
First, we identify an observability coefficient as a link between the sensor choice and the maximum eigenvalue of each posterior distribution. We also provide an analysis of its sensitivity to model approximations.
Second, we decompose the noise covariance matrix for any observation operator to allow fast computation of the observability gain under expansion with additional sensors.
Third, we propose a sensor selection algorithm that iteratively constructs an observation operator from a large set of sensors to increase the observability coefficient over all hyper-parameters.
The algorithm is applicable to correlated noise models, and requires, through the efficient use of \ac{mor} techniques, only a \textit{single} full-order model evaluation per selected sensor.

While the main idea and derivation of the observability coefficient are similar to \cite{Aretz2020}, this work additionally features 1) an analysis of the observability coefficient regarding model approximations, 2) explicit computational details for treating correlated noise models, and 3) a comprehensive discussion of the individual steps in the sensor selection algorithm. 
Moreover, the proposed method is tested using a large-scale geophysical model of the Perth Basin.

This paper is structured as follows:
In Section \ref{sec:Bayes:problem} we introduce the hyper-parameterized inverse problem setting, including all assumptions for the prior distribution, the noise model, and the forward model.
In Section \ref{sec:Bayes:observability}, we then establish and analyze the connection between the observability coefficient and the posterior uncertainty.
We finally propose our sensor selection algorithm in Section \ref{sec:Bayes:sensors} which exploits the presented analysis to choose sensors that improve the observability coefficient even in a hyper-parameterized setting.
We demonstrate the applicability and scalability of our approach on a high-dimensional geophysical model in Section \ref{sec:Bayes:results} before concluding in Section \ref{sec:Bayes:conclusion}.

%% file: 20_problem.tex
\section{Problem setting}
\label{sec:Bayes:problem}

Let $\myX$ be a Hilbert space with inner product $\myXinner{\cdot}{\cdot}$ and induced norm $\myXnorm{\myx}^2 := \myXinner{\myx}{\myx}$.
We consider the problem of identifying unknown states $\myxtrue(\mypara) \in \myX$ of a single physical system under changeable configurations $\mypara$ from 
noisy measurements 
\begin{align*}
\mydata (\mypara)
\approx \mytranspose{\myobsk[1]{\myxtrue(\mypara)}, \dots, \myobsk[\mydatadim]{\myxtrue(\mypara)}}
\in \mathbb{R}^{\mydatadim}.
\end{align*}
The measurements are obtained by a set of $\mydatadim$ unique \textit{sensors} (or \textit{experiments}) $\myobskmap[1], \dots, \myobskmap[\mydatadim]  \in \myXdual$.
Our goal is to choose these sensors from a large \textit{sensor library} $\mylibrary \subset \myXdual$ of options in a way that optimizes how much information is gained from their measurements for any configurations $\mypara$.

\subsubsection*{Hyper-parameterized forward model}

We consider the unknown state $\myxtrue$ to be uniquely characterized by two sources of information: 
\begin{itemize}
\item an unknown parameter $\myutrue \in \myU$ describing uncertainties in the governing physical laws, and
\item a hyper-parameter (or configuration\footnote{We call $\mypara$ interchangeably hyper-parameter or configuration to either stress its role in the mathematical model or physical interpretation.}) $\mypara \in \myparadomain \subset \myparaspace$ describing dependencies on controllable configurations under which the system may be observed (such as material properties or loading conditions) where $\myparadomain$ is a given compact set enclosing all possible configurations. 
\end{itemize}
For any given $\myu \in \myU$ and $\mypara \in \myparadomain$, we let $\mystate \in \myX$ be the solution of an abstract model equation $\mymodel(\mystate; \myu) = 0$ and assume that the map $\myu \rightarrow \mystate$ is well-defined, linear, and uniformly continuous in $\myu$, i.e.
\begin{align}\label{def:Bayes:myetasup}
&\exists ~ \bar{\eta} > 0: &\myetasup := \sup_{\myu \in \myU} \frac{\myXnorm{\mystate}}{\mypriornorm{\myu}} < \bar{\eta} && \forall ~\mypara \in \myparadomain.
\end{align}

\begin{remark}
Although we assumed that $\myutrue$ lies in the Euclidean space $\myU$, any other linear space can be considered via an affine transformation onto an appropriate basis (see \cite{Alexanderian2016, DaPrato2006}).
For infinite-dimensional spaces, we first discretize with appropriate treatment of the adjoint operator (c.f. \cite{Bui-ThanhGhattasMartinEtAl13}).
\end{remark}

\begin{remark}
By keeping the model equation general, we stress the applicability of our approach to a wide range of problems.
For instance, time-dependent states can be treated by choosing $\myX$ as a Bochner space or its discretization (c.f. \cite{Schwab2009a}).
We also do not formally restrict the dimension of $\myX$, though any implementation relies on the ability to compute $\mystate$ with sufficient accuracy.
To this end, we note that the analysis in Section \ref{sec:Bayes:parameterreduction} can be applied to determine how discretization errors affect the observability criterion in the sensor selection.
\end{remark}

Following a probabilistic approach to inverse problems, we express the initial uncertainty in $\myutrue = \myutrue(\mypara)$ of any $\myxtrue = \mystate[\myutrue]$ in configuration $\mypara$ through a random variable $\myu$ with Gaussian prior $\mypriormeasure= \mathcal{N}\left(\mypriormean, \mypriorCov\right)$, where $\mypriormean \in \myU$ is the prior mean and $\mypriorCov \in \mathbb{R}^{\myfedimU \times \myfedimU}$ is a \ac{spd} covariance matrix.
The latter defines the inner product $\mypriorinner{\cdot}{\cdot}$ and its induced norm $\mypriornorm{\cdot}$ through
\begin{align}
&\mypriorinner{\myu}{\myv} := \myu^T \mypriorCovInv \tilde{\myu},
&&\mypriornorm{\myu}^2 := \mypriorinner{\myu}{\myu},
&&\forall ~\myu, \myv \in \myU. \label{eq:Bayes:priornorm}
\end{align}
With these definitions, the \ac{pdf} for $\mypriormeasure$ is
\begin{align*}
\mypriorpdf (\myu) = \frac{1}{\sqrt{(2\pi)^{\myfedimU}\det \mypriorCov}} \exp\left(-\frac12 \mypriornorm{\myu - \mypriormean}^2 \right).
\end{align*}
For simplicity, we assume $\{\myutrue(\mypara)\}_{\mypara \in \myparadomain}$ to be independent realizations of $\myu$ such that we may consider the same prior for all $\mypara$ without accounting for a possible history of measurements at different configurations.

\subsubsection*{Sensor library and noise model}

For taking measurements of the unknown states $\{\myxtrue(\mypara)\}_{\mypara}$, we call any linear functional $\ell \in \myXdual$ a \textit{sensor}, and its application to a state $\myx \in \myX$ its \textit{measurement} $\ell(\myx) \in \mathbb{R}$.
We model experimental measurements $d_{\ell} \in \mathbb{R}$ of the actual physical state $\myxtrue$ as $d_{\ell} = \ell(\myxtrue) + \mynoise[\ell]$ where $\mynoise[\ell] \sim \mathcal{N}(0, \myCov(\mynoise[\ell], \mynoise[\ell]))$ is a Gaussian random variable.
We permit noise in different sensor measurements to be correlated with a known covariance function $\myCov$.
In a slight overload of notation, we write $\myCov : \mylibrary \times \mylibrary \rightarrow \mathbb{R}$, $\myCov(\ell_i, \ell_j) := \myCov(\mynoise[\ell_i], \mynoise[\ell_j])$ as a symmetric bilinear form over the sensor library.
Any ordered subset $\mysensorset = \{\myobskmap[1], \dots, \myobskmap[\mydatadim] \} \subset \mylibrary$ of sensors can then form a (linear and continuous) \textit{observation operator} through
\begin{align*}
\myobsmap &:= \mytranspose{\myobskmap[1], \dots, \myobskmap[\mydatadim]} : \myX \rightarrow \mathbb{R}^{\mydatadim},&
\myobs{\myx} := \mytranspose{\myobsk[1]{\myx}, \dots, \myobsk[\mydatadim]{\myx}}.
\end{align*}
The experimental measurements of $\myobsmap$ have the form
\begin{align}\label{eq:Bayes:noisemodel}
\mydata = \mytranspose{\myobsk[1]{\myxtrue} + \mynoise[{\myobskmap[1]}], \dots, \myobsk[\mydatadim]{\myxtrue} + \mynoise[{\myobskmap[{\mydatadim}]}]} = \myobs{\myxtrue} + \mynoise
\qquad \text{with} \qquad
\mynoise = \mytranspose{\mynoise[{\myobskmap[{1}]}], \dots, \mynoise[{\myobskmap[{\mydatadim}]}]} \sim \mathcal{N}(\myzeromatrix, \myreg \mynoiseCov),
\end{align}
where $\myreg \mynoiseCov$ is the noise covariance matrix defined through
\begin{align}\label{def:Bayes:mynoiseCov}
\mynoiseCov &\in \mathbb{R}^{\mydatadim \times \mydatadim}, 
\qquad \text{such that} \qquad
\myveceval{\myreg \mynoiseCov}{i,j} := \myCov(\myobskmap[j], \myobskmap[i]) = \myCov(\mynoise[{\myobskmap[j]}], \mynoise[{\myobskmap[i]}])
\end{align}
with an auxiliary scaling parameter\footnote{We introduce $\myreg$ here as an additional variable to ease the discussion of scaling in Section \ref{def:Bayes:paratoobsinfsup}. However, we can set $\myreg = 1$ \ac{wlog}.} $\myreg > 0$.
We assume that the library $\mylibrary$ and the noise covariance function $\myCov$ have been chosen such that $\mynoiseCov$ is \ac{spd} for any combination of sensors in $\mylibrary$.
This assumption gives rise to the $\myobsmap$-dependent inner product and its induced norm
\begin{align}
&\mynoiseinner{\myd}{\tilde{\myd}} := \myd^T \mynoiseCovInv \tilde{\myd},
&&\mynoisenorm{\myd}^2 := \mynoiseinner{\myd}{\myd},
&&\forall ~\myd, \tilde{\myd} \in \mathbb{R}^{\mydatadim}.\label{eq:Bayes:datanorm}
\end{align}
Measured with respect to this norm, the largest observation of any (normalized) state is thus
\begin{align}\label{def:Bayes:obscont}
\myobscont &:= \sup_{\myXnorm{\myx} = 1} \mynoisenorm{\myobs{\myx}} = \sup_{\myx \in \myX}  \frac{\mynoisenorm{\myobs{\myx}}}{\myXnorm{\myx}}.
\end{align}
We show in Section \ref{sec:Cholesky} that $\myobscont$ increases under expansion of $\myobsmap$ with additional sensors despite the change in norm, and is therefore bounded by $\myobscont \le \myobscont[\mylibrary]$.

We also define the \textit{parameter-to-observable map}
\begin{align}\label{def:Bayes:myparatoobsmap}
\myparatoobsmap : \myU \rightarrow \mathbb{R}^{\mydatadim},
\quad \text{such that} \quad
\myparatoobs{\myu} := \myobs{\mystate[\myu]}.
\end{align}
With the assumptions above -- in particular the linearity and uniform continuity \eqref{def:Bayes:myetasup} of $\myx$ in $\myu$ -- the map $\myparatoobsmap$ is linear and uniformly bounded in $\myu$.
We let $\myparatoobsmatrix \in \mathbb{R}^{\mydatadim \times \myfedimU}$ denote its matrix representation with respect to the unit basis $\{\myek[m]\}_{m=1}^{\myfedimU}$.
The likelihood of $\myd \in \mathbb{R}^{\mydatadim}$ obtained through the observation operator $\myobsmap$ for the parameter $\myu \in \myU$ and the system configuration $\mypara$ is then
\begin{align*}
\mylikelyhood\left(\mydata ~\big|~ \myu, \mypara\right) := \frac{1}{\sqrt{2^{\mydatadim}\det \mynoiseCov}} \exp\left(-\dfrac{1}{2\myreg} \mynoisenorm{\myd - \myparatoobs{\myu}}^2 \right).
\end{align*}
Note that $\myparatoobsmap$ and $\myparatoobsmatrix$ may depend non-linearly on $\mypara$.

\subsubsection*{Posterior distribution}

Once noisy measurement data $\mydata \approx \myobs{\myxtrue(\mypara)}$ is available, Bayes' theorem yields the posterior \ac{pdf} as
\begin{equation}\label{eq:posterior}
\myposterior (\myu ~|~ \mydata) = \frac{1}{Z(\mypara)} \exp \left( - \dfrac{1}{2\myreg} \mynoisenorm{\myparatoobs{\myu} - \mydata}^2 - \dfrac{1}{2} \mypriornorm{\myu - \mypriormean}^2 \right)
\propto \mypriorpdf(\myu) \cdot \mylikelyhood\left(\mydata ~\big|~ \myu, \mypara\right),
\end{equation}
with normalization constant
\begin{equation*}
Z(\mypara) := \int_{\myparaspace} \exp \left(-\frac{1}{2\myreg} \mynoisenorm{\myparatoobs{\myu} - \mydata}^2\right)  ~d\mypriormeasure.
\end{equation*}
Due to the linearity of the parameter-to-observable map, the posterior measure $\myposteriormeasure$ is a Gaussian 
\begin{align*}
\myposteriormeasure = \mathcal{N}(\myposteriormean(\mydata), \mypostCov)
\end{align*}
with known (c.f. \cite{Stuart2010}) mean and covariance matrix
\begin{align}
&\myposteriormean(\mydata) = \mypostCov \left(\textstyle \frac{1}{\myreg} \myparatoobsmatrix^T \mynoiseCovinvobs \mydata + \mypriorCovInv \mypriormean \right)
&&  \in \mathbb{R}^{\myfedimU}, \label{eq:posteriormean}\\
&\mypostCov = \left( \textstyle \frac{1}{\myreg}\myparatoobsmatrix^T \mynoiseCovinvobs \myparatoobsmatrix  + \mypriorCovInv \right)^{-1}
&&  \in \mathbb{R}^{\myfedimU \times \myfedimU}. \label{eq:posteriorCov}
\end{align}
The posterior $\myposteriormeasure$ thus depends not only on the choice of sensors, but also on the configuration $\mypara$ under which their measurements were obtained.
Therefore, to decrease the uncertainty in all possible posteriors with a single, $\mypara$-independent observation operator $\myobsmap$, the construction of $\myobsmap$ should account for all admissible configurations $\mypara \in \myparadomain$ under which $\myxtrue$ may be observed.

\begin{remark}
The linearity of $\mystate$ in $\myu$ is a strong assumption that dictates the Gaussian posterior.
However, in combination with the hyper-parameter $\mypara$, our setting here can be re-interpreted as the Laplace-approximation for a non-linear state map $\mypara \mapsto \myx(\mypara)$ (c.f. \cite{Long2013, WuChenGhattas20, Aretz2022}).
The sensor selection presented here is then an intermediary step for \ac{oed} over non-linear forward models.
\end{remark}

%% file: 620_observability.tex
\section{The Observability Coefficient}
\label{sec:Bayes:observability}

In this section, we characterize how the choice of sensors in the observation operator $\myobsmap$ and its associated noise covariance matrix $\mynoiseCov$ influence the uncertainty in the posteriors $\myposteriormeasure$, $\mypara \in \myparadomain$.
We identify an observability coefficient that bounds the eigenvalues of the posterior covariance matrices $\mypostCov$, $\mypara \in \myparadomain$ with respect to $\myobsmap$, and facilitates the sensor selection algorithm presented in Section \ref{sec:Bayes:sensors}.

%% file: 623_eigenvalues.tex
\subsection{Eigenvalues of the Posterior Covariance Matrix}
\label{sec:Bayes:eigenvalues}

The uncertainty in the posterior $\myposterior$ for any configuration $\mypara \in \myparadomain$ is uniquely characterized by the posterior covariance matrix $\mypostCov$, which is in turn connected to the observation operator $\myobsmap$ through the parameter-to-observable map $\myparatoobsmap$ and the noise covariance matrix $\mynoiseCov$.
To measure the uncertainty in $\mypostCov$, the \ac{oed} literature suggests a variety of different utility functions to be minimized over $\myobsmap$ in order to optimize the sensor choice.
Many of these utility functions can be expressed in terms of the eigenvalues $\mypostCovEig[1] \ge \dots \ge \mypostCovEig[\myfedimU] > 0$ of $\mypostCov$, e.g.,
\begin{align*}
&\text{A-OED:}&&
\trace(\mypostCov) = \sum_{m=1}^{\myfedimU} \mypostCovEig[m]
&&\text{(mean variance)} \\
&\text{D-OED:}&&
\det(\mypostCov) = \prod _{m=1}^{\myfedimU} \mypostCovEig[m]
&&\text{(volume)} \\
&\text{E-OED:}&&
\lambda_{\max}(\mypostCov) = \mypostCovEig[1]
&&\text{(spectral radius).}
\end{align*}
In practice, the choice of the utility function is dictated by the application.
In \ac{eoed}, for instance, posteriors whose uncertainty ellipsoids stretch out into any one direction are avoided, whereas \acs{doed} minimizes the overall volume of the uncertainty ellipsoid regardless of the uncertainty in any one parameter direction.
We refer to \cite{Ucinski2004} for a detailed introduction and other \ac{oed} criteria.

Considering the hyper-parameterized setting where each configuration $\mypara$ influences the posterior uncertainty, we seek to choose a single observation operator $\myobsmap$ such that the selected utility function remains small for \textit{all} configurations $\mypara \in \myparadomain$, e.g., for \acs{eoed}, minimizing
\begin{align*}
    \min_{\myobskmap[1], \dots, \myobskmap[\mydatadim] \in \mylibrary} ~ \max_{\mypara \in \myparadomain} ~ \lambda_{\max}(\mypostCov)
    \quad \text{such that} \quad \myobsmap = \mytranspose{\myobskmap[1], \dots, \myobskmap[\mydatadim]}
\end{align*}
guarantees that the longest axis of each posterior covariance matrix $\mypostCov$ for any $\mypara \in \myparadomain$ has the same guaranteed upper bound. 
The difficulty here is that the minimization over $\myparadomain$ necessitates repeated, cost-intensive model evaluations to compute the utility function for many different configurations $\mypara$.
In the following, we therefore introduce an upper bound to the posterior eigenvalues that can be optimized through an observability criterion with far fewer model solves.
The bound's optimization indirectly reduces the different utility functions through the posterior eigenvalues.

Recalling that $\mypostCov$ is \ac{spd}, let $\{\myum[m]\}_{m=1}^{\myfedimU}$ be an orthonormal eigenvector basis of $\mypostCov$, i.e. $\myum^T \myum[n] = \delta_{m,n}$ and
\begin{align}\label{eq:Bayes:eigenvalueproblem}
\mypostCov \myum &= \mypostCovEig \myum & m=1, \dots, \myfedimU.
\end{align}
Using the representation \eqref{eq:posteriorCov}, any eigenvalue $\mypostCovEig$ can be written in the form
\begin{equation}\label{eq:Bayes:eigenvalue:equality}
\begin{aligned}
\frac{1}{\mypostCovEig} 
= \myum^T \mypostCovInv \myum 
= \myum^T \mymatrixsum{\textstyle \frac{1}{\myreg}\myparatoobsmatrix^T \mynoiseCovinvobs \myparatoobsmatrix + \mypriorCovInv } \myum 
= \frac{1}{\myreg} \mynoisenorm{\myparatoobs{\myum}}^2 + \mypriornorm{\myum}^2.
\end{aligned}
\end{equation}
Since $\myum$ depends implicitly on $\myobsmap$ and $\mypara$ through $\eqref{eq:Bayes:eigenvalueproblem}$, we cannot use this representation directly to optimize over $\myobsmap$.
To take out the dependency on $\myum$, we bound $\mypriornorm{\myum}^2 \ge \frac{1}{\mypriorCovEig[\max]}$ in terms of the maximum eigenvalue of the prior covariance matrix $\mypriorCov$.
Likewise, we define
\begin{align}\label{def:Bayes:paratoobsinfsup}
\myparatoobsinfsup 
&:= \inf_{\myu \in \myU} \frac{\mynoisenorm{\myparatoobs{\myu}}}{\mypriornorm{\myu}}
=\inf_{\myu \in \myU} \frac{\mynoisenorm{\myobs{\mystate}}}{\mypriornorm{\myu}}, 
\end{align}
as the minimum ratio between an observation for a parameter $\myu$ relative to the prior's covariance norm.
From \eqref{eq:Bayes:eigenvalue:equality} and \eqref{def:Bayes:paratoobsinfsup} we obtain the upper bound
\begin{align*}
\mypostCovEig 
=\left(\frac{1}{\myreg} \frac{\mynoisenorm{\myparatoobs{\myum}}^2}{\mypriornorm{\myum}^2} + 1 \right)^{-1} \mypriornorm{\myum}^{-2}
\le \left(\frac{1}{\myreg} \myparatoobsinfsup^2 + 1\right)^{-1} \mypriorCovEig[\max].
\end{align*}
Geometrically, this bound means that the radius $\mypostCovEig[1]$ of the outer ball around the posterior uncertainty ellipsoid is smaller than that of the prior uncertainty ellipsoid by at least the factor $\left(\frac{1}{\myreg} \myparatoobsinfsup^2 + 1\right)^{-1}$.
By choosing $\myobsmap$ to maximize $\min_{\mypara} \myparatoobsinfsup$, we therefore minimize this outer ball containing all uncertainty ellipsoids (i.e., for any $\mypara \in \myparadomain$).
As expected, the influence of $\myobsmap$ is strongest when the measurement noise is small such that data can be trusted ($\myreg \ll 1$), and diminishes with increasing noise levels ($\myreg \gg 1$).

%% file: 625_parameterreduction.tex
\subsection{Parameter Restriction}\label{sec:Bayes:parameterreduction}

An essential property of $\myparatoobsinfsup$ is that $\myparatoobsinfsup = 0$ if $\mydatadim < \myfedimU$, i.e., the number of sensors in $\myobsmap$ is smaller than the number of parameter dimensions.
In this case, $\myparatoobsinfsup$ cannot distinguish between sensors during the first $\myfedimU-1$ steps of an iterative algorithm, or in general when less than a total of $\myfedimU$ sensors are supposed to be chosen.
For medium-dimensional parameter spaces ($\myfedimU \in \mycurlyO{10}$), we mitigate this issue by restricting $\myu$ to the subspace $\myspan\{\myUfebasis[1], \dots, \myUfebasis[{\min\{\mydatadim, \myfedimU\}}]\} \subset \myU$ spanned by the first $\min\{\mydatadim, \myfedimU\}$ eigenvectors of $\mypriorCov$ corresponding to its largest eigenvalues, i.e., the subspace with the largest prior uncertainty.
For high-dimensional parameter spaces or when the model $\mymodel$ has a non-trivial null-space, we bound $\myparatoobsinfsup$ further
\begin{equation}\label{eq:Bayes:paratoobsinfsup:lower}
\begin{aligned}
\myparatoobsinfsup 
&= \inf_{\myu \in \myU} \frac{\mynoisenorm{\myobs{\mystate}}}{\myXnorm{\mystate}} \frac{\myXnorm{\mystate}}{\mypriornorm{\myu}} 
\ge \inf_{x \in \myW} \frac{\mynoisenorm{\myobs{x}}}{\myXnorm{x}} \inf_{\myu \in \myU} \frac{\myXnorm{\mystate}}{\mypriornorm{\myu}}
= \myobsinfsup ~ \myetainf
\end{aligned}
\end{equation}
where we define the linear space $\myW$ of all achievable states
\begin{align*}
\myW := \{ \mystate \in \myX: ~\myu \in \myU \}
\end{align*}
and the coefficients
\begin{align}\label{def:Bayes:obsinfsup}
\myobsinfsup &:= \inf_{\mygenstate \in \myW} \frac{\mynoisenorm{\myobs{\mygenstate}}}{\myXnorm{\mygenstate}},
&\myetainf &:= \inf_{\myu \in \myU} \frac{\myXnorm{\mystate}}{\mypriornorm{\myu}}.
\end{align}
The value of $\myetainf$ describes the minimal state change that a parameter $\myu$ can achieve relative to its prior-induced norm $\mypriornorm{\myu}$.
It can filter out parameter directions that have little influence on the states $\mystate$.
In contrast, the observability coefficient $\myobsinfsup$ depends on the prior only implicitly via $\myW$; it quantifies the minimum amount of information (measured with respect to the noise model) that can be obtained on any state in $\myW$ relative to its norm.
Future work will investigate how to optimally restrict the parameter space based on $\myetainf$ before choosing sensors that maximize $\myobsinfsup$.
Existing parameter reduction approaches in a similar context include \cite{Chen2019a,Cui2015,Bui-Thanh2008,Lieberman2010}.
In this work, however, we solely focus on the maximization of $\myparatoobsinfsup$ and, by extension, $\myobsinfsup$ and henceforth assume that $\myfedimU$ is sufficiently small and $\myetainfBND := \inf_{\mypara \in \myparadomain}\myetainf > 0$ is bounded away from zero.

%% file: 624_surrogate.tex
\subsection{Observability under model approximations}
\label{sec:Bayes:surrogate}

To optimize the observability coefficient $\myparatoobsinfsup$ or $\myobsinfsup$, it must be computed for many different configurations $\mypara \in \myparadomain$.
The accumulating computational cost motivates the use of \textit{reduced-order} surrogate models, which typically yield considerable computational savings versus the original \textit{full-order} model.
However, this leads to errors in the state approximation.
In the following, we thus quantify the influence of state approximation error on the observability coefficients $\myparatoobsinfsup$ and $\myobsinfsup$.
An analysis of the change in posterior distributions when the entire model $\mymodel$ is substituted in the inverse problem can be found in \cite{Stuart2010}.


Suppose a reduced-order surrogate model $\myRBmodel(\myRBstate; \myu) = 0$ is available that yields for any configuration $\mypara \in \myparadomain$ and parameter $\myu \in \myU$ a unique solution $\myRBstate \in \myX$ such that 
\begin{equation}\label{eq:Bayes:RBapprox}
\myXnorm{\mystate-\myRBstate} \le \myhypereps \myXnorm{\mystate}
\quad \text{with accuracy} \quad
0 \le \myhypereps \le \varepsilon < 1.
\end{equation}
Analogously to \eqref{def:Bayes:paratoobsinfsup} and \eqref{def:Bayes:obsinfsup}, we define 
the reduced-order observability coefficients
\begin{align}\label{def:Bayes:myparatoobsrbinfsup}
\myparatoobsrbinfsup &:= \inf_{\myu \in \myU} \frac{\mynoisenorm{\myobs{\myRBstate}}}{\mypriornorm{\myu}}, &
\myobsrbinfsup &:= \inf_{\myu \in \myU} \frac{\mynoisenorm{\myobs{\myRBstate}}}{\myXnorm{\myRBstate}}
\end{align}
to quantify the smallest observations of the surrogate states.
For many applications, it is possible to choose a reduced-order model whose solution can be computed at a significantly reduced cost such that $\myparatoobsrbinfsup$ and $\myobsrbinfsup$ are much cheaper to compute than their full-order counterparts $\myparatoobsinfsup$ and $\myobsinfsup$.
Since the construction of such a surrogate model depends strongly on the application itself, we refer to the literature (e.g., \cite{Benner2015,Schilders2008,Hesthaven2016,Quarteroni2015, Haasdonk2017}) for tangible approaches.

Recalling the definition of $\myobscont$ in \eqref{def:Bayes:obscont}, we start by bounding how closely the surrogate observability coefficient $\myobsrbinfsup$ approximates the full-order $\myobsinfsup$.

\begin{proposition}\label{thm:Bayes:obsinfsup:approx}
Let $\myetainf > 0$ hold, and let $\myRBstate \in \myX$ be an approximation to $\mystate$ such that \eqref{eq:Bayes:RBapprox} holds for all $\mypara \in \myparadomain$, $\myu \in \myU$.
Then
\begin{align}\label{eq:Bayes:myobsinfsup:RB}
(1-\myhypereps)~ \myobsrbinfsup - \myobscont \myhypereps
~\le~ \myobsinfsup
~\le~ (1+\myhypereps)~ \myobsrbinfsup + \myobscont \myhypereps.
\end{align}
\end{proposition}

\begin{proof}
Let $\myu \in \myU \setminus \{\mathbf{0}\}$ be arbitrary.
Using \eqref{eq:Bayes:RBapprox} and the (reversed) triangle inequality, we obtain the bound
\begin{align}\label{eq:proof:1}
\frac{\myXnorm{\myRBstate}}{\myXnorm{\mystate}} 
\ge \frac{\myXnorm{\mystate} - \myXnorm{\mystate - \myRBstate}}{\myXnorm{\mystate}} 
\ge 1 - \myhypereps.
\end{align}
Note here that $\myetainf > 0$ implies $\myXnorm{\mystate} > 0$ so the quotient is indeed well defined.
The ratio of observation to state can now be bounded from below by
\begin{align*}
\frac{\mynoisenorm{\myobs \mystate}}{\myXnorm{\mystate}}
&\ge \frac{\mynoisenorm{\myobs \myRBstate}}{\myXnorm{\mystate}} - \frac{\mynoisenorm{\myobs{(\mystate - \myRBstate)}}}{\myXnorm{\mystate}} \\
&\ge  \frac{\myXnorm{\myRBstate }}{\myXnorm{\mystate}} \frac{\mynoisenorm{\myobs \myRBstate}}{\myXnorm{\myRBstate}} - \myobscont \frac{\myXnorm{\mystate-\myRBstate}}{\myXnorm{\mystate}}  \\
&\ge  (1-\myhypereps) \frac{\mynoisenorm{\myobs \myRBstate}}{\myXnorm{\myRBstate}} - \myobscont \myhypereps \\
&\ge (1-\myhypereps) \myobsrbinfsup - \myobscont \myhypereps,
\end{align*}
where we have applied the reverse triangle inequality, definition \eqref{def:Bayes:obscont}, the bounds \eqref{eq:Bayes:RBapprox}, \eqref{eq:proof:1}, and definition \eqref{def:Bayes:myparatoobsrbinfsup} of $\myobsrbinfsup$.
Since $\myu$ is arbitrary, the lower bound in \eqref{eq:Bayes:myobsinfsup:RB} follows from definition \eqref{def:Bayes:paratoobsinfsup} of $\myobsinfsup$.
The upper bound in \eqref{eq:Bayes:myobsinfsup:RB} follows analogously.
\end{proof}

For the observability of the parameter-to-observable map $\myparatoobsmap$ and its approximation $\myu \mapsto \myobs{\myRBstate}$, we obtain a similar bound.
It uses the norm $\myetasup$ of $\myx_{\mypara} : \myu \mapsto \mystate$ as a map from the parameter to the state space, see \eqref{def:Bayes:myetasup}.

\begin{proposition}\label{thm:Bayes:paratoobsinfsup:approx}
Let $\myRBstate \in \myX$ be an approximation to $\mystate$ such that \eqref{eq:Bayes:RBapprox} holds for all $\mypara \in \myparadomain$, $\myu \in \myU$.
Then
\begin{align}\label{eq:Bayes:myparatoobsinfsup:RB}
\myparatoobsrbinfsup - \myobscont \myetasup \myhypereps
\le \myparatoobsinfsup 
\le \myparatoobsrbinfsup + \myobscont \myetasup \myhypereps.
\end{align}
\end{proposition}

\begin{proof}
Let $\myu \in \myU \setminus \{\mathbf{0}\}$ be arbitrary.
Then
\begin{align*}
\mynoisenorm{\myobs{\mystate}}
&\ge \mynoisenorm{\myobs{\myRBstate}} - \mynoisenorm{\myobs{(\mystate-\myRBstate)}} \\
&\ge \mynoisenorm{\myobs{\myRBstate}}  - \myobscont \myXnorm{\mystate - \myRBstate} \\
&\ge \mynoisenorm{\myobs{\myRBstate}} - \myobscont \myhypereps \myXnorm{\mystate} \\
&\ge \mynoisenorm{\myobs{\myRBstate}} - \myobscont \myhypereps \myetasup \mypriornorm{\myu}, 
\end{align*}
where we have used the reverse triangle inequality, followed by \eqref{def:Bayes:obscont}, \eqref{eq:Bayes:RBapprox}, and \eqref{def:Bayes:myetasup}.
We divide by $\mypriornorm{\myu}$ and take the infimum over $\myu$ to obtain
\begin{align*}
\myparatoobsinfsup 
= \inf_{\myu \in \myU} \frac{\mynoisenorm{\myobs{\mystate}}}{\mypriornorm{\myu}} 
\ge \inf_{\myu \in \myU} \frac{\mynoisenorm{\myobs{\myRBstate}} }{\mypriornorm{\myu}} - \myobscont ~ \myetasup ~ \myhypereps 
= \myparatoobsrbinfsup - \myobscont ~ \myetasup ~ \myhypereps.
\end{align*}
The upper bound in \eqref{eq:Bayes:myparatoobsinfsup:RB} follows analogously.
\end{proof}

If $\myhypereps$ is sufficiently small, Propositions \ref{thm:Bayes:obsinfsup:approx} and \ref{thm:Bayes:paratoobsinfsup:approx} justify employing the surrogates $\myobsrbinfsup$ and $\myparatoobsrbinfsup$ instead of the original full-order observability coefficients $\myobsinfsup$ and $\myparatoobsinfsup$.
This substitution becomes especially necessary when the computation of $\mystate$ is too expensive to evaluate $\myobsinfsup$ or $\myparatoobsinfsup$ repeatedly for different configurations $\mypara$.

Another approximation step in our sensor selection algorithm relies on the identification of a parameter direction $\myv \in \myU$ with comparatively small observability, i.e.
\begin{align*}
    \frac{\mynoisenorm{\myobs{\mystate[\myv]}}}{\mypriornorm{\myv}} \approx \inf_{\myu \in \myU} \frac{\mynoisenorm{\myobs{\mystate}}}{\mypriornorm{\myu}} = \myparatoobsinfsup 
    \qquad \text{or} \qquad
    \frac{\mynoisenorm{\myobs{\mystate[\myv]}}}{\myXnorm{\mystate[\myv]}} \approx  \inf_{\mygenstate \in \myW} \frac{\mynoisenorm{\myobs{\mygenstate}}}{\myXnorm{\mygenstate}} = \myobsinfsup.
\end{align*}
The ideal choice would be the infimizer of respectively $\myparatoobsinfsup$ or $\myobsinfsup$, but its computation involves $\myfedimU$ full-order model evaluations (c.f. Section \ref{sec:obsComp}).
To avoid these costly computations, we instead choose $\myv$ as the infimizer of the respective reduced-order observability coefficient.
This choice is justified for small $\myhypereps < 1$ by the following proposition:

\begin{proposition}
Let $\myetainf > 0$ hold, and let $\myRBstate \in \myX$ be an approximation to $\mystate$ such that \eqref{eq:Bayes:RBapprox} holds for all $\mypara \in \myparadomain$, $\myu \in \myU$.
Suppose 
$\myv \in \arg \inf_{\myu \in \myU} \mypriornorm{\myu}^{-1} \mynoisenorm{\myobs{\myRBstate}}$, 
then
\begin{align}\label{eq:Bayes:argsubstituion:G}
    \myparatoobsinfsup 
    \le \frac{\mynoisenorm{\myobs{\mystate[\myv]}}}{\mypriornorm{\myv}}
    \le \myparatoobsinfsup + 2\myobscont \myetasup \myhypereps.
\end{align}
Likewise, if 
$\myv \in \arg \inf_{\myu \in \myU} \myXnorm{\myRBstate}^{-1} \mynoisenorm{\myobs{\myRBstate}}$, 
then
\begin{align}\label{eq:Bayes:argsubstituion:W}
    \myobsinfsup
    \le \frac{\mynoisenorm{\myobs{\mystate[\myv]}}}{\myXnorm{\mystate[\myv]}}
    \le \frac{1+\myhypereps}{1-\myhypereps} ~\left(\myobsinfsup + \myobscont \myhypereps \right) + \myobscont \myhypereps.
\end{align}
\end{proposition}

\begin{proof}
For both \eqref{eq:Bayes:argsubstituion:G} and \eqref{eq:Bayes:argsubstituion:W} the lower bound follows directly from definitions \eqref{def:Bayes:paratoobsinfsup} and \eqref{def:Bayes:obsinfsup}.
To prove the upper bound in \eqref{eq:Bayes:argsubstituion:G}, let $\myv \in \arg \inf_{\myu \in \myU} \mypriornorm{\myu}^{-1}\mynoisenorm{\myobs{\myRBstate}}$.
Following the same steps as in the proof of Proposition \ref{thm:Bayes:paratoobsinfsup:approx}, we can then bound
\begin{align*}
    \frac{\mynoisenorm{\myobs{\mystate[\myv]}}}{\mypriornorm{\myv}}
    \le \frac{\mynoisenorm{\myobs{\myRBstate[\myv]}}}{\mypriornorm{\myv}} + \frac{\mynoisenorm{\myobs{(\mystate[\myv]-\myRBstate[\myv])}}}{\mypriornorm{\myv}}
    \le \myparatoobsrbinfsup + \myobscont \myetasup \myhypereps.
\end{align*}
The upper bound in \eqref{eq:Bayes:argsubstituion:G} then follows with Proposition \ref{thm:Bayes:paratoobsinfsup:approx}.

To prove the upper bound in \eqref{eq:Bayes:argsubstituion:W}, let $\myv \in \arg \inf_{\myu \in \myU} \myXnorm{\myRBstate}^{-1}\mynoisenorm{\myobs{\myRBstate}}$.
Then 
\begin{align*}
    \frac{\mynoisenorm{\myobs{\mystate[\myv]}}}{\myXnorm{\mystate[\myv]}}
    \le \frac{\mynoisenorm{\myobs{\myRBstate[\myv]}}}{\myXnorm{\myRBstate[\myv]}} \frac{\myXnorm{\myRBstate[\myv]}}{\myXnorm{\mystate[\myv]}} + \frac{\mynoisenorm{\myobs{(\mystate[\myv]-\myRBstate[\myv])}}}{\myXnorm{\mystate[\myv]}}
    \le (1+\varepsilon) ~ \myobsrbinfsup + \myobscont \myhypereps.
\end{align*}
The result then follows with Proposition \ref{thm:Bayes:obsinfsup:approx}.
\end{proof}

%% file: 630_sensors.tex
\section{Sensor selection}
\label{sec:Bayes:sensors}

In the following, we present a sensor selection algorithm that iteratively increases the minimal observability coefficient $\min_{\mypara \in \myparadomain} \myparatoobsinfsup$ and thereby decreases the upper bound for the eigenvalues of the posterior covariance matrix for all admissible system configurations $\mypara \in \myparadomain$.
The iterative approach is relatively easy to implement, allows a simple way of dealing with combinatorial restrictions, and can deal with large\footnote{
For instance, in Section \ref{sec:Bayes:results:unrestricted} we apply the presented algorithm to a library with $\mylibrarydim = 11,045$ available sensor positions.
} sensor libraries.


%% file: 631_cholesky.tex
\subsection{Cholesky decomposition}
\label{sec:Cholesky}

The covariance function $\myCov$ connects an observation operator $\myobsmap$ to its observability coefficients $\myparatoobsinfsup$ and $\myobsinfsup$ through the noise covariance matrix $\mynoiseCov$.
Its inverse enters the norm $\mynoisenorm{\cdot}$ and the posterior covariance matrix $\mypostCov$.
The inversion poses a challenge when the noise is correlated, i.e., when $\mynoiseCov$ is not diagonal, as even the expansion of $\myobsmap$ with a single sensor $\myobskmap[] \in \mylibrary$ changes each entry of $\mynoiseCovInv$.
In naive computations of the observability coefficients and the posterior covariance matrix, this leads to $\myfedimU$ dense linear system solves of order $\mycurlyO{(\mydatadim+1)^3}$ each time the observation operator is expanded.
In the following, we therefore expound on how $\mynoiseCovInv$ changes under expansion of $\myobsmap$ to exploit its structure when comparing potential sensor choices.

Suppose $\myobsmap = \mytranspose{\myobskmap[1], \dots, \myobskmap[\mydatadim]}$ has already been chosen with sensors $\myobskmap \in \myXdual$, but shall be expanded by another sensor $\ell$ to
\begin{align*}
\myextendedobsmap &:= \mytranspose{\myobskmap[1], \dots, \myobskmap[\mydatadim], \myobskmap[]} : \myX \rightarrow \mathbb{R}^{\mydatadim+1}.
\end{align*}
Following definition \eqref{def:Bayes:mynoiseCov}, the noise covariance matrix $\mynoiseCov[\myextendedobsmap]$ of the expanded operator $\myextendedobsmap$ has the form
\begin{align*}
\mynoiseCov[\myextendedobsmap] =
\left(
\begin{array}{cc}
\mynoiseCovobs & \myCholvector \\
\myCholvector^T & \myCholvar
\end{array}
\right) = 
\left(
\begin{array}{cc}
\mynoiseChol & \mathbf{0} \\
\myCholvectorw^T & \myCholbeta
\end{array}
\right)
\left(
\begin{array}{cc}
\mynoiseCholadj & \myCholvectorw \\
\mathbf{0} & \myCholbeta
\end{array}
\right),
\end{align*}
where $\mynoiseChol \mynoiseCholadj = \mynoiseCov \in \mathbb{R}^{\mydatadim \times \mydatadim}$ is the Cholesky decomposition of the \ac{spd}\ noise covariance matrix $\mynoiseCov$ for the original observation operator $\myobsmap$, and $\myCholvector, \myCholvectorw \in \mathbb{R}^{\mydatadim}$, $\myCholvar, \myCholbeta \in \mathbb{R}$ are defined through
\begin{align*}
\myveceval{\myCholvector}{i} &:= \myCov(\myobskmap[i], \myobskmap[]), &
\myCholvectorw &:= \mynoiseCholinv \myCholvector, \\
\myCholvar &:= \myCov(\myobskmap[], \myobskmap[]), &
\myCholbeta &:= \sqrt{\myCholvar - \myCholvectorw^T \myCholvectorw}.
\end{align*}
Note that $\mynoiseCov[\myextendedobsmap]$ is \ac{spd}\ by the assumptions posed on $\myCov$ in Section \ref{sec:Bayes:problem}; consequently, $\myCholbeta$ is well-defined and strictly positive.
With this factorization, the expanded Cholesky matrix $\mynoiseChol[\myextendedobsmap]$ with $\mynoiseChol[\myextendedobsmap] \mynoiseChol[\myextendedobsmap]^T = \mynoiseCov[\myextendedobsmap]$ can be computed in $\mycurlyO{\mydatadim^2}$, dominated by the linear system solve with the triangular $\mynoiseChol$ for obtaining $\myCholvectorw$.
It is summarized in Algorithm \ref{alg:Bayes:Cholesky} for later use in the sensor selection algorithm.

\begin{algorithm}[t]
\DontPrintSemicolon
\KwIn {observation operator $\myobsmap = \mytranspose{\myobskmap[1], \dots, \myobskmap[\mydatadim]}$, noise covariance matrix $\mynoiseCov$, Cholesky matrix $\mynoiseChol $, new sensor $\ell  \in \myXdual$}
\vspace{0.2cm}
$\myobsmap \gets \mytranspose{\myobskmap[1], \dots, \myobskmap[\mydatadim], \ell}$ \tcp*[r]{operator expansion}
\If{$\myiter = 0$}{
$\mynoiseCov \gets \left(\myCov(\myobskmap[], \myobskmap[])\right), ~ \mynoiseChol \gets  \left(\sqrt{\myCov(\myobskmap[], \myobskmap[])}\right) \in \mathbb{R}^{1 \times 1}$ \tcp*[r]{first sensor}
}\Else{
$\mathbf{v} \gets \mytranspose{\myCov(\myobskmap[1], \myobskmap[]), \dots, \myCov(\myobskmap[\myiter], \myobskmap[])} \in \mathbb{R}^{\myiter}$ \tcp*[r]{matrix expansion}
$\mathbf{w} \gets \mynoiseChol^{-1}\mathbf{v} \in \mathbb{R}^{\myiter}$, $s \gets \myCov(\myobskmap[], \myobskmap[])$, $c \gets s - \mathbf{w}^T \mathbf{w} \in \mathbb{R}$ \;
$\mynoiseCov \gets \left(
\begin{array}{cc}
\mynoiseCov & \mathbf{v} \\
\mathbf{v}^T & s
\end{array}
\right) $, 
$\mynoiseChol \gets \left(
\begin{array}{cc}
\mynoiseChol & \myzeromatrix \\
\mathbf{w}^{T}& c
\end{array}
\right) \in \mathbb{R}^{(\myiter + 1) \times (\myiter + 1)}$ \;
}
\vspace{0.2cm}
\Return $\myobsmap$, $\mynoiseCov$, $\mynoiseChol$
\caption{\texttt{CholeskyExpansion}}
\label{alg:Bayes:Cholesky}
\end{algorithm}

Using the Cholesky decomposition, the inverse of $\mynoiseCov[\myextendedobsmap]$ factorizes to
\begin{align*}
\mynoiseCov[\myextendedobsmap]^{-1} =
\left(
\begin{array}{cc}
\mynoiseCholadj & \myCholvectorw \\
\mathbf{0} & \myCholbeta
\end{array}
\right)^{-1}
\left(
\begin{array}{cc}
\mynoiseChol & \mathbf{0} \\
\myCholvectorw^T & \myCholbeta
\end{array}
\right)^{-1}
=
\left(
\begin{array}{cc}
\mynoiseCholinvadj & \myCholvectorx \\
\mathbf{0} & 1/\myCholbeta
\end{array}
\right)
\left(
\begin{array}{cc}
\mynoiseCholinv & \mathbf{0} \\
\myCholvectorx^T & 1/\myCholbeta
\end{array}
\right),
\end{align*}
where 
\begin{align*}
\myCholvectorx := - \frac{1}{\myCholbeta} \mynoiseCholinvadj \myCholvectorw =  - \frac{1}{\myCholbeta} \mynoiseCholinvadj  \mynoiseCholinv \myCholvector = - \frac{1}{\myCholbeta} \mynoiseCovinvobs \myCholvector.
\end{align*}
For an arbitrary state $\myx \in \myX$, the norm of the extended observation $\myextendedobsmap(\myx) = \mytranspose{\myobs{\myx}^T, \myobsk[]{\myx}} \in \mathbb{R}^{\mydatadim + 1}$ in the corresponding norm $\mygennorm{\cdot}{\mynoiseCov[\myextendedobsmap]^{-1}}$ is hence connected to the original observation $\myobs{\myx} \in \mathbb{R}^{\mydatadim}$ in the original norm $\mynoisenorm{\cdot}$ via
\begin{equation} \label{eq:Bayes:observabilityimprovement}
\begin{aligned}
\mygennorm{~\myextendedobsmap(\myx)~}{\mynoiseCov[\myextendedobsmap]^{-1}}^2 
&= \left(
\begin{array}{c}
\myobs \mygenstate \\
\myobsk[]{\myx}
\end{array}
\right)^T
\left(
\begin{array}{cc}
\mynoiseCovobs & \myCholvector \\
\myCholvector^T & \myCholvar
\end{array}
\right)^{-1}
\left(
\begin{array}{c}
\myobs \mygenstate \\
\myobsk[]{\myx}
\end{array}
\right)\\
&=
\left(
\begin{array}{c}
\myobs \mygenstate \\
\myobsk[]{\myx}
\end{array}
\right)^T
\left(
\begin{array}{cc}
\mynoiseCholinvadj & \myCholvectorx \\
\mathbf{0} & 1/\myCholbeta
\end{array}
\right)
\left(
\begin{array}{cc}
\mynoiseCholinv & \mathbf{0} \\
\myCholvectorx^T & 1/\myCholbeta
\end{array}
\right)
\left(
\begin{array}{c}
\myobs \mygenstate \\
\myobsk[]{\myx}
\end{array}
\right) \\
&=
\left(
\begin{array}{c}
\mynoiseCholinv \myobs \mygenstate \\
\myCholvectorx^T \myobs \mygenstate + \myobsk[]{\myx} / \myCholbeta
\end{array}
\right)^T
\left(
\begin{array}{c}
\mynoiseCholinv \myobs \mygenstate \\
\myCholvectorx^T \myobs \mygenstate + \myobsk[]{\myx} / \myCholbeta
\end{array}
\right) \\
&= (\myobs \mygenstate)^T \mynoiseCholinvadj  \mynoiseCholinv \myobs \mygenstate  + (\myCholvectorx^T \myobs \mygenstate + \myobsk[]{\myx} / \myCholbeta)^2\\
&= \mynoisenorm{\myobs \mygenstate}^2 + (\myCholvectorx^T \myobs{ \mygenstate} + \myobsk[\mydatadim+1]{\myx} / \myCholbeta)^2 \\
&\ge \mynoisenorm{\myobs \mygenstate}^2.
\end{aligned}
\end{equation}
We conclude from this result that the norm $\mynoisenorm{\myobs \mygenstate}$ of any observation, and therefore also the continuity coefficient $\myobscont$ defined in \eqref{def:Bayes:obscont}, is increasing under expansion of $\myobsmap$ despite the change in norms.
For any configuration $\mypara$, the observability coefficients $\myparatoobsinfsup$ and $\myobsinfsup$ are thus non-decreasing when sensors are selected iteratively.

\begin{algorithm}[t]
\DontPrintSemicolon
\KwIn {observation operator $\myobsmap = \mytranspose{\myobskmap[1], \dots, \myobskmap[\mydatadim]}$, Cholesky matrix $\mynoiseChol $, sensor candidate $\ell \in \myXdual$, state $\myx \in \myX$}
\vspace{0.2cm}
$\mydata \gets \myobs{\myx}$,
$\mathbf{z} \gets \mynoiseChol^{-1} \mydata$ \tcp*[r]{preparation}
\If{$\myiter = 0$}{
\Return $\myobsk[]{\myxiter}^2 / \myCov(\myobskmap[], \myobskmap[])$ \tcp*[r]{one sensor only}
}\Else{
$\mathbf{v} \gets \mytranspose{\myCov(\myobskmap[1], \myobskmap[]), \dots, \myCov(\myobskmap[\myiter], \myobskmap[])} \in \mathbb{R}^{\myiter}$\tcp*[r]{general case}
$\mathbf{w} \gets \mynoiseChol^{-1}\mathbf{v} \in \mathbb{R}^{\myiter}$\;
\Return $\frac{\left( \myobsk[]{\myxiter} -\mathbf{w}^T \mathbf{z} \right)^2}{\myCov(\myobskmap[], \myobskmap[]) - \mathbf{w}^T \mathbf{w}}$
}
\caption{\texttt{ObservabilityGain}}
\label{alg:Bayes:observabilityGain}
\end{algorithm}

Given a state $\mygenstate \in \myX$ and an observation operator $\myobsmap$, we can determine the sensor $\myobskmap[\mydatadim+1] \in \mylibrary$ that increases the observation of $\mygenstate$ the most by comparing the increase $(\myCholvectorx^T \myobs{ \mygenstate} + \myobsk[]{\myx} / \myCholbeta)^2$ for all $\ell \in \mylibrary$.
Algorithm \ref{alg:Bayes:observabilityGain} summarizes the computation of this observability gain for use in the sensor selection algorithm (see Section \ref{sec:algorithm}).
Its general runtime is determined by $\mydatadim + 1$ sensor evaluations and two linear solves with the triangular Cholesky matrix $\mynoiseChol$ in $\mycurlyO{\mydatadim^2}$.
When called with the same $\myobsmap$ and the same state $\myx$ for different candidate sensors $\ell$, the preparation step must only be performed once, which reduces the runtime to one sensor evaluation and one linear system solve in all subsequent calls.
Compared to computing $\mygennorm{~\myextendedobsmap(\myx)~}{\mynoiseCov[\myextendedobsmap]^{-1}}^2$ for all $\mylibrarydim$ candidate sensors in the library $\mylibrary$, we save $\mycurlyO{\mylibrarydim \mydatadim^2}$.

%% file: 632_obsComp.tex
\subsection{Computation of the observability coefficient}
\label{sec:obsComp}

We next discuss the computation of the observability coefficient $\myparatoobsinfsup$ for a given configuration $\mypara$ and observation operator $\myobsmap$.

Let $\mypriorCov = \mySVDU^T \mySVDSigma  \mySVDU$ be the eigenvalue decomposition of the \ac{spd}\ prior covariance matrix with $\mySVDU = \mymatrixsum{\myUfebasis[1], \dots, \myUfebasis[\myfedimU]} \in \mathbb{R}^{\myfedimU \times \myfedimU}$, $\myUfebasis \in \mathbb{R}^{\myfedimU}$ orthonormal in the Euclidean inner product, and $\mySVDSigma  = \diag(\mySVDSigmaEntry[1], \dots, \mySVDSigmaEntry[\myfedimU])$ a diagonal matrix containing the eigenvalues $\mySVDSigmaEntry[1] \ge \dots \ge \mySVDSigmaEntry[\myfedimU] > 0$ in decreasing order.
Using the eigenvector basis $\{\myUfebasis[m]\}_{m=1}^{\myfedimU}$, we define the matrix 
\begin{align}\label{eq:Bayes:Mallobservations}
\mathbf{M}(\mypara) := \mymatrixsum{
\myobs{\mystate[{\myUfebasis[1]}]},
\dots,
\myobs{\mystate[{\myUfebasis[\myfedimU]}]}
} \in \mathbb{R}^{\myiter \times \myfedimU}
\end{align}
featuring all observations of the associated states $\mystate[{\myUfebasis}]$ for the configuration $\mypara$.
The observability coefficient $\myparatoobsinfsup$ can then be computed as the square root of the minimum eigenvalue $\myeigvalmin$ of the generalized eigenvalue problem
\begin{align}\label{eq:Bayes:obscoef:comp}
\mathbf{M}(\mypara)^T \mynoiseChol^{-T}\mynoiseChol^{-1}\mathbf{M}(\mypara)  \myuvec_{\min}
&= \myeigvalmin \mySVDSigma^{-1}\myuvec_{\min}.
\end{align}
Note that \eqref{eq:Bayes:obscoef:comp} has $\myfedimU$ real, non-negative eigenvalues because the matrix on the left is symmetric positive semi-definite, and $\mySVDSigma$ is \ac{spd}\ (c.f. \cite{Golub2013}).
The eigenvector $\myuvec_{\min}$ contains the basis coefficients in the eigenvector basis $\{\myUfebasis[m]\}_{m=1}^{\myfedimU}$ of the ``worst-case" parameter, i.e. the infimizer of $\myparatoobsinfsup$.

\begin{remark}
For computing $\myobsinfsup$, we exchange the right-hand side matrix $\mySVDSigma^{-1}$ in \eqref{eq:Bayes:obscoef:comp} with the $\myX$-inner-product matrix for the states $\mystate[{\myUfebasis[1]}], \dots, \mystate[{\myUfebasis[\myfedimU]}]$.
\end{remark}

The solution of the eigenvalue problem can be computed in $\mycurlyO{\myfedimU^3}$, with an additional $\mycurlyO{\myfedimU \mydatadim^2 + \myfedimU^2 \mydatadim}$ for the computation of the left-hand side matrix in \eqref{eq:Bayes:obscoef:comp}.
The dominating cost is hidden in $\mathbf{M}(\mypara)$ since it requires $\mydatadim \myfedimU$ sensor observations and $\mydatadim$ full-order model solves.
To reduce the computational cost, we therefore approximate $\myparatoobsinfsup$ with $\myparatoobsrbinfsup$ by exchanging the full-order states $\mystate[\myUfebasis]$ in \eqref{eq:Bayes:Mallobservations} with their reduced-order approximations $\myRBstate[\myUfebasis]$.
The procedure is summarized in Algorithm \ref{alg:Bayes:obscoef}.

\begin{algorithm}[t]
\DontPrintSemicolon
\KwIn {configuration $\mypara \in \myparadomain$, observation operator $\myobsmap = \mytranspose{\myobskmap[1], \dots, \myobskmap[\mydatadim]}$ with $\mydatadim > 0$, Cholesky matrix $\mynoiseChol$}
\vspace{0.2cm}
$N \gets \min\{ \myfedimU, \mydatadim \}$ \tcp*[r]{parameter restriction}
$\mathbf{M} \gets 
\mymatrixsum{
\myobs{\myRBstate[{\myUfebasis[1]}]},
\dots,
\myobs{\myRBstate[{\myUfebasis[N]}]}}$, $\mathbf{S} \gets \myveceval{\myXinner{\mystate[{\myUfebasis[i]}]}{\mystate[{\myUfebasis[j]}]}}{i,j =1}^{N}$ \tcp*[r]{matrix setup}
Find  $(\myeigvalmin, \myuvec_{\min})$ of $\mytranspose{\mynoiseChol^{-1}
\mathbf{M}}\left[\mynoiseChol^{-1}
\mathbf{M}\right]
 \myuvec_{\min}
= \myeigvalmin \mathbf{S} \myuvec_{\min}$ \tcp*[r]{eigenvalue problem}
\vspace{0.2cm}
\Return {$\sqrt{\myeigvalmin}$, $\myuvec_{\min}$}
\caption{\texttt{SurrogateObservability}}
\label{alg:Bayes:obscoef}
\end{algorithm}

\begin{remark}
If $\mydatadim < \myfedimU$, Algorithm \ref{alg:Bayes:obscoef} restricts the parameter space, as discussed in Section \ref{sec:Bayes:parameterreduction}, to the span of the first $\mydatadim$ eigenvectors $\myUfebasis[1], \dots, \myUfebasis[\mydatadim]$ encoding the least certain directions in the prior.
A variation briefly discussed in \cite{Binev2018b} in the context of the \ac{pbdw} method to prioritize the least certain parameters even further is to only expand the parameter space once the observability coefficient on the subspace surpasses a predetermined threshold.
\end{remark}

%% file: 633_algorithm.tex
\subsection{Sensor selection}
\label{sec:algorithm}

In our sensor selection algorithm, we iteratively expand the observation operator $\myobsmap$ and thereby increase the observability coefficient $\myobsmap$ for all $\mypara \in \myparadomain$.
Although this procedure cannot guarantee finding the maximum observability over all sensor combinations, the underlying greedy searches are well-established in practice, and can be shown to perform with exponentially decreasing error rates in closely related settings, see \cite{Binev2011a, Binev2018b, Cohen2020,Buffa2012,Jagalur-Mohan2021}.
In each iteration, the algorithm performs two main steps:
\begin{itemize}
\item A \textbf{greedy search} over a training set $\myparadomaintrain \subset \myparadomain$ to identify the configuration $\mypara \in \myparadomaintrain$ for which the observability coefficient $\myparatoobsinfsup$ is minimal;
\item A \textbf{data-matching step} to identify the sensor in the library that maximizes the observation of the ``worst-case" parameter at the selected configuration $\mypara$.
\end{itemize}
The procedure is summarized in Algorithm \ref{alg:Bayes:greedyOMP}.
It terminates when $\mydatadimmax \le \mylibrarydim$ sensors have been selected.\footnote{This termination criterion can easily be adapted to prescribe a minimum value of the observability coefficient. This value should be chosen with respect to the observability $\myparatoobsinfsup[\mylibrary]$ achieved with the entire sensor library.}
In the following, we explain its computational details.

\begin{algorithm}[t]
\DontPrintSemicolon
\KwIn {sensor library $\mylibrary \subset \myXdual$, training set $\myparadomaintrain \subset \myparadomain$, maximum number of sensors $\mydatadimmax \le \myabs{\mylibrarydim}$, surrogate model $\myRBmodel$, covariance function $\myCov : \mylibrary \times \mylibrary \rightarrow \mathbb{R}$}
\vspace{0.2cm}
Compute $\mypriorCov = \mytranspose{\myUfebasis[1], \dots, \myUfebasis[\myfedimU]} \mySVDSigma \mymatrixsum{\myUfebasis[1], \dots, \myUfebasis[\myfedimU]}$ \tcp*[r]{eigenvalue decomposition}
For all $\mypara \in \myparadomaintrain$, $1 \le m \le \myfedimU$, compute $\myRBstate[{\myUfebasis[m]}]$\tcp*[r]{preparation}
$\myiter \gets 0$, 
$\myparaiter[0] \gets \text{arg} \max_{\theta \in \myparadomaintrain} \myXnorm{\myRBstate[{\myUfebasis[1]}]}$,
$\myuiter[0] \gets \myUfebasis[1]$ \tcp*[r]{initialization}
\vspace{0.2cm}
\While{$\myiter < \mydatadimmax$}{
\vspace{0.2cm}
Solve full-order equation $\mymodel[\myparaiter](\myxiter, \myuiter)$ for $\myxiter$ \tcp*[r]{"worst-case" state}
\vspace{0.2cm}
$\myobskmap[\myiter+1] \gets \text{arg} \max_{\ell \in \mylibrary} \mathtt{ObservabilityGain}(\myobsmap, \mynoiseChol, \ell)$
\tcp*[r]{sensor selection}
\vspace{0.2cm}
$\myobsmap, \mynoiseCov, \mynoiseChol \gets \mathtt{CholeskyExpansion}(\myobsmap, \mynoiseCov, \mynoiseChol, \myobskmap[\myiter+1])$ \tcp*[r]{expansion}
$\myiter \gets \myiter + 1$ \;
\vspace{0.2cm}
\For{$\mypara \in \myparadomaintrain$}{
$\myobsrbinfsup, \myuvec_{\min}(\mypara) \gets \mathtt{SurrogateObservability}(\theta, \myobsmap, \mynoiseChol)$ \tcp*[r]{update coefficients}
}
\vspace{0.2cm}
$\myparaiter \gets \text{arg} \min_{\mypara \in \myparadomaintrain} \myobsrbinfsup$ \tcp*[r]{greedy step}
$\myuiter \gets \sum_{m=1}^{\min\{\myfedimU, \myiter\}} \myveceval{\myuvec_{\min}(\myparaiter)}{m} \myUfebasis[m]$ \;
}
\vspace{0.2cm}
\Return {$\myobsmap$, $\mynoiseChol$}
\caption{\texttt{SensorSelection}}
\label{alg:Bayes:greedyOMP}
\end{algorithm}

\subsubsection*{Preparations}

In order to increase $\myparatoobsinfsup$ uniformly over the hyper-parameter domain $\myparadomain$, we consider a finite training set, $\myparadomaintrain \subset \myparadomain$, that is chosen to be fine enough to capture the $\mypara$-dependent variations in $\mystate$.
We assume a reduced-order model is available such that we can compute approximations $\myRBstate[{\myUfebasis[m]}] \approx \mystate[{\myUfebasis[m]}]$ for each $\mypara \in \myparadomaintrain$ and $1 \le m \le \myfedimU$ within an acceptable computation time while guaranteeing the accuracy requirement \eqref{eq:Bayes:RBapprox}.
If necessary, the two criteria can be balanced via adaptive training domains (e.g., \cite{Eftang2010,Eftang2011}).

\begin{remark}
If storage allows (e.g.,~with projection-based surrogate models), we only compute the surrogate states once and avoid unnecessary re-computations when updating the surrogate observability coefficients $\myparatoobsrbinfsup$ in each iteration.
\end{remark}

As a first ``worst-case" parameter direction, $\myuiter[0]$, we choose the vector $\myUfebasis[1]$ with the largest prior uncertainty.
Likewise, we choose the ``worst-case" configuration $\myparaiter \in \myparadomain$ as the one for which the corresponding state $\myRBstate[{\myUfebasis[1]}]$ is the largest.

\subsubsection*{Data-matching step}
In each iteration, we first compute the full-order state $\myxiter = \myx_{\myparaiter}(\myuiter)$ at the ``worst-case'' parameter $\myuiter$ and configuration $\myparaiter$.
We then choose the sensor $\myobskmap[{\myiter+1}]$ which most improves the observation of the ``worst-case'' state $\myxiter$ under the expanded observation operator $[\myobsmap^T, \myobskmap[{\myiter+1}]]^T$ and its associated norm.
We thereby iteratively approximate the information that would be obtained by measuring with all sensors in the library $\mylibrary$.
For fixed $\myparaiter$ and in combination with selecting $\mygenstate$ to have the smallest observability in $\myW$, we arrive at an algorithm similar to \acl{wcomp} (c.f. \cite{Binev2018b, Maday2015e}) but generalized to deal with the covariance function $\myCov$ in the noise model \eqref{eq:Bayes:noisemodel}.

\begin{remark}
We use the full-order state $\myx_{\myparaiter}(\myuiter)$ rather than its reduced-order approximation in order to avoid training on local approximation inaccuracies in the reduced-order model.
Here, by using the ``worst-case" parameter direction $\myuiter$, we only require a single full-order solve per iteration instead of the $\myfedimU$ required for setting up the entire posterior covariance matrix $\mypostCov$.
\end{remark}

\subsubsection*{Greedy step}

We train the observation operator $\myobsmap$ on all configurations $\mypara \in \myparadomaintrain$ by varying for which $\mypara$ the ``worst-case" state is computed.
Specifically, we follow a greedy approach where, in iteration $\myiter$, we choose the minimizer $\myparaiter$ of $\myparatoobsinfsup$ over the training domain $\myparadomaintrain$, i.e., the configuration for which the current observation operator $\myobsmap$ is the least advantageous.
The corresponding ``worst-case" parameter $\myuiter$ is the parameter direction for which the least significant observation is achieved.
By iteratively increasing the observability at the ``worst-case" parameters and hyper-parameters, we increase the minimum of $\myparatoobsinfsup$ throughout the training domain.

\begin{remark}
Since the computation of $\myparatoobsrbinfsup$ requires as many reduced-order model solves as needed for the posterior covariance matrix over the surrogate model, it is possible to directly target an (approximated) \ac{oed} utility function in the greedy step in place of $\myobsrbinfsup$ without major concessions in the computational efficiency.
The \ac{omp} step can then still be performed for the ``worst-case" parameter with only one full-order model solve, though its benefit for the utility function should be evaluated carefully.
\end{remark}

\subsubsection*{Runtime}

Assuming the dominating computational restriction is the model evaluation to solve for $\mystate$ -- as is usually the case for \ac{pde} models -- then the runtime of each iteration in Algorithm \ref{alg:Bayes:greedyOMP} is determined by one full-order model evaluation, and $\mylibrarydim$ sensor measurements of the full-order state.
Compared to computed the posterior covariance matrix for the chosen configuration, the \ac{omp} step saves $\myfedim - 1$ full-order model solves.

The other main factor in the runtime of Algorithm \ref{alg:Bayes:greedyOMP} is the $\myabs{\myparadomaintrain} \myfedimU$ reduced-order model evaluations with $\mylibrarydim$ sensor evaluations each that need to be performed in each iteration (unless they can be pre-computed).
The parameter dimension $\myfedimU$ not only enters as a scaling factor, but also affects the cost of the reduced-order model itself since larger values of $\myfedimU$ generally require larger or more complicated reduced-order models to achieve the desired accuracy \eqref{eq:Bayes:RBapprox}.
In turn, the computational cost of the reduced-order model indicates how large $\myparadomaintrain$ may be chosen for a given computational budget.
While some cost can be saved through adaptive training sets and models, overall, this connection to $\myfedimU$ stresses the need for an adequate initial parameter reduction as discussed in Section \ref{sec:Bayes:parameterreduction}.

%% file: 640_results.tex
\section{Numerical Results}
\label{sec:Bayes:results}

We numerically confirm the validity of our sensor selection approach using a geophysical model of a section of the Perth Basin in Western Australia.
The basin has raised interest in the geophysics community due to its high potential for geothermal energy, e.g., \cite{Regenauer-Lieb2007,Corbel2012,Sheldon2012,Schilling2013,Pujol2015}.
We focus on a subsection that spans an area of $63\ \text{km} \times 70\ \text{km}$ and reaches 19 km below the surface.
The model was introduced in \cite{Wellmann2014} and the presented section of the model was discussed extensively in the context of \ac{mor} in \cite{Degen2020b,Degen2020Phd}.
In particular, the subsurface temperature distribution is described through a steady-state heat conduction problem with different subdomains for the geological layers, and local measurements may be obtained through boreholes.
The borehole locations need to be chosen carefully due to their high costs (typically several million dollars, \cite{Bauer2014}), which in turn motivates our application of Algorithm \ref{alg:Bayes:greedyOMP}.
For demonstration purposes, we make the following simplifications to our test model:
1) We neglect radiogenic heat production;
2) we merge geological layers with similar conductive behaviors; and
3) we scale the prior to emphasize the influence of different sensor measurements on the posterior.
All computations were performed in \texttt{Python 3.7} on a computer with a 2.3 GHz Quad-Core Intel Core i5 processor and 16 GB of RAM.
The code will be available in a public GitHub repository for another geophysical test problem.\footnote{The Perth Basin Model is available upon request from the third author.}

%% file: 641_model.tex
\subsection{Model Description}\label{sec:Bayes:results:model}

\begin{figure}
    \centering
    \includegraphics[width = \textwidth]{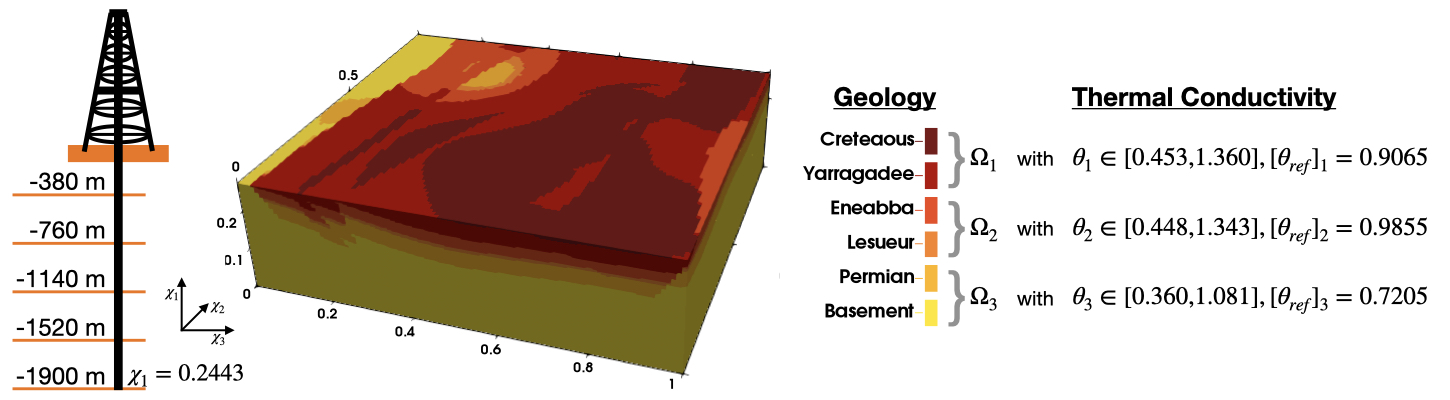}
    \caption{Schematic overview of the Perth Basin section including (merged) geological layers, depths for potential measurements, and configuration range for thermal conductivity $\mypara$ on each subdomain. The bounds are obtained from the reference values (c.f. \cite{Wellmann2014, Degen2020b}) with a $\pm 50\%$ margin. Adapted from \cite{Degen2020b}.}
    \label{fig:Bayes:PerthBasinExerp}
\end{figure}

We model the temperature distribution $\myx_{\mypara}$ with the steady-state \ac{pde}
\begin{align}\label{eq:Bayes:PerthBasinPDE}
- \nabla \left(\mypara \nabla \myx_{\mypara}\right) = 0
\qquad \text{in } \myOmega := (0, 0.2714) \times (0,0.9) \times (0,1) \subset \mathbb{R}^3,
\end{align}
where the domain $\myOmega$ is a non-dimensionalized representation of the basin, and $\mypara : \myOmega \rightarrow \mathbb{R}_{>0}$ the local thermal conductivity.
The section comprises three main geological layers $\myOmega = \bigcup_{i = 1, 2, 3}\myOmega_i$, each characterized by different rock properties, i.e. thermal conductivity $\left. \mypara \right|_{\myOmega_i} \equiv \mypara_i$ shown in Figure \ref{fig:Bayes:PerthBasinExerp}.
We consider the position of the geological layers to be fixed as these are often determined beforehand by geological and geophysical surveys but allow the thermal conductivity to vary.
In a slight abuse of notation, this lets us identify the field $\mypara$ with the vector
\begin{align*}
\mypara = (\mypara_1, \mypara_2, \mypara_3) \in \myparadomain :=  [0.453 , 1.360]\times [0.448 , 1.343] \times [0.360 , 1.081].
\end{align*}
in the hyper-parameter domain $\myparadomain$.

We impose zero-Dirichlet boundary conditions at the surface\footnote{Non-zero Dirichlet boundary conditions obtained from satellite data could be considered via a lifting function and an affine transformation of the measurement data (see \cite{Degen2020Phd}).}, and zero-Neumann (``no-flow") boundary conditions at the lateral faces of the domain.
The remaining boundary $\myInflow$ corresponds to an area spanning 63 km $\times$ 70 km area in the Perth basin 19 km below the surface.
At this depth, local variations in the heat flux have mostly stabilized which makes modeling possible, but since most boreholes -- often originating from hydrocarbon exploration -- are found in the uppermost 2 km we treat it as uncertain.
Specifically, we model it as a Neumann boundary condition
\begin{align*}
\mathbf{n} \cdot \nabla \myx_{\mypara}
&= \myu \cdot \mathbf{p}
&\text{a.e. on } \myInflow := \{0\} \times [0,0.9] \times [0,1]
\end{align*}
where $\mathbf{n} : \myInflow \rightarrow \mathbb{R}^3$ is the outward pointing unit normal on $\myOmega$, $\mathbf{p} : \myInflow \rightarrow \mathbb{R}^5$ is a vector composed of quadratic, $L^2(\myInflow)$-orthonormal polynomials on the basal boundary that vary either in north-south or east-west direction, and $\myu \sim \myprior = \mathcal{N}(\mypriormean, \mypriorCov)$ is a random variable.
The prior is chosen such that the largest uncertainty is attributed to a constant entry in $\mathbf{p}$, and the quadratic terms are treated as the most certain with prior zero.
This setup reflects typical geophysical boundary conditions, where it is most common to assume a constant Neumann heat flux (e.g., \cite{Degen2020b}), and sometimes a linear one (e.g., \cite{Wellmann2014}).
With the quadratic functions, we allow an additional degree of freedom than typically considered.

The problem is discretized using a linear \ac{fe} basis of dimension 132,651.
The underlying mesh was created with \texttt{GemPy} (\cite{Varga2019}) and \acs{moose} (\cite{Permann2020}).
Since the \ac{fe} matrices decouple in $\mypara$, we precompute and store an affine decomposition using \texttt{DwarfElephant} (\cite{Degen2020b}).
Given a configuration $\mypara$ and a coefficient vector $\myu$ for the heat flux at $\myInflow$, the computation of a full-order solution $\mystate \in \myX$ then takes 2.96 s on average.
We then exploit the affine decomposition further to construct a \ac{rb} surrogate model via a greedy algorithm (c.f. \cite{Binev2011a, Dahmen2014}).
Using the inner product\footnote{Note that $\myXinner{\cdot}{\cdot}$ is indeed an inner product due to the Dirichlet boundary conditions.}
$\myXinner{\myx}{\myXtestfct} := \int_{\myOmega} \nabla \myx \cdot \nabla \myXtestfct d\myOmega$ and an \acs{aposteriori} error bound $\myerrorboundPrimal$, we prescribe the relative target accuracy 
\begin{align}\label{eq:Bayes:results:errorbound}
\max_{\myu \in \myU} \frac{\myXnorm{\mystate - \myRBstate}}{\myXnorm{\myRBstate}}
\le \max_{\myu \in \myU} \frac{\myerrorboundPrimal}{\myXnorm{\myRBstate}}
< \varepsilon := \mathtt{1e-4}
\end{align}
to be reached for 511,000 consecutively drawn, uniformly distributed samples of $\mypara$.
The training phase and final computational performance of the \ac{rb} surrogate model are summarized in Figure \ref{fig:Bayes:RBconstruction}.
The speedup of the surrogate model (approximately a factor of 3,000 without error bounds) justifies its offline training time, with computational savings expected already after 152 approximations of $\myparatoobsinfsup$.

\begin{figure}[t]
\centering
\begin{minipage}{0.6\textwidth}
\includegraphics[width=\textwidth]{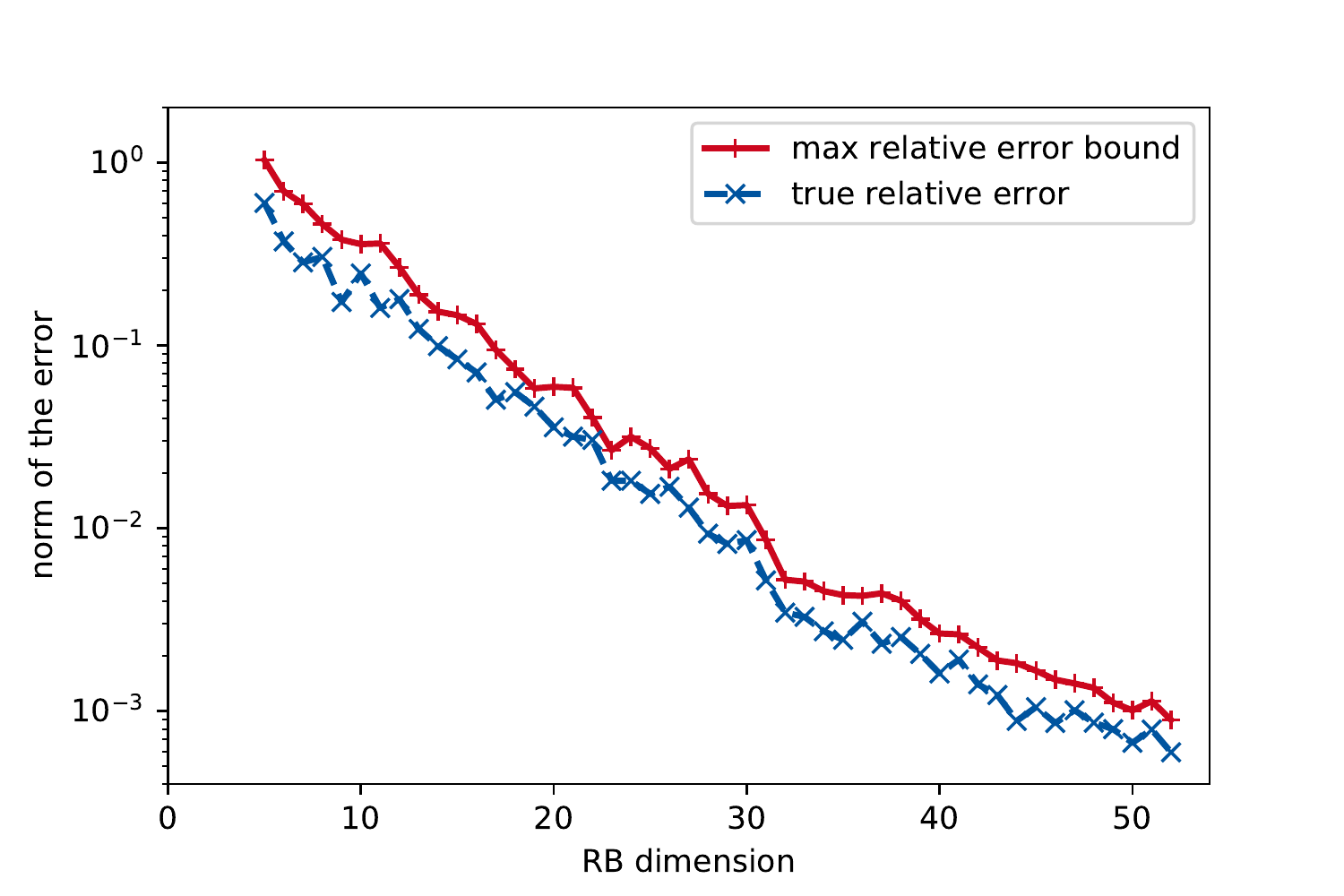} 
\end{minipage}
\begin{minipage}{0.39\textwidth}
\begin{tabular}{ll}
\textbf{Reduced-order model} & \\
RB dimension & 83 \\
training time & 37.58 min \\
training accuracy & \texttt{1e-4} \\
RB solve & 0.97 ms \\
$\hookrightarrow$ speedup & 3,058 \\
RB error bound & 4.78 ms \\
$\hookrightarrow$ speedup & 515
\end{tabular}
\end{minipage}
\caption[Training of the \ac{rb} surrogate model for the Perth Basin section]{Training of the \ac{rb} surrogate model for the Perth Basin section. On the left: Maximum relative error bound \eqref{eq:Bayes:results:errorbound} in the course of the greedy algorithm, computed over the training set $\myparadomaintrain$ together with the true relative error at the corresponding configuration $\mypara$.
On the right: Performance pointers for the obtained \ac{rb} model after \eqref{eq:Bayes:results:errorbound} was reached; online computation times and speedups are averages computed over 1000 randomly drawn configurations $\mypara$.} \label{fig:Bayes:RBconstruction}
\end{figure}

For taking measurements, we consider a $47 \times 47$ grid over the surface to represent possible drilling sites.
At each, a single point evaluation\footnote{Point evaluations are standard for geophysical models because a borehole (diameter approximately 1 m) is very small compared to the size of the model.} of the basin's temperature distribution may be made at any one of five possible depths as shown in Figure \ref{fig:Bayes:PerthBasinExerp}.
In total, we obtain a set $\mylibrary \subset \myOmega$ of $11,045$ admissible points for measurements.
We model the noise covariance between sensors $\myobskmap[\myOmegaPoint], \myobskmap[\tilde{\myOmegaPoint}] \in \mylibrary$ at points $\myOmegaPoint, \tilde{\myOmegaPoint} \in \myOmega$ via
\begin{align*}
\myCov(\myobskmap[\myOmegaPoint],\myobskmap[\tilde{\myOmegaPoint}] )
:= a + b - y(h) 
\end{align*}
with the exponential variogram model
\begin{align*}
y(h) := a + (b-a)\left(\frac32 \max\{\frac{h}{c}, 1\} - \frac12 \max\{\frac{h}{c}, 1\}^3 \right)
\end{align*}
where $h^2 := (\myOmegaPoint_2 - \tilde{\myOmegaPoint}_2)^2 + (\myOmegaPoint_3 - \tilde{\myOmegaPoint}_3)^2$ is the horizontal distance between the points and
\begin{align*}
& a := 2.2054073480730403 && \text{(sill)} \\
& b := 1.6850672040263555 && \text{(nugget)} \\
& c := 20.606782733391228 && \text{(range)}
\end{align*}
The covariance function was computed via kriging (c.f. \cite{Cressie1990}) from the existing measurements \cite{Holgate2010}.
With this covariance function, the noise between measurements at any two sensor locations is increasingly correlated the closer they are on the horizontal plane.
Note that for any subset of sensor locations, the associated noise covariance matrix remains regular as long as each sensor is placed at a distinct drilling location.
We choose this experimental setup because measurements in typical geothermal data sets are often made at the bottom of a borehole (``bottom hole temperature measurements'') within the first 2 km below the surface.

%% file: 642_restricted.tex
\subsection{Restricted Library}
\label{sec:Bayes:results:restricted}

To test the feasibility of the observability coefficient for sensor selection, we first consider a small sensor library (denoted as $\mylibrary_{5\times 5}$ below) with 25 drilling locations positioned on a $5 \times 5$ grid.
We consider the problem of choosing 8 pair-wise different, unordered sensor locations out of the given 25 positions; this is a combinatorial problem with 1,081,575 possible combinations.

\begin{figure}
\centering
\begin{subfigure}{0.48\textwidth}
\includegraphics[width = \textwidth]{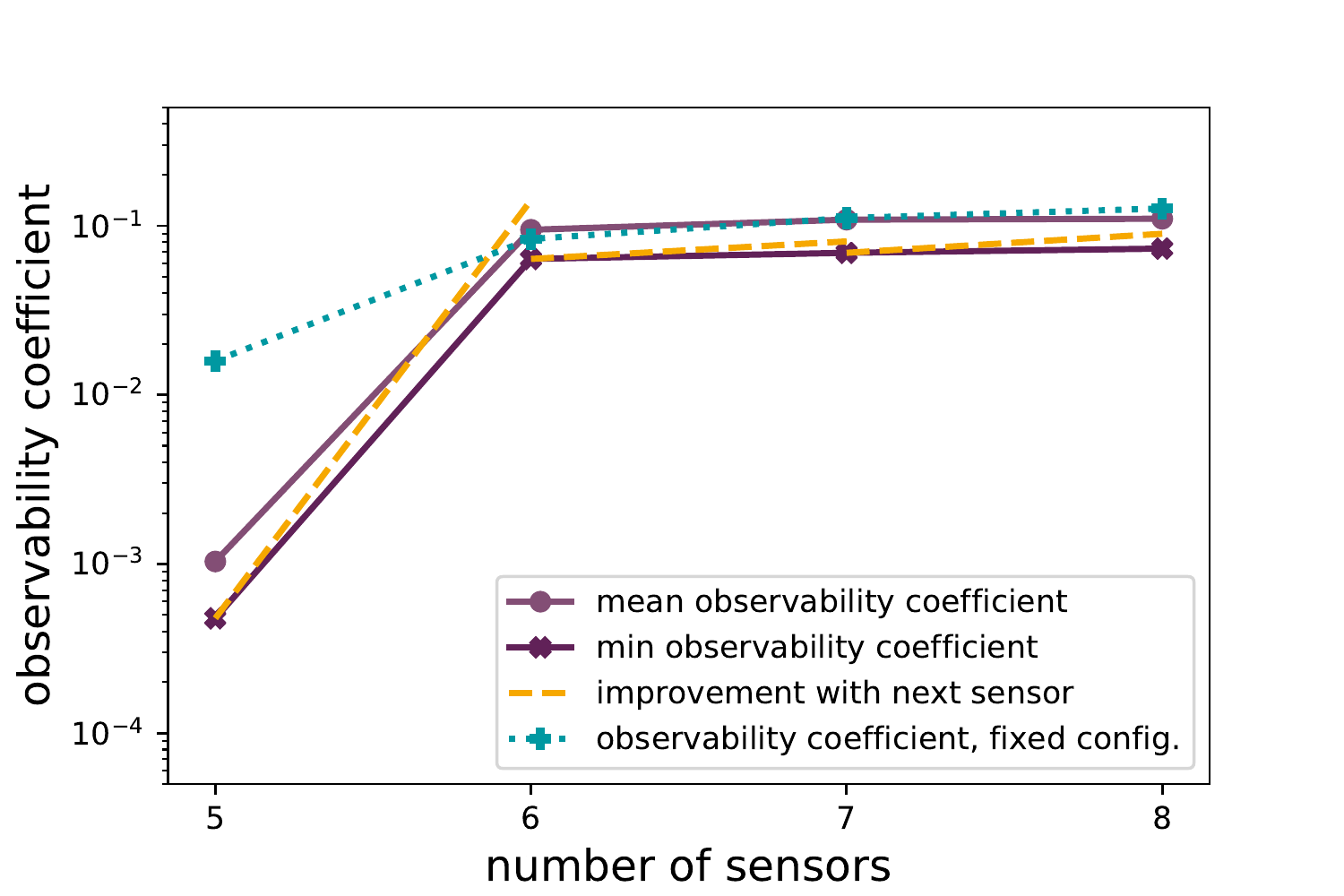} 
\caption{Observability during sensor selection}
\label{fig:Bayes:ex1:obscoefincrease}
\end{subfigure}
\begin{subfigure}{0.48\textwidth}
\includegraphics[width = \textwidth]{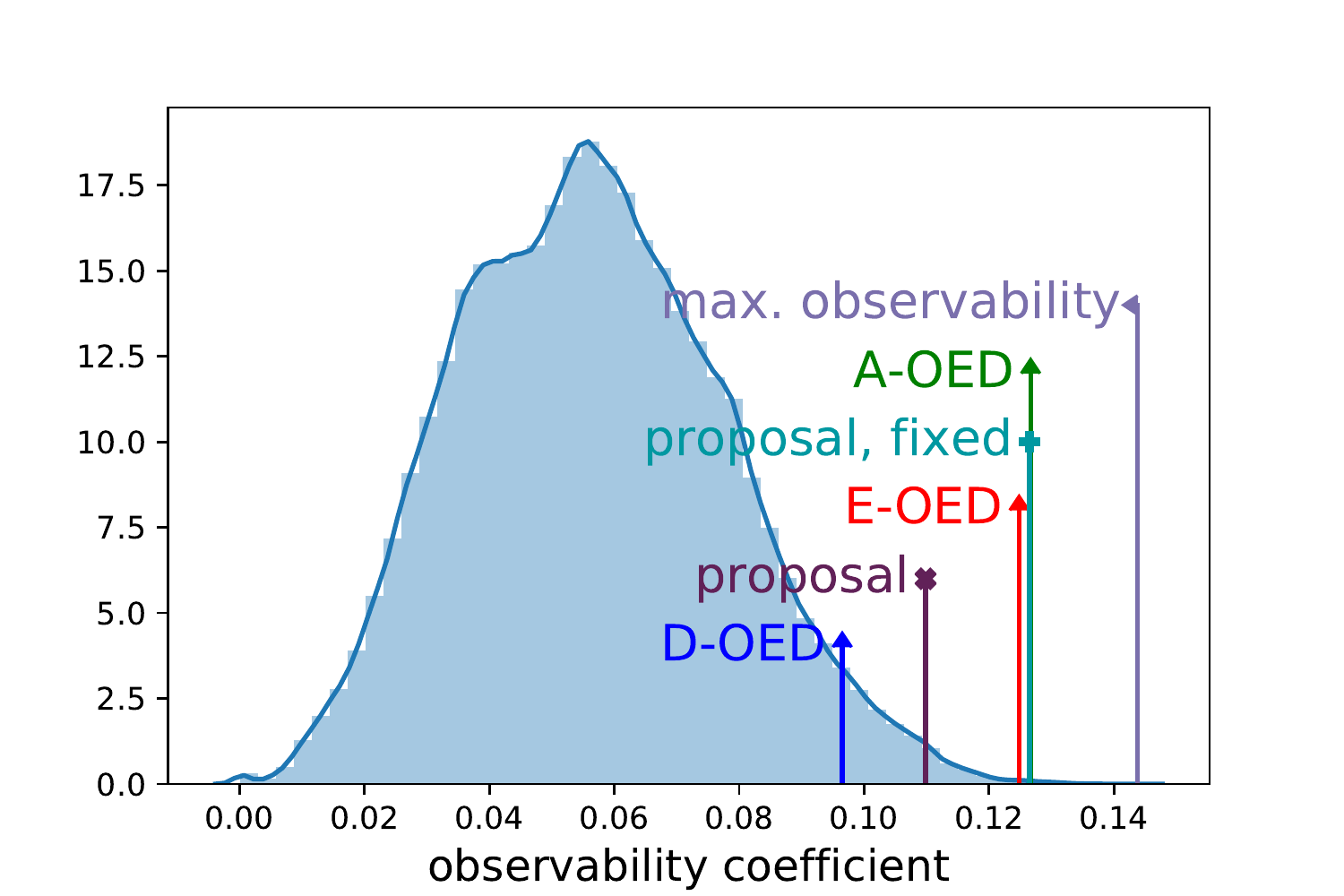} 
\caption{Histogram of $\myparatoobsinfsup[\mypararef]$}
\label{fig:Bayes:ex1:marginObs}
\end{subfigure}
\caption[Observability coefficients for different methods when choosing 8 out of 25 sensor locations]{Observability coefficient for different methods when choosing 8 out of 25 sensor locations.
Left: Minimum and mean over $\mypara$ of $\myparatoobsrbinfsup$ as well as $\myparatoobsinfsup[\mypararef]$ obtained in the course of running Algorithm \ref{alg:Bayes:greedyOMP} once for 512,000 configurations and once for the training set $\{\mypararef\}$.
Right: Distribution of $\myparatoobsinfsup[\mypararef]$ over all possible sensor combinations with indicators for the A-, D-, and E-optimal choices, the combination with maximum observability, and the sensors chosen by the Algorithm \ref{alg:Bayes:greedyOMP} with $\myparadomaintrain$-training (``proposal'', purple, marked ``x'') and $\mypararef$-training (``proposal, fixed'', turquoise, marked ``+'').
Note that the height of the indicator line was chosen solely for readability.
}
\end{figure}

\subsubsection*{Sensor selection}

We run Algorithm \ref{alg:Bayes:greedyOMP}, using the \ac{rb} surrogate model and a training set $\myparadomaintrain \subset \myparadomain$ with 512,000 configurations on an $80 \times 80 \times 80$ regular grid on $\myparadomain$.
When new sensors are chosen, the surrogate observability coefficient $\myparatoobsrbinfsup$ increases monotonously with a strong incline just after the initial $\myfedimU = 5$ sensors, followed by a visible stagnation (see Figure \ref{fig:Bayes:ex1:obscoefincrease}) as is often observed for similar \ac{omp}-based sensor selection algorithms (e.g., \cite{Binev2018b, Maday2015f, Taddei2017d, Aretz-Nellesen2019a}).
Algorithm \ref{alg:Bayes:greedyOMP} terminates in 7.93 min with a minimum reduced-order observability of $\myparatoobsrbinfsup = \texttt{7.3227e-2}$ and an average of \texttt{1.0995e-1}.
At the reference configuration $\mypararef$, the full-order observability coefficient is $\myparatoobsinfsup[\mypararef] = 1.0985$, slightly below the reduced-order average.
We call this training procedure ``$\myparadomaintrain$-training" hereafter and denote the chosen sensors as ``$\myparadomaintrain$-trained sensor set'' in the subsequent text and as ``proposal'' in the plots.

In order to get an accurate understanding of how the surrogate model $\myRBstate$ and the large configuration training set $\myparadomaintrain$ influence the sensor selection, we run Algorithm \ref{alg:Bayes:greedyOMP} again, this time restricted on the full-order \ac{fe} model $\mystatepara{\mypararef}{\myu}$ at only the reference configuration $\mypararef$.
The increase in $\myparatoobsinfsup[\mypararef]$ in the course of the algorithm is shown in Figure \ref{fig:Bayes:ex1:obscoefincrease}.
The curve starts significantly above the average for $\myparadomaintrain$-training, presumably because conflicting configurations cannot occur, e.g., when one sensor would significantly increase the observability at one configuration but cause little change in another.
However, in the stagnation phase, the curve comes closer to the average achieved with $\myparadomaintrain$-training.
The computation finishes within 12.53 s, showing that the long runtime before can be attributed to the size of $\myparadomaintrain$.
The final observability coefficient with 8 sensors is $\myparatoobsinfsup[\mypararef] = \texttt{1.2647e-1}$, above the average over $\myparatoobsrbinfsup$ achieved training on $\myparadomaintrain$.
We call this training procedure ``$\mypararef$-training" hereafter, and the sensor configuration ``$\mypararef$-trained" in the text or ``proposal, fixed config." in the plots.

\subsubsection*{Comparison at the reference configuration}

For comparing the performance of the $\myparadomaintrain$- and $\mypararef$-trained sensor combinations, we compute -- at the reference configuration $\mypararef$ -- all 1,081,575 posterior covariance matrices $\mypostCov[\mypararef, \myobsmap]$ for all unordered combinations $\myobsmap$ of 8 distinct sensors in the sensor library $\mylibrary_{5\times 5}$.
For each matrix, we compute the trace (A-\ac{oed} criterion), the determinant (D-\ac{oed} criterion), the maximum eigenvalue (E-\ac{oed} criterion), and the observability coefficient $\myparatoobsinfsup[\mypararef]$.
This lets us identify the A-, D-, and E-optimal sensor combinations.
The total runtime for these computations is 4 min -- well above the 12.53 s of $\mypararef$-training.
The (almost) 8 min for $\myparadomaintrain$-training remain reasonable considering it is trained on $\myabs{\myparadomaintrain} = 512,000$ configurations and not only $\mypararef$.

A histogram for the distribution of $\myparatoobsinfsup[\mypararef]$ is given in Figure \ref{fig:Bayes:ex1:marginObs} with markers for the values of the A-, D-, and E-optimal choices and the $\myparadomaintrain$- and $\mypararef$-trained observation operators.
Out of these five, the D-optimal choice has the smallest value, since the posterior determinant is influenced less by the maximum posterior eigenvalue and hence the observability coefficient.
In contrast, both the A- and E-optimal sensor choices are among the 700 combinations with the largest $\myparatoobsinfsup[\mypararef]$ (this corresponds to the top 0.065\%).
The $\mypararef$-trained sensors have similar observability and are even among the top 500 combinations.
For the $\myparadomaintrain$- trained sensors, the observability coefficient is smaller, presumably because $\myparadomaintrain$-training is not as optimized for $\mypararef$.
Still, it ranks among the top 0.705 \% of sensor combinations with the largest observability.

\begin{figure}
\centering
\begin{subfigure}{0.54\textwidth}
\includegraphics[width = \textwidth]{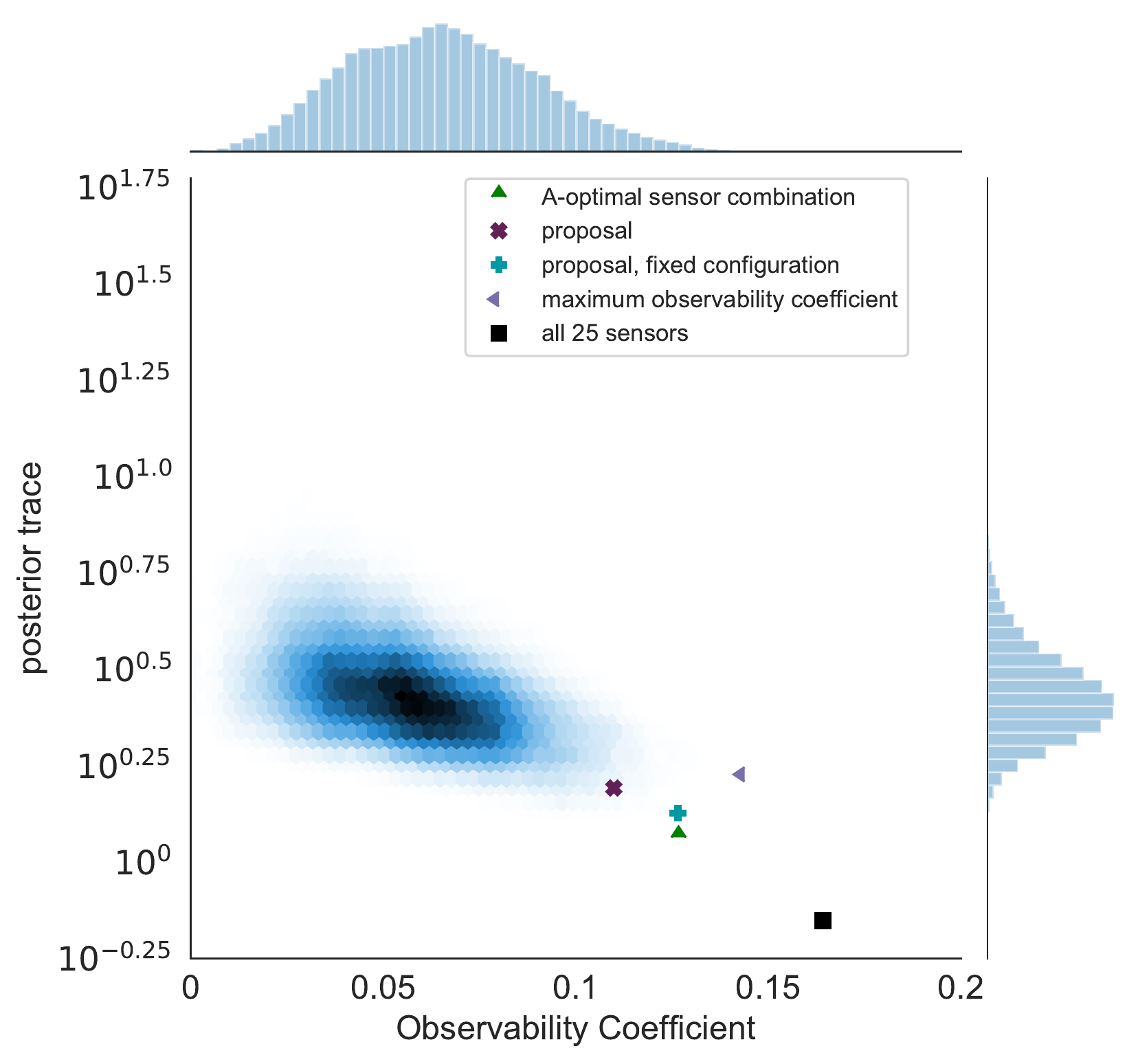} 
\end{subfigure}
\begin{subfigure}{0.45\textwidth}
\centering
~\\[1.3cm]
\begin{tabular}{l|l}
best~~~ & sensor selection \\
\noalign{\smallskip}\hline\noalign{\smallskip}
0.587 \% & proposal \\
0.022 \% & proposal, fixed config. \\
1.778 \% & max. observability
\end{tabular}
~\\[0.5cm]
\includegraphics[width = \textwidth]{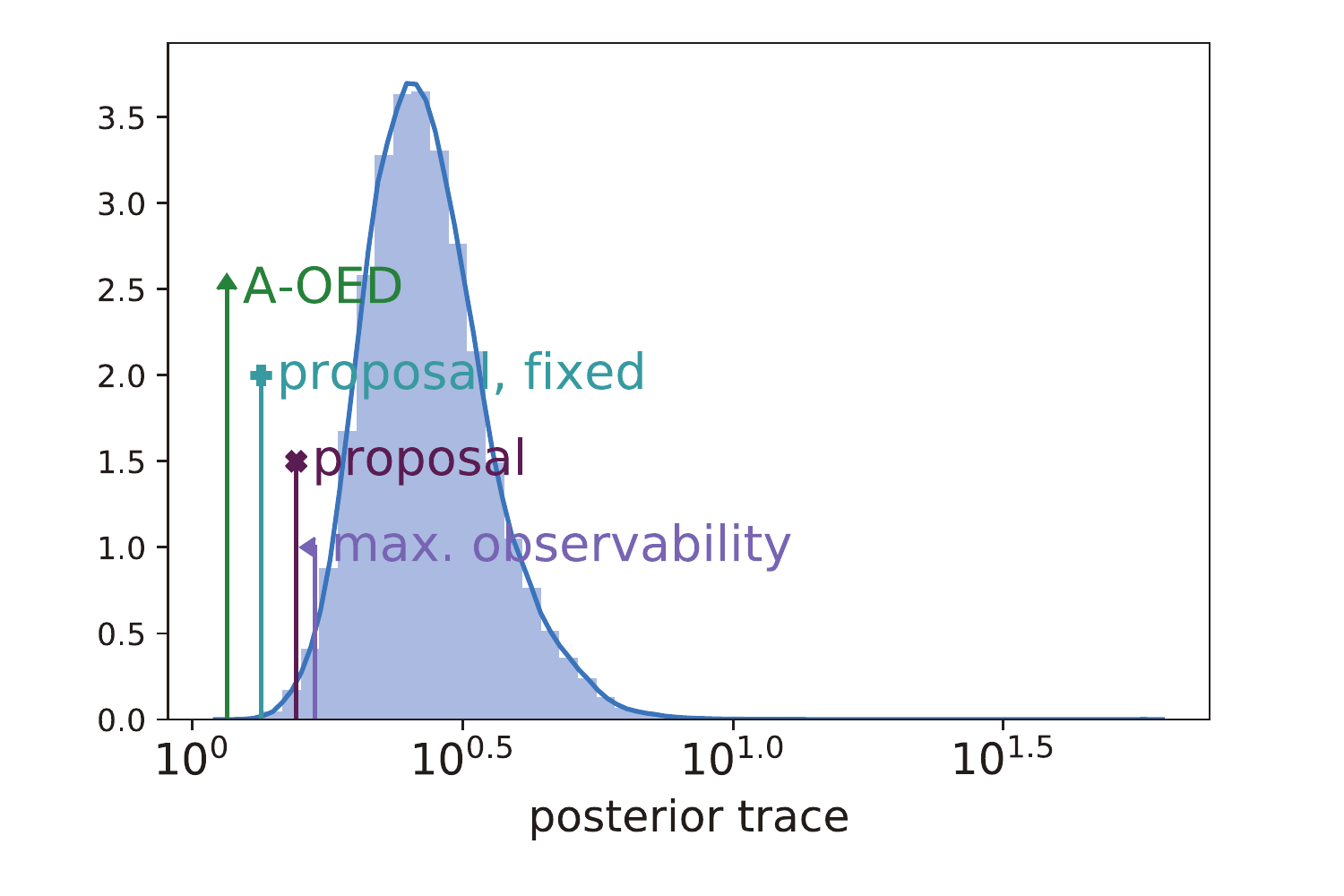} 
\end{subfigure}
\caption{Distribution of $\trace(\mypostCov)$ for $\mypara = \mypararef$ over all 1,081,575 combinations for choosing 8 out of the 25 sensor locations.
On the left: distribution of $\trace(\mypostCov)$ against the observability coefficient $\myparatoobsinfsup[\mypararef]$.
Note that the marginal distribution of the horizontal axis is provided in Figure \ref{fig:Bayes:ex1:marginObs}.
On the right: histogram of $\trace(\mypostCov)$ (marginal distribution for the plot on the left) with for the different sensor combinations (in percent out of 1,081,575 combinations).
The plots include markers for the A-optimal sensor choice, the sensors chosen by Algorithm \ref{alg:Bayes:greedyOMP} with $\myparadomaintrain$-training (``proposal'') and with $\{\mypararef\}$-training (``proposal, fixed configuration"), the sensor combination with maximum observability $\myparatoobsinfsup[\mypararef]$, and when all 25 sensors are included.
}\label{fig:Bayes:ex1:distributionplots:A}
\end{figure}

\begin{figure}
\begin{subfigure}{0.54\textwidth}
\includegraphics[width = \textwidth]{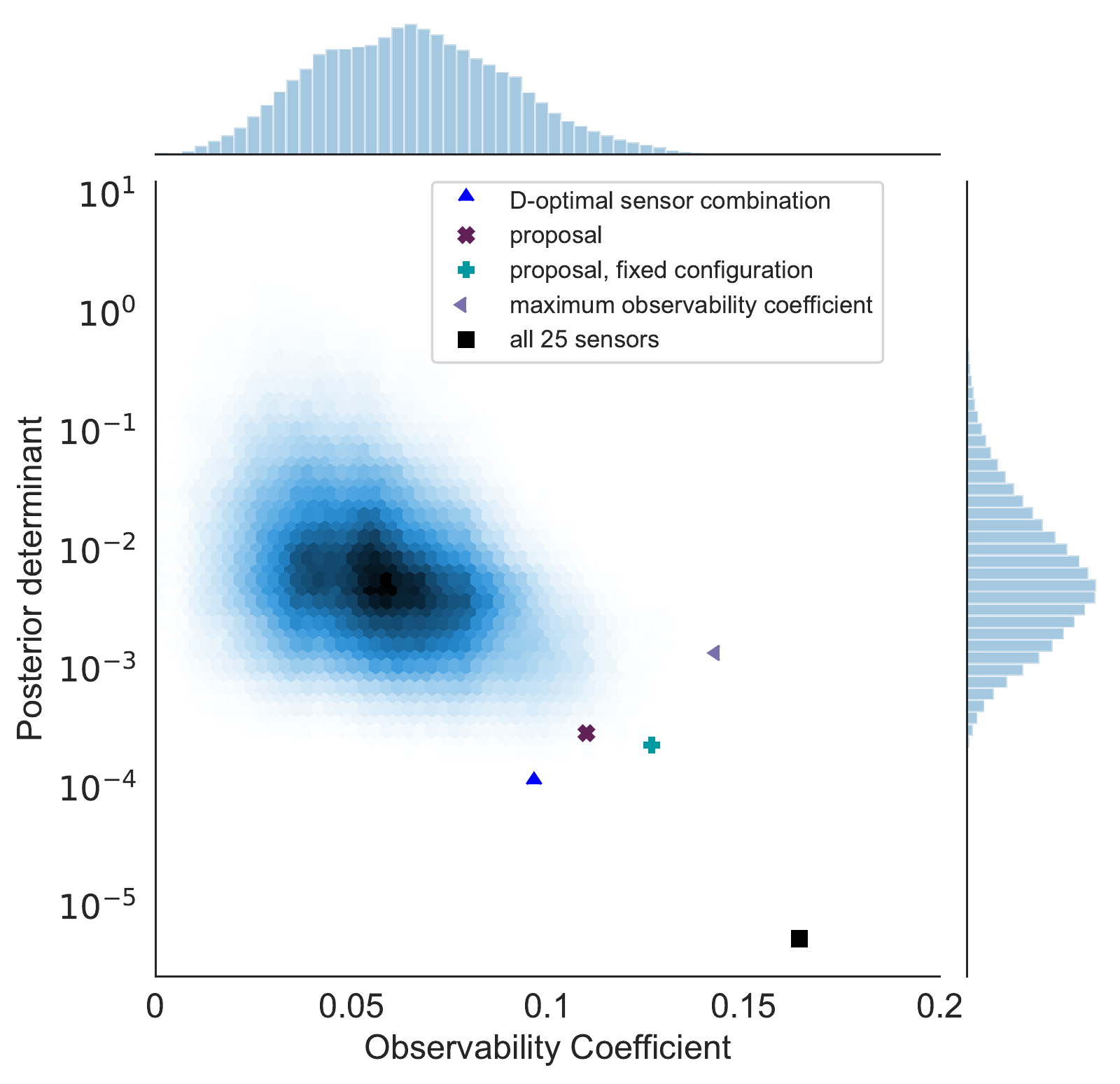} 
\end{subfigure}
\begin{subfigure}{0.45\textwidth}
\centering
~\\[1.3cm]
\begin{tabular}{r|l}
best~~~~~ & sensor selection \\
\noalign{\smallskip}\hline\noalign{\smallskip}
0.252 \% & proposal \\
0.081 \% & proposal, fixed config. \\
12.923 \% & max. observability
\end{tabular}
~\\[0.5cm]
\includegraphics[width = \textwidth]{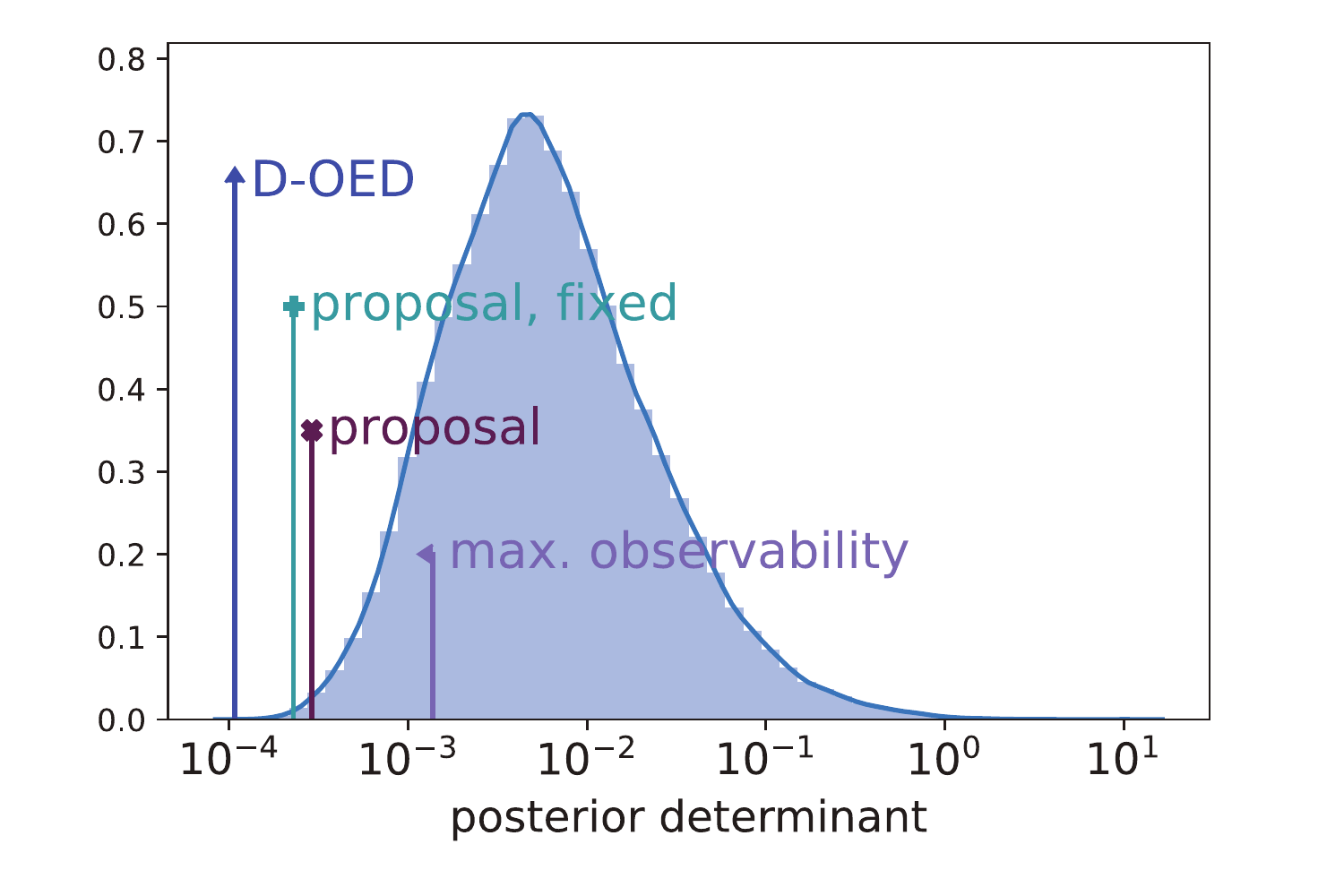} 
\end{subfigure}
\caption[Distribution of the determinant of the posterior covariance matrix over 1,081,575 combinations.]{Distribution of the posterior determinant $\det(\mypostCov)$ for $\mypara = \mypararef$. See Figure \ref{fig:Bayes:ex1:distributionplots:A} for details about the plot structure.}
\label{fig:Bayes:ex1:distributionplots:D}
\end{figure}

\begin{figure}
\begin{subfigure}{0.54\textwidth}
\includegraphics[width = \textwidth]{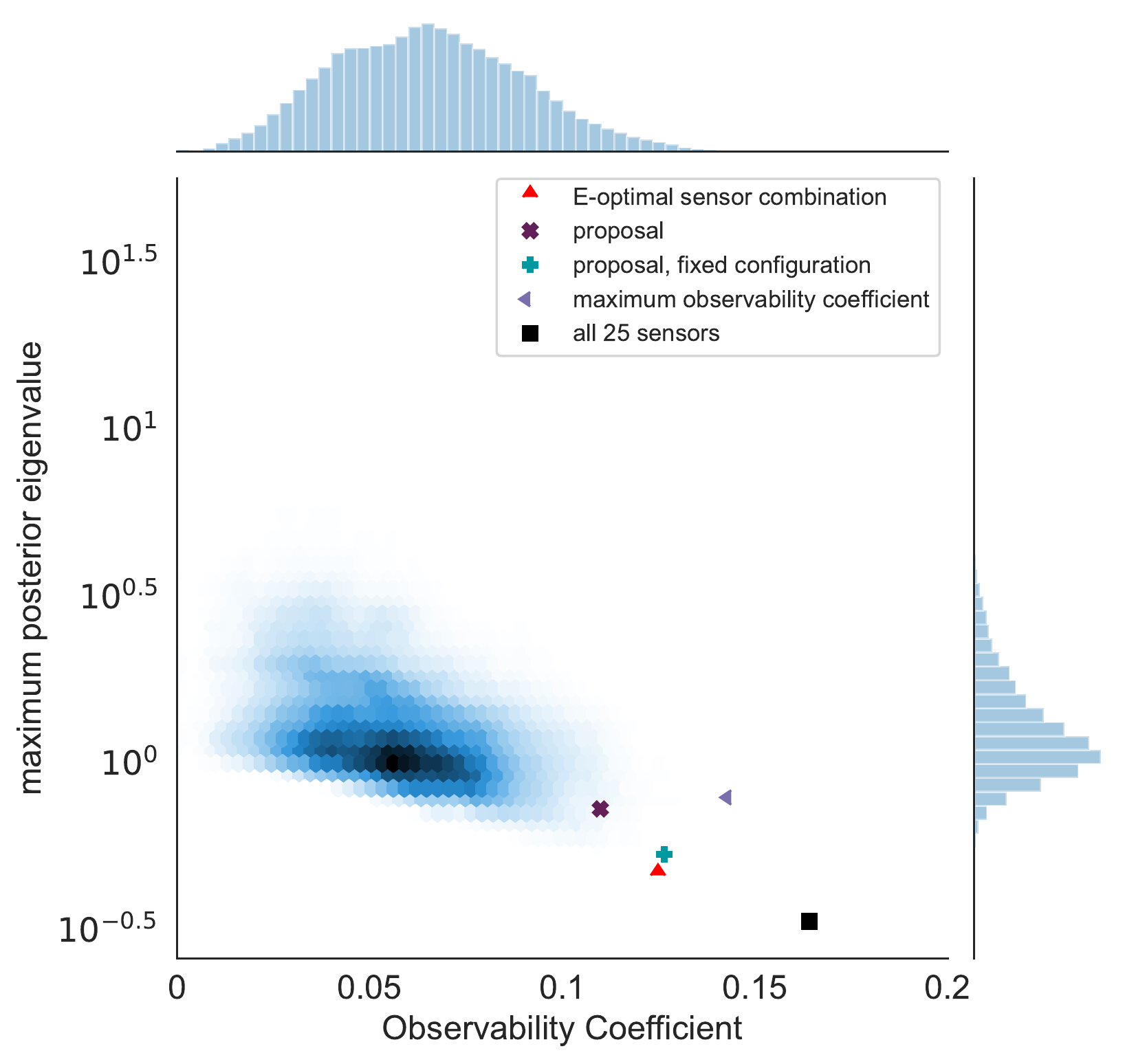} 
\end{subfigure}
\begin{subfigure}{0.45\textwidth}
\centering
~\\[1.3cm]
\begin{tabular}{l|l}
best~~~ & sensor selection \\
\noalign{\smallskip}\hline\noalign{\smallskip}
1.679 \% & proposal \\
0.001 \% & proposal, fixed config. \\
4.080 \% & max. observability
\end{tabular}
~\\[0.5cm]
\includegraphics[width = \textwidth]{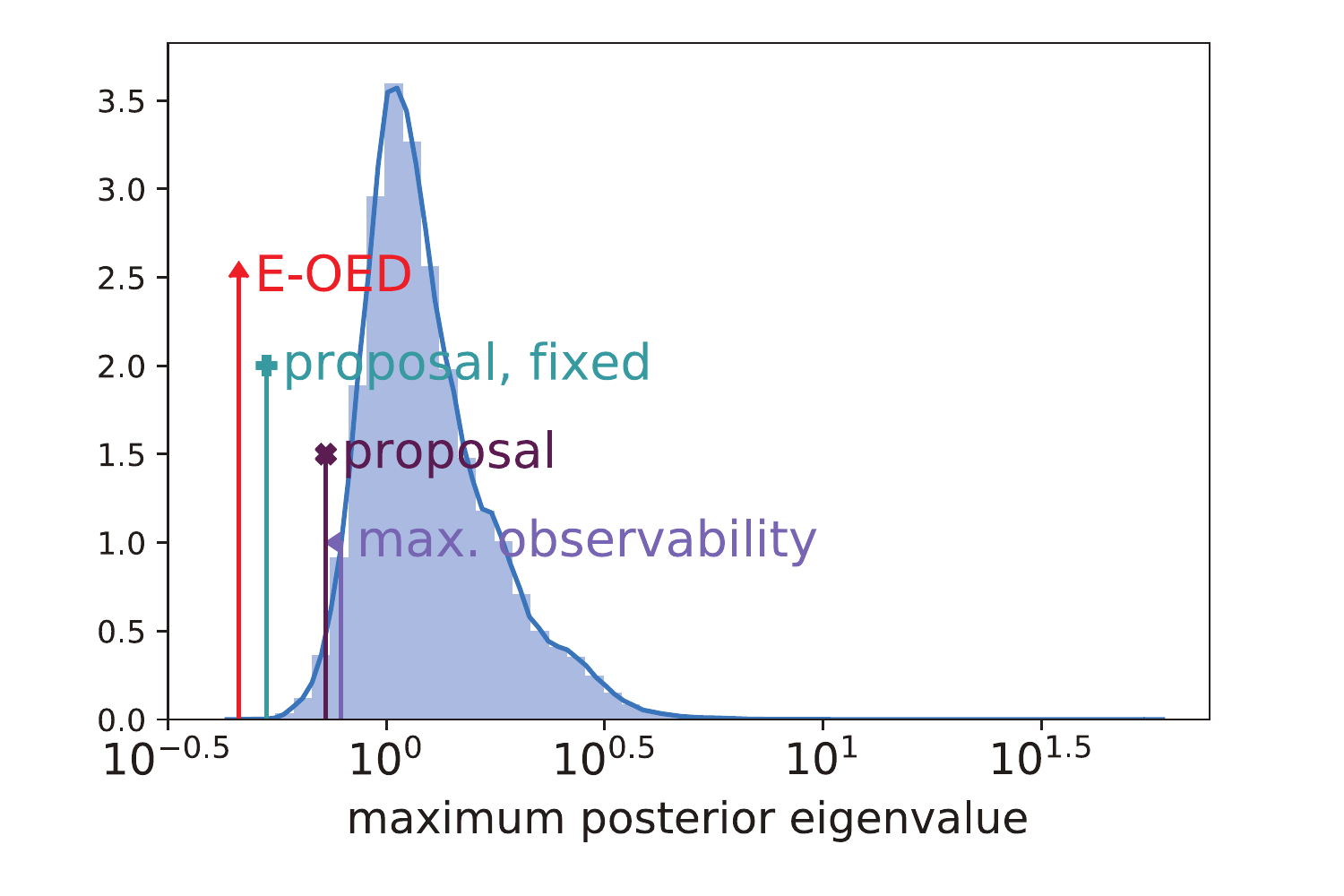} 
\end{subfigure}
\caption[Distribution of the maximum eigenvalue of the posterior covariance matrix over 1,081,575 combinations.]{Distribution of the maximum eigenvalue of the posterior covariance matrix $\mypostCov$ for $\mypara = \mypararef$. See Figure \ref{fig:Bayes:ex1:distributionplots:A} for details about the plot structure.
Note that the $\mypararef$-trained sensor combination has the 101-st smallest maximum posterior eigenvalue among all 1,081,575 possibilities.
}
\label{fig:Bayes:ex1:distributionplots:E}
\end{figure}

In order to visualize the connection between the observability coefficient $\myparatoobsinfsup[\mypararef]$ and the classic A-, D-, and E-\ac{oed} criteria, we plot the distribution of the posterior covariance matrix's trace, determinant, and maximum eigenvalue over all sensor combinations against $\myparatoobsinfsup$ in Figures \ref{fig:Bayes:ex1:distributionplots:A}, \ref{fig:Bayes:ex1:distributionplots:D}, \ref{fig:Bayes:ex1:distributionplots:E}.
Overall we observe a strong correlation between the respective \ac{oed} criteria and $\myparatoobsinfsup[\mypararef]$: It is the most pronounced in Figure \ref{fig:Bayes:ex1:distributionplots:E} for E-optimality, and the least pronounced for D-optimality in Figure \ref{fig:Bayes:ex1:distributionplots:D}.
For all \ac{oed} criteria, the correlation becomes stronger for smaller scaling factors $\myreg$ and weakens for large $\myreg$ when the prior is prioritized (plots not shown).
This behavior aligns with the discussion in Section \ref{sec:Bayes:eigenvalues} that $\myparatoobsinfsup$ primarily targets the largest posterior eigenvalue and is most decisive for priors with higher uncertainty.

%% file: 643_seedComparison.tex
\subsubsection*{Comparison for different libraries}

We finally evaluate the influence of the library $\mylibrary_{5 \times 5}$ on our results.
To this end, we randomly select 200 sets of new measurement positions, each consisting of 25 drilling locations with an associated drilling depth.
For each library, we run Algorithm \ref{alg:Bayes:greedyOMP} to choose 8 sensors, once with $\myparadomaintrain$-training on the surrogate model, and once with the full-order model at $\mypararef$ only.
For comparison, we then consider in each library each possible combination of choosing 8 unordered sensor sets and compute the trace, determinant, and maximum eigenvalue of the associated posterior covariance matrix at the reference configuration $\mypararef$ together with its observability coefficient.
This lets us identify the A-, D-, and E-optimal sensor combinations.

\begin{figure}
    \centering
    \begin{subfigure}{0.46\textwidth}
    \includegraphics[width = \textwidth]{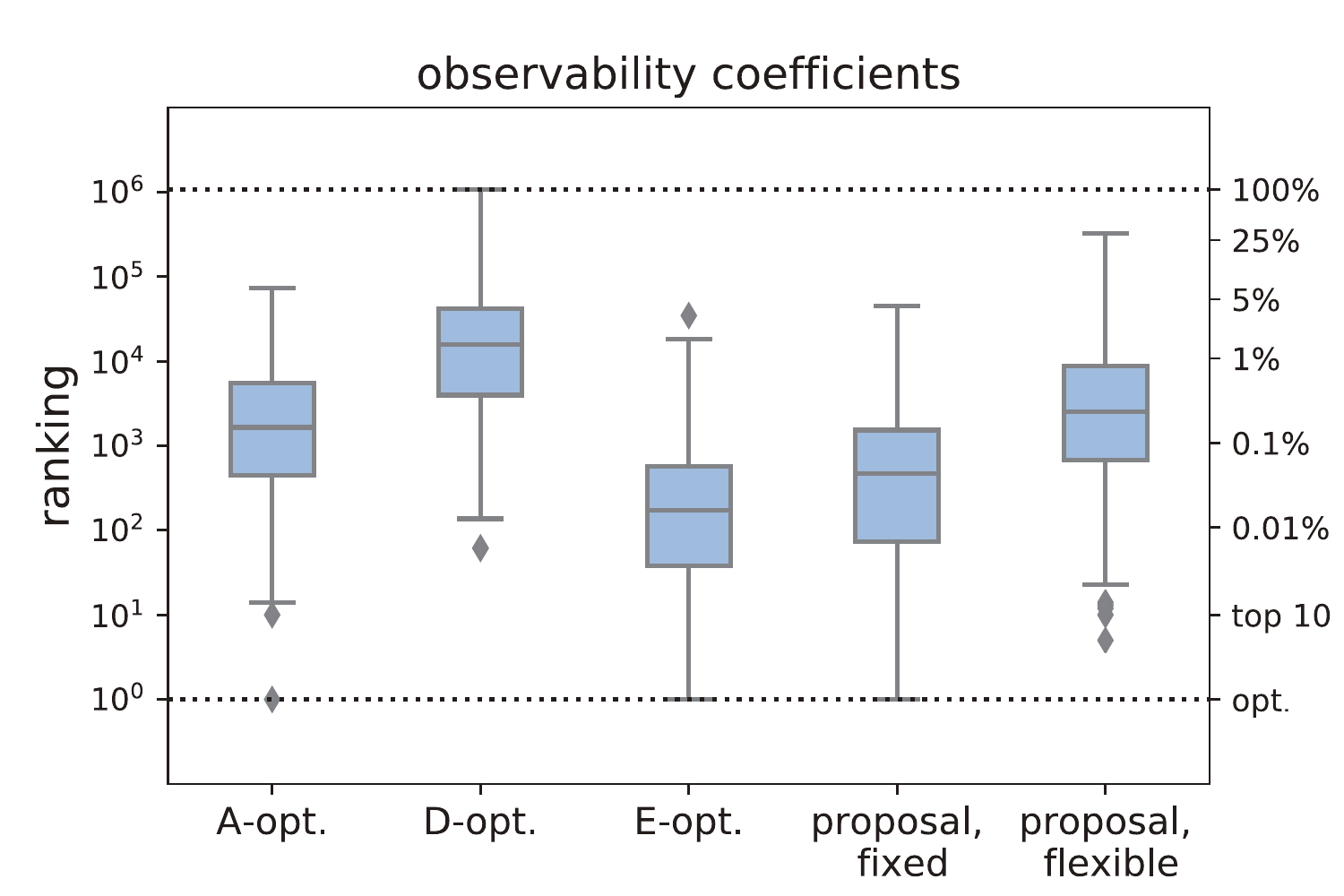} 
    \end{subfigure}
    \begin{subfigure}{0.53\textwidth}
    \begin{tabular}{l|rrr|rr}
    & \multicolumn{3}{c|}{design criterion} & \multicolumn{2}{c}{training} \\
pctl  & A-\ac{oed} & D-\ac{oed} & E-\ac{oed} & $\mypararef$~~ & $\myparadomaintrain$~ \\
\noalign{\smallskip}\hline \hline \noalign{\smallskip}
99-th & 3.5835 & 81.8508 & 1.1724 & 2.2223 & 9.2512 \\
95-th & 2.2747 & 26.8430 & 0.3601 & 0.7846 & 4.0374 \\
75-th & 0.5141 & 3.8600 & 0.0532 & 0.1419 & 0.8106 \\
50-th & 0.1527 & 1.4641 & 0.0159 & 0.0438 & 0.2354 \\
25-th & 0.0414 & 0.3669 & 0.0035 & 0.0068 & 0.0621 \\
\end{tabular}
    \end{subfigure}
    \caption{Ranking in $\myparatoobsinfsup[\mypararef]$ of the A-, D-, E- optimal and the $\mypararef$- and $\myparadomaintrain$-trained sensor choices for all possible combinations of choosing 8 unordered sensors in the library.
    Left: Boxplots obtained over 200 random sensor libraries. Right: worst-case ranking (in percent) of the corresponding percentiles (``pctl'').}
    \label{fig:Bayes:seedComparison:OED}
\end{figure}

Figure \ref{fig:Bayes:seedComparison:OED} shows how $\myparatoobsinfsup[\mypararef]$ is distributed over the 200 libraries, with percentiles provided in the adjacent table.
For 75\% of the libraries, the A- and E-optimal, and the $\myparadomaintrain$- and $\mypararef$-trained sensor choices rank among the top 1\% of combinations with the largest observability.
Due to its non-optimized training for $\mypararef$, the $\myparadomaintrain$-trained sensor set performs slightly worse than what is achieved with $\mypararef$-training, but still yields a comparatively large value for $\myparatoobsinfsup[\mypararef]$.
In contrast, overall, the D-optimal sensor choices have smaller observability coefficients, presumably because the minimization of the posterior determinant is influenced less by the maximum posterior eigenvalue.

The ranking of the $\myparadomaintrain$- and $\mypararef$-trained sensor configurations in terms of the posterior covariance matrix's trace, determinant, and maximum eigenvalue over the 200 libraries is given in Figure \ref{fig:Bayes:seedComparison:Both}.
Both perform well and lie for 75\% of the libraries within the top 1\% of combinations.
As the ranking is performed for the configuration parameter $\mypararef$, the $\mypararef$-trained sensor combination performs better, remaining in 95\% of the libraries within the top 5\% of sensor combinations.

\begin{figure}
\centering
\begin{minipage}{0.5\textwidth}
\includegraphics[width = \textwidth]{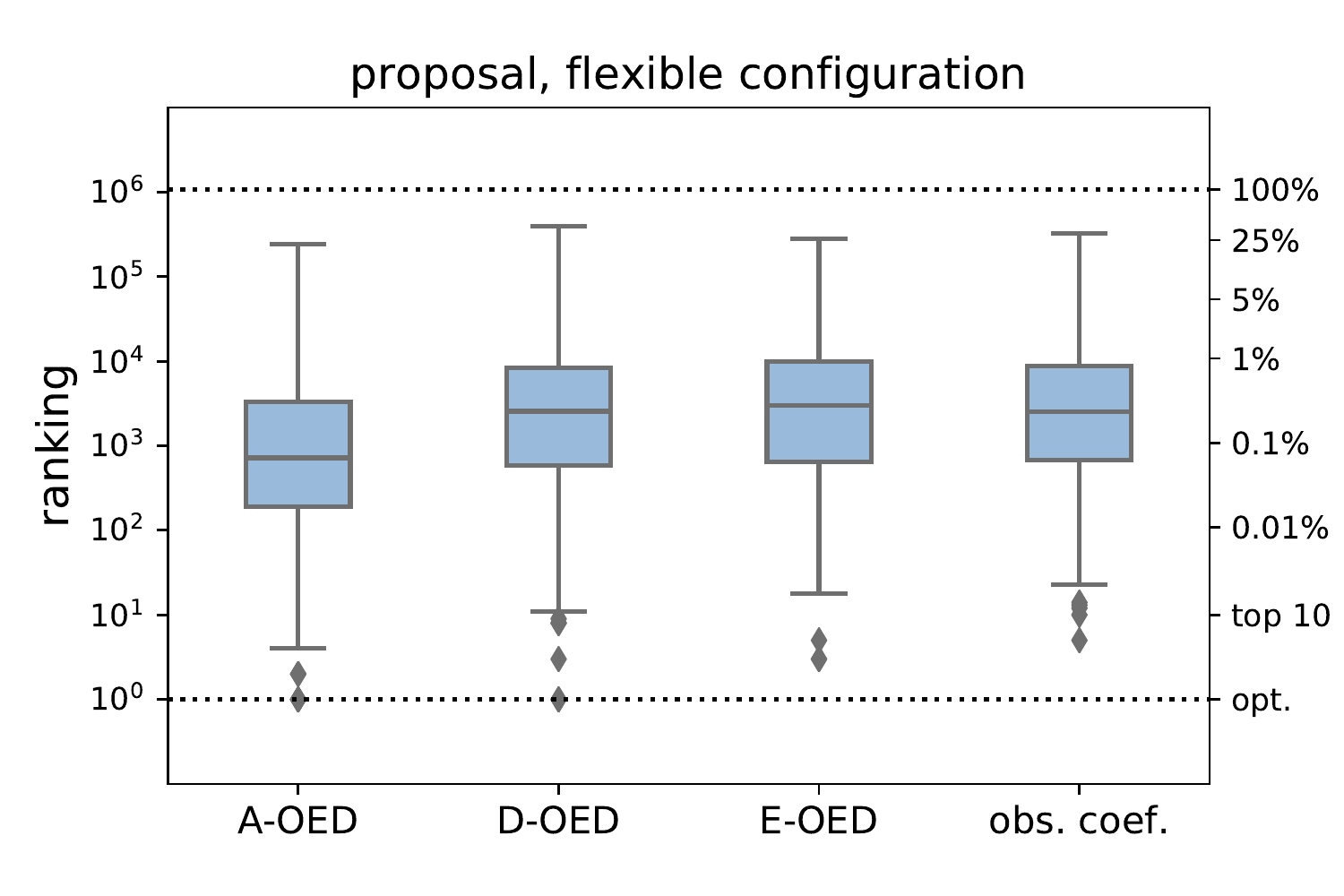} 
\end{minipage}
\begin{minipage}{0.49\textwidth}\centering
\begin{tabular}{l|rrrr}
pctl  & A-\ac{oed} & D-\ac{oed} & E-\ac{oed} & $\myparatoobsinfsup[\mypararef]$ \\
\noalign{\smallskip}\hline \hline \noalign{\smallskip}
99-th & 3.9240 & 6.2372 & 10.8391 & 9.2512 \\
95-th & 1.9093 & 3.1544 & 4.5583 & 4.0374 \\
75-th & 0.3083 & 0.7718 & 0.9185 & 0.8106 \\
50-th & 0.0664 & 0.2361 & 0.2763 & 0.2354 \\
25-th & 0.0177 & 0.0536 & 0.0596 & 0.0621 \\
\end{tabular}
\end{minipage} \\
\begin{minipage}{0.5\textwidth}
\includegraphics[width = \textwidth]{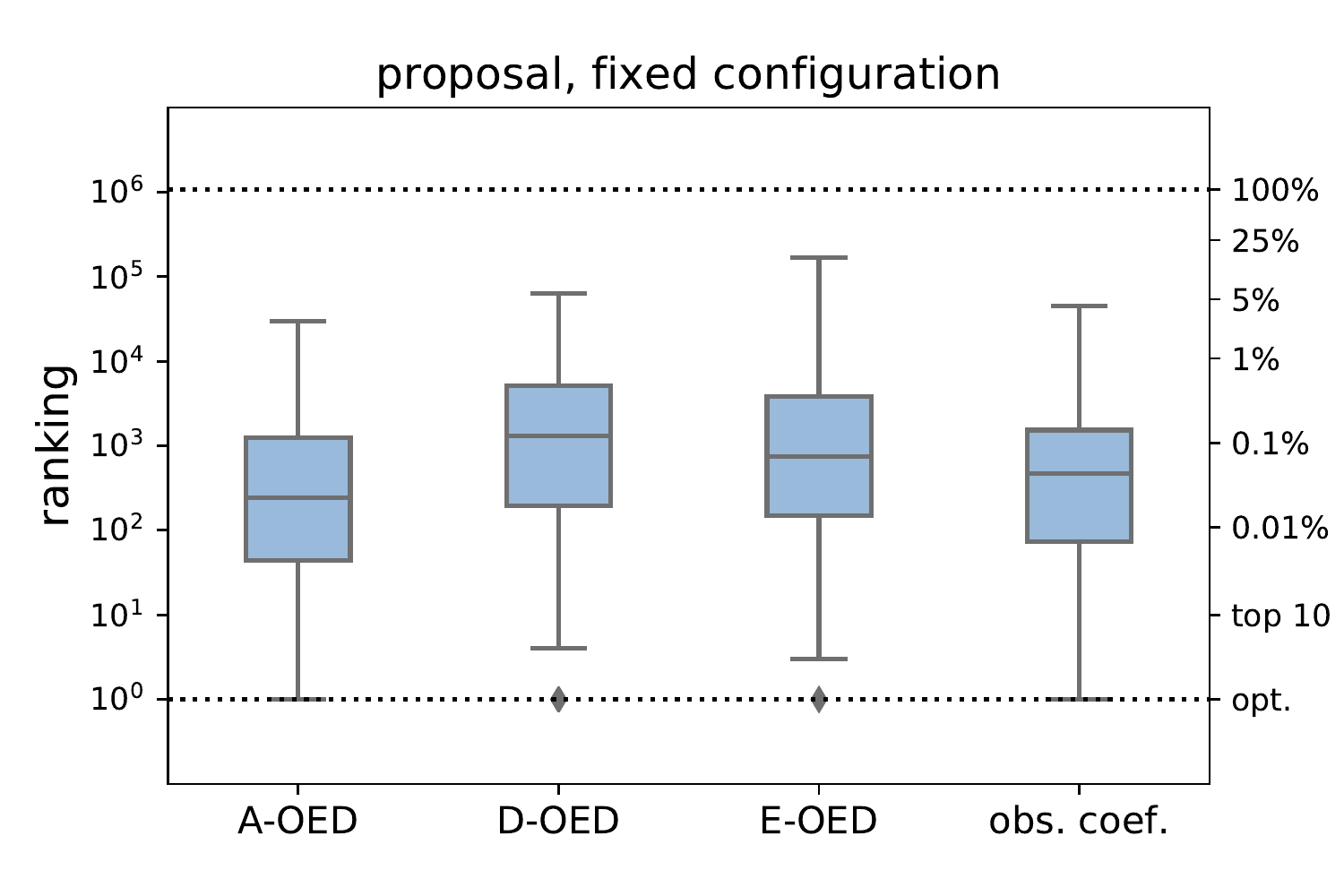} 
\end{minipage}
\begin{minipage}{0.49\textwidth}
\centering
\begin{tabular}{l|rrrr}
pctl & A-\ac{oed} & D-\ac{oed} & E-\ac{oed} & $\myparatoobsinfsup[\mypararef]$ \\
\noalign{\smallskip}\hline \hline \noalign{\smallskip}
99-th & 2.5261 & 2.9752 & 11.1534 & 2.2223 \\
95-th & 1.0134 & 1.8324 & 2.8458 & 0.7846 \\
75-th & 0.1155 & 0.4698 & 0.3549 & 0.1419 \\
50-th & 0.0224 & 0.1212 & 0.0687 & 0.0438 \\
25-th & 0.0041 & 0.0181 & 0.0138 & 0.0068 \\
\end{tabular}
\end{minipage}
\caption
{Ranking of the posterior covariance matrix $\mypostCov[\mypararef,\myobsmap]$ in terms of the A-, D-, E-\ac{oed} criteria and the observability coefficient $\myparatoobsinfsup[\mypararef]$ when the observation operator $\myparatoobsmap$ is chosen with Algorithm \ref{alg:Bayes:greedyOMP} and $\myparadomaintrain$-training (top) or $\mypararef$-training (bottom). 
The ranking is obtained by comparing all possible unordered combinations of 8 sensors in each sensor library.
On the left: Boxplots of the ranking over 200 sensor libraries; on the right: ranking (in percent) among different percentiles.
}
\label{fig:Bayes:seedComparison:Both}
\end{figure}

%% file: 643_unrestricted.tex
\subsection{Unrestricted Library}\label{sec:Bayes:results:unrestricted}

We next verify the scalability of Algorithm \ref{alg:Bayes:greedyOMP} to large sensor libraries by permitting all 2,209 drilling locations, at each of which at most one measurement may be taken at any of the 5 available measurement depths.
Choosing 10 unordered sensors yields approximately \texttt{7.29e+33} possible combinations.
Using the \ac{rb} surrogate model from before, we run Algorithm \ref{alg:Bayes:greedyOMP} once on a training grid $\myparadomaintrain \subset \myparadomain$ consisting of 10,000 randomly chosen configurations using only the surrogate model (runtime 14.19 s), and once on the reference configuration $\mypararef$ using the full-order model (runtime 15.85 s) for comparison.
We terminate the algorithm whenever 10 sensors are selected.
Compared to the training time on $\mylibrary_{5 \times 5}$ before, the results confirm that the size of the library itself has little influence on the overall runtime but that the full-order computations and the size of $\myparadomaintrain$ relative to the surrogate compute dominate.

\begin{figure}
    \centering
\begin{subfigure}{0.48\textwidth}
\includegraphics[width = \textwidth]{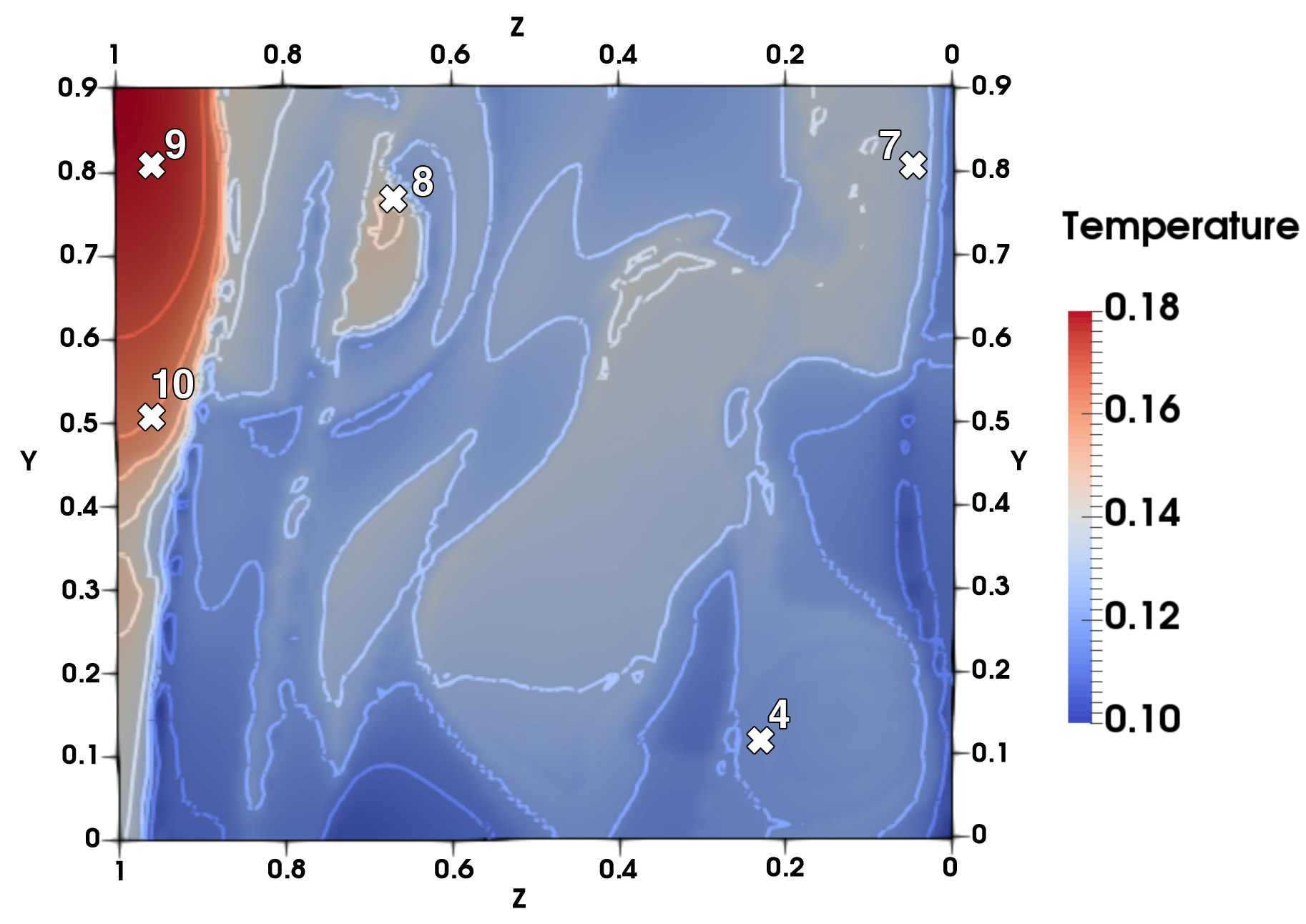} 
\caption{upmost layer, $\myparadomaintrain$-training}
\end{subfigure} 
\begin{subfigure}{0.48\textwidth}
\includegraphics[width = \textwidth]{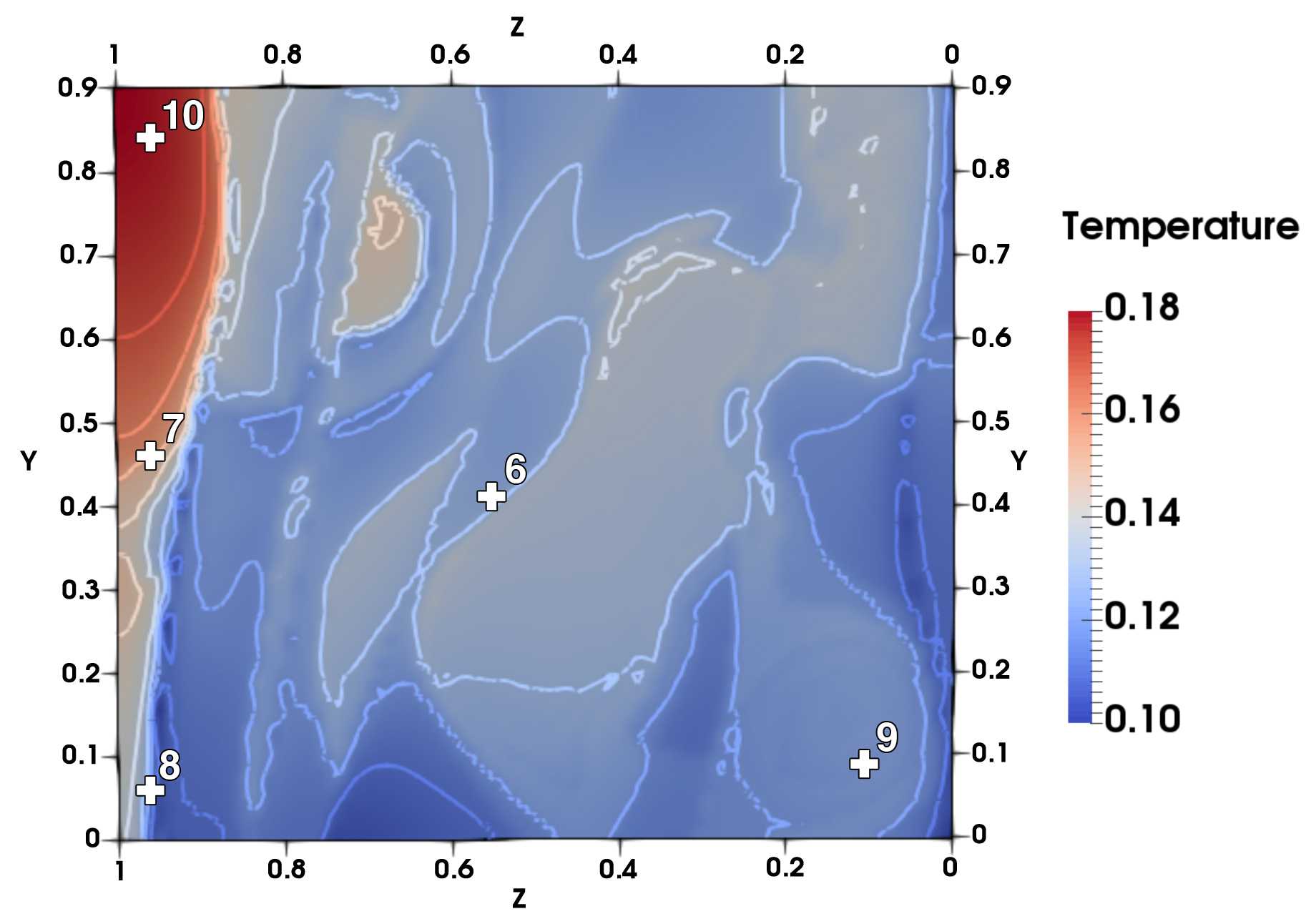} 
\caption{upmost layer, $\mypararef$-training}
\end{subfigure}
\newline
\begin{subfigure}{0.48\textwidth}
\includegraphics[width = \textwidth]{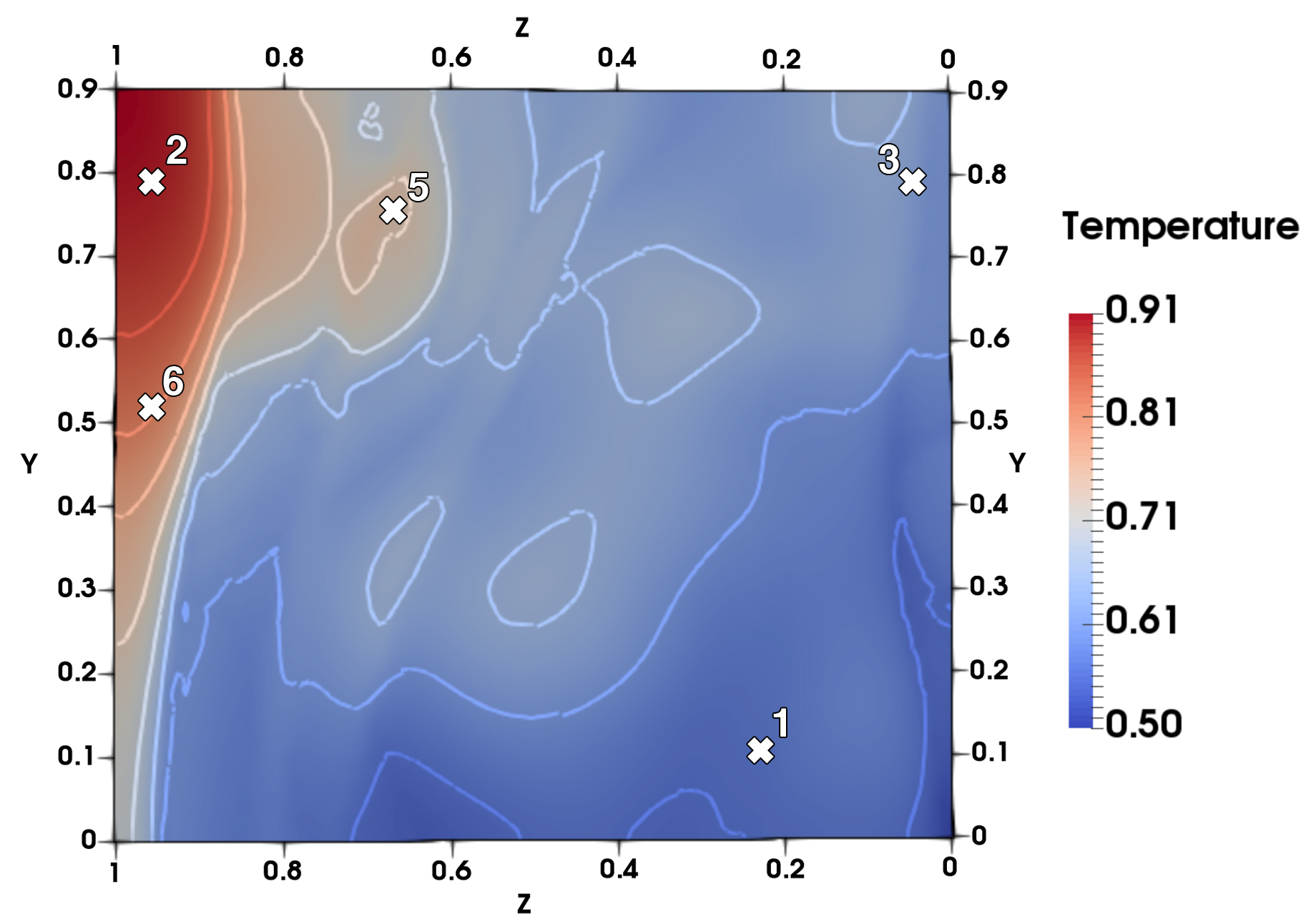} 
\caption{lowest layer, $\myparadomaintrain$-training}
\end{subfigure}
\begin{subfigure}{0.48\textwidth}
\includegraphics[width = \textwidth]{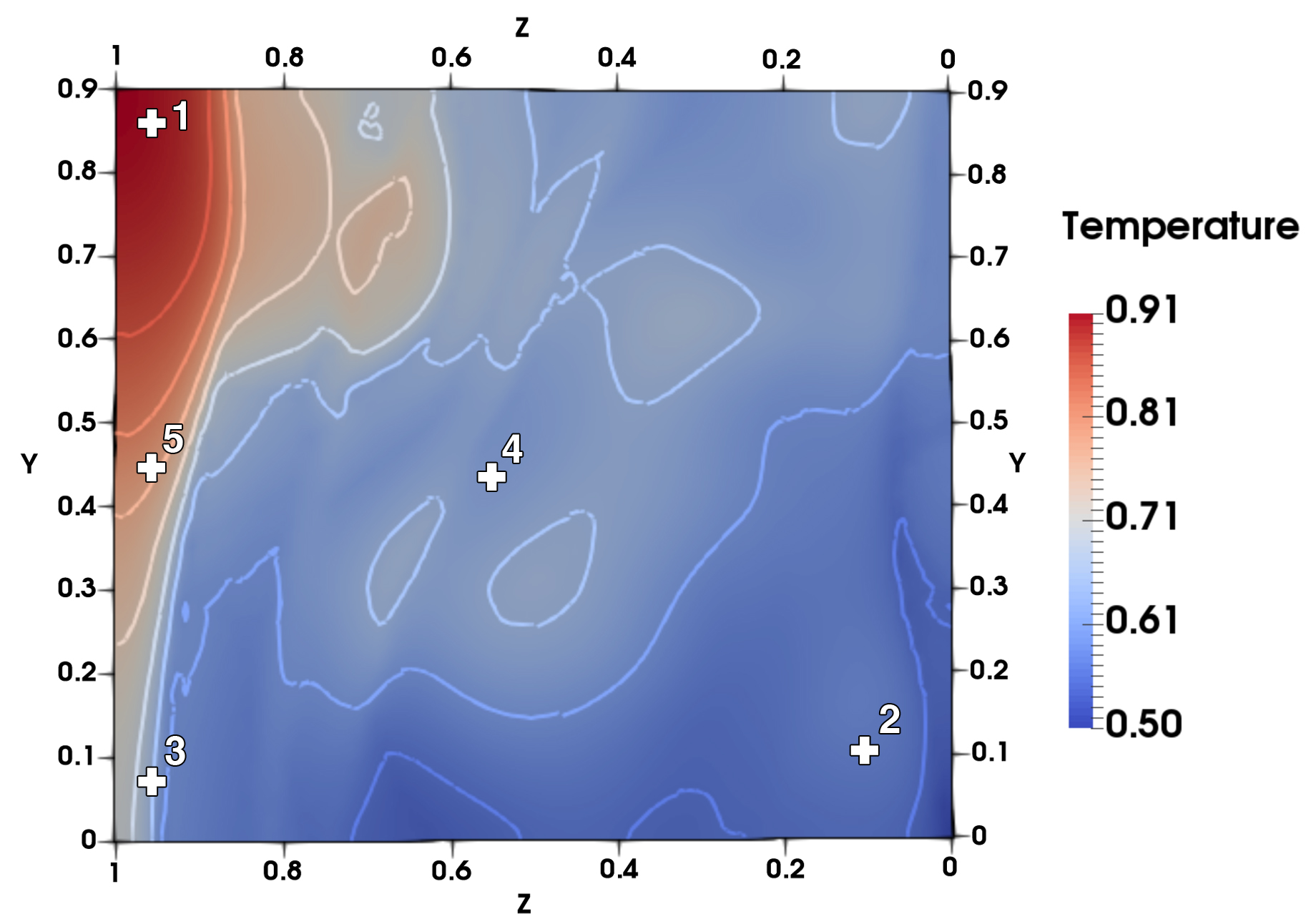} 
\caption{lowest layer, $\mypararef$-training}
\end{subfigure}
    \caption{Sensor positions chosen by Algorithm \ref{alg:Bayes:greedyOMP} from a grid of $47 \times 47$ available horizontal positions with available 5 depths each, though only the lowest (bottom) and upmost (top) layers were chosen.
    The underlying plot shows cuts through the full-order solution $\mystate$ at $\mypara = \mypararef$.
Left: $\myparadomaintrain$-training with the \ac{rb} surrogate model on a training set $\myparadomaintrain \subset \myparadomain$ with 10,000 random configurations; runtime 14.19 s for 10 sensors.
Right: $\mypararef$-training with full-order model at reference parameter; runtime 15.85 s for 10 sensors.
}\label{fig:Bayes:unrestricted:positions}
\end{figure}

The sensors chosen by the two runs of Algorithm \ref{alg:Bayes:greedyOMP} are shown in Figure \ref{fig:Bayes:unrestricted:positions}.
They share many structural similarities:
\begin{itemize}
\item \textbf{Depth:}
Despite the availability of 5 measurement depths, sensors have only been chosen on the lowest and the upmost layers with 5 sensors each.
The lower sensors were chosen first (with one exception, sensor 3 in $\mypararef$-training), presumably because the lower layer is closer to the uncertain Neumann boundary condition and therefore yields larger measurement values.
\item \textbf{Pairing}
Each sensor on the lowest layer has a counterpart on the upmost layer that has almost the same position on the horizontal plane.
This pairing targets noise sensitivity: With the prescribed error covariance function, the noise in two measurements is increasingly correlated the closer the measurements lie horizontally, independent of their depth coordinate.
Choosing a reference measurement near the zero-Dirichlet boundary at the surface helps filter out noise terms in the lower measurement.
\item \textbf{Organization}
On each layer, the sensors are spread out evenly and approximately aligned in 3 rows and 3 columns.
The alignment helps distinguish between the constant, linear, and quadratic parts of the uncertain Neumann flux function in north-south and east-west directions.
\end{itemize}

Figure \ref{fig:Bayes:unrestricted:both} (left side) shows the increase in the observability coefficients $\myparatoobsrbinfsup$ (for $\myparadomaintrain$-training) and $\myparatoobsinfsup[\mypararef]$ (for $\mypararef$-training) over the number of chosen sensors.
We again observe a strong initial incline followed by stagnation for the $\myparadomaintrain$-trained sensors, whereas the curve for $\mypararef$-training already starts at a large value to remain then almost constant.
The latter is explained by the positions of the first 5 sensors in Figure \ref{fig:Bayes:unrestricted:positions} (right), as they are already spaced apart in both directions for the identification of quadratic polynomials.
In contrast, for $\myparadomaintrain$-training, the ``3 rows, 3 columns'' structure is only completed after the sixth sensor (c.f. Figure \ref{fig:Bayes:unrestricted:positions}, left).
With 6 sensors, the observability coefficients in both training schemes have already surpassed the final observability coefficients with 8 sensors in the previous training on the smaller library $\mylibrary_{5 \times 5}$.
The final observability coefficients at the reference parameter $\mypararef$ are $\myparatoobsinfsup[\mypararef] = 0.4042$ for $\mypararef$-training, and $\myparatoobsinfsup[\mypararef] = 0.3595$ for $\myparadomaintrain$-training.

\begin{figure}
    \centering
    \begin{subfigure}{0.49\textwidth}
    \centering
    \includegraphics[width = \textwidth]{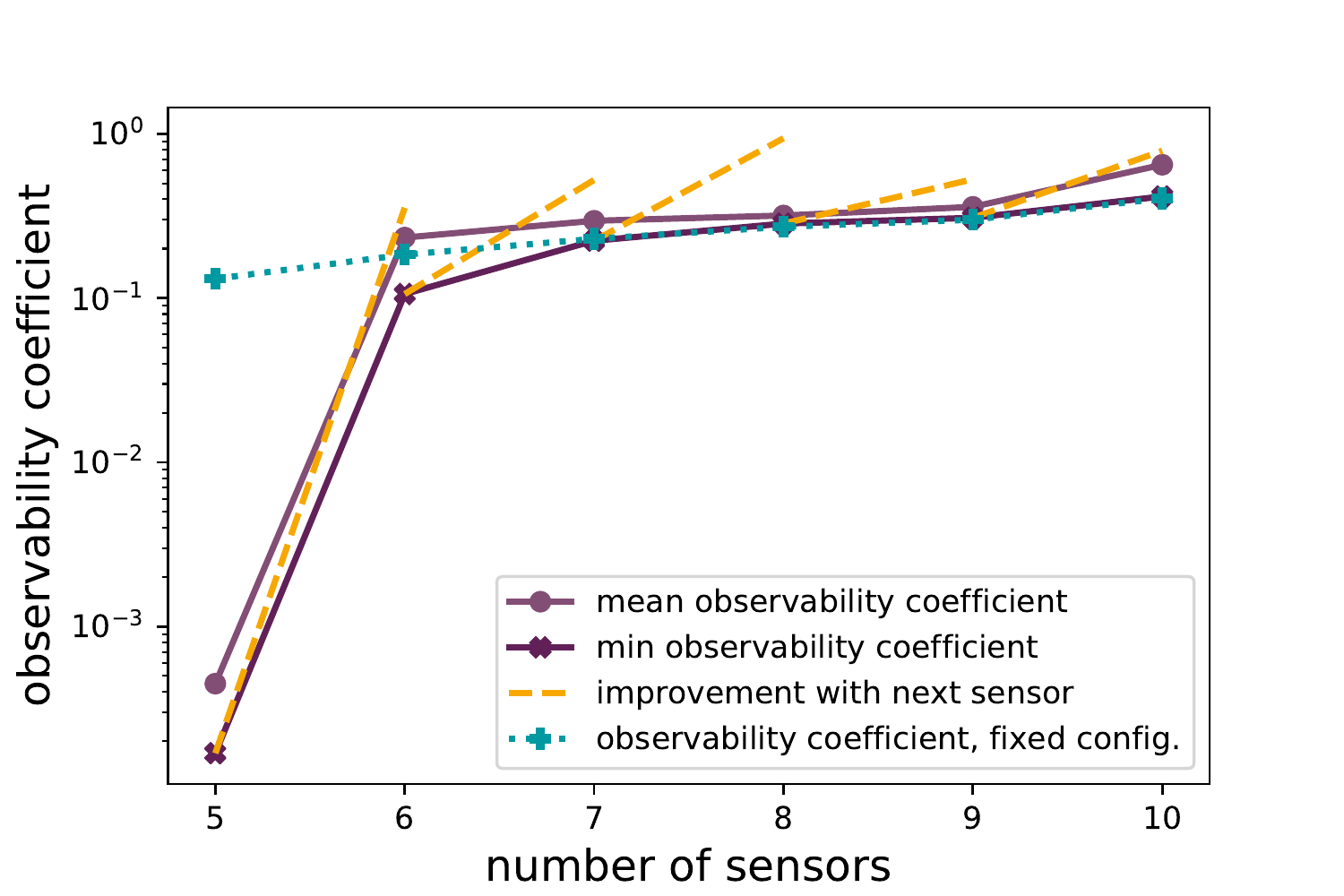} 
    \end{subfigure}
    \begin{subfigure}{0.49\textwidth}
    \centering
    \includegraphics[width = \textwidth]{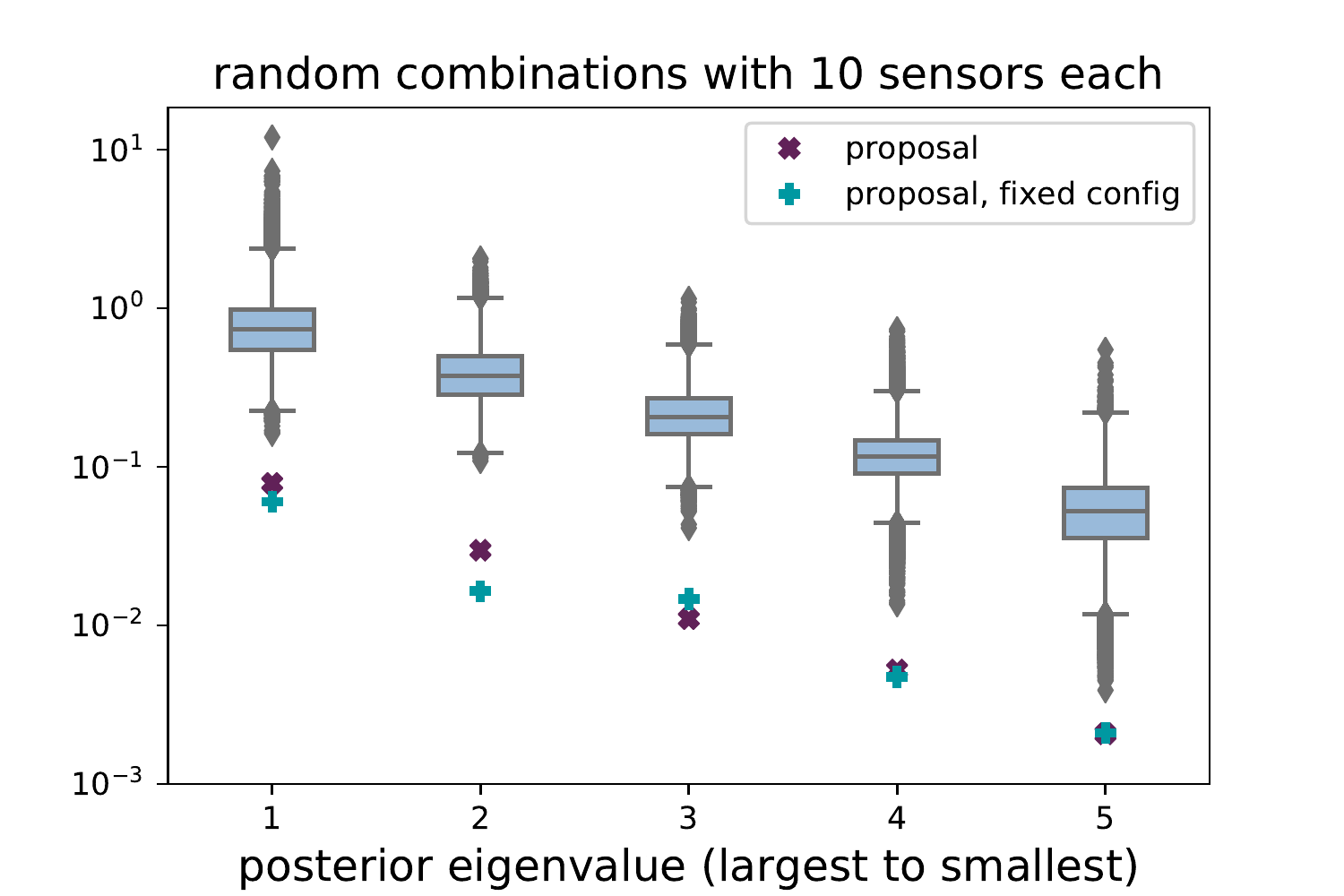} 
    \end{subfigure}
    \caption{Left:
    Observabity coefficients during sensor selection with $\myparadomaintrain$- and $\mypararef$-training for a library with 11,045 measurement positions and combinatorial restrictions.
Shown are 1) the minimum and mean surrogate observability coefficient $\myparatoobsrbinfsup$ over a training set with 10,000 random configurations with final values $\min_{\mypara} \myparatoobsrbinfsup = 0.4160$ and $\text{mean}_{\mypara} \myparatoobsrbinfsup = 0.6488$, and 2) the full-order observability coefficient $\myparatoobsinfsup[\mypararef]$ when training on the reference parameter $\mypararef$ alone (final value $\myparatoobsinfsup[\mypararef] = 0.4042$).
Right:
Boxplots for the 5 eigenvalues of the posterior covariance matrix $\mypostCov$ over 50,000 sets of 10 sensors chosen uniformly from a $5 \times 47 \times 47$ grid with imposed combinatorial restrictions.
The eigenvalues are compared according to their order from largest to smallest.
Indicated are also the eigenvalues for the $\myparadomaintrain$-trained (purple, ``x''-marker) and $\mypararef$-trained (turquoise, ``+''-marker) sensors from Figure \ref{fig:Bayes:unrestricted:positions}.
}
    \label{fig:Bayes:unrestricted:both}
\end{figure}

As a final experiment, we compare the eigenvalues of the posterior covariance matrix $\mypostCov[\myobsmap, \mypararef]$ for the $\myparadomaintrain$- and $\mypararef$-trained sensors against 50,000 sets of 10 random sensors each.
We confirm that all 50,000 sensor combinations comply with the combinatorial restrictions.
Boxplots of the eigenvalues are provided in Figure \ref{fig:Bayes:unrestricted:both} (right side).
The eigenvalues of the posterior covariance matrix with sensors chosen by Algorithm \ref{alg:Bayes:greedyOMP} are smaller\footnote{Here we compare the largest eigenvalue of one matrix to the largest eigenvalue of another, the second largest to the second largest, and so on.} than all posterior eigenvalues for the random sensor combinations.

%% file: 650_conclusion.tex
\section{Conclusion}
\label{sec:Bayes:conclusion}

In this work, we analyzed the connection between the observation operator and the eigenvalues of the posterior covariance matrix in the inference of an uncertain parameter via Bayesian inversion for a linear, hyper-parameterized forward model.
We identified an observability coefficient whose maximization decreases the uncertainty in the posterior probability distribution for all hyper-parameters.
To this end, we proposed a sensor selection algorithm that expands an observation operator iteratively to guarantee a uniformly large observability coefficient for all hyper-parameters.
Computational feasibility is retained through a reduced-order model in the greedy step and an \ac{omp} search for the next sensor that only requires a single full-order model evaluation.
The validity of the approach was demonstrated on a large-scale heat conduction problem over a section of the Perth Basin in Western Australia.
Future extensions of this work are planned to address 1) high-dimensional parameter spaces through parameter reduction techniques, 2) the combination with the \ac{pbdw} \acs{infsup}-criterion to inform sensors by functionalanalytic means in addition to the noise covariance, and 3) the expansion to non-linear models through a Laplace approximation.

%% file: 000_main.bbl
\begin{thebibliography}{70}
\expandafter\ifx\csname natexlab\endcsname\relax\def\natexlab#1{#1}\fi
\providecommand{\url}[1]{\texttt{#1}}
\providecommand{\href}[2]{#2}
\providecommand{\path}[1]{#1}
\providecommand{\DOIprefix}{doi:}
\providecommand{\ArXivprefix}{arXiv:}
\providecommand{\URLprefix}{URL: }
\providecommand{\Pubmedprefix}{pmid:}
\providecommand{\doi}[1]{\href{http://dx.doi.org/#1}{\path{#1}}}
\providecommand{\Pubmed}[1]{\href{pmid:#1}{\path{#1}}}
\providecommand{\bibinfo}[2]{#2}
\ifx\xfnm\relax \def\xfnm[#1]{\unskip,\space#1}\fi
\bibitem[{Stuart(2010)}]{Stuart2010}
\bibinfo{author}{A.~M. Stuart},
\newblock \bibinfo{title}{{Inverse problems: a Bayesian perspective}},
\newblock \bibinfo{journal}{Acta numerica} \bibinfo{volume}{19}
  (\bibinfo{year}{2010}) \bibinfo{pages}{451--559}.
\bibitem[{Stober and Bucher(2012)}]{Stober2012}
\bibinfo{author}{I.~Stober}, \bibinfo{author}{K.~Bucher},
  \bibinfo{title}{{Geothermie}}, \bibinfo{publisher}{Springer},
  \bibinfo{year}{2012}.
\bibitem[{Ucinski(2004)}]{Ucinski2004}
\bibinfo{author}{D.~Ucinski}, \bibinfo{title}{{Optimal measurement methods for
  distributed parameter system identification}}, \bibinfo{publisher}{CRC
  press}, \bibinfo{year}{2004}.
\bibitem[{Melas(2006)}]{Melas2006}
\bibinfo{author}{V.~B. Melas}, \bibinfo{title}{{Functional approach to optimal
  experimental design}}, volume \bibinfo{volume}{184},
  \bibinfo{publisher}{Springer Science {\&} Business Media},
  \bibinfo{year}{2006}.
\bibitem[{Pronzato(2008)}]{Pronzato2008}
\bibinfo{author}{L.~Pronzato},
\newblock \bibinfo{title}{{Optimal experimental design and some related control
  problems}},
\newblock \bibinfo{journal}{Automatica} \bibinfo{volume}{44}
  (\bibinfo{year}{2008}) \bibinfo{pages}{303--325}.
\bibitem[{Aretz-Nellesen et~al.(2021)Aretz-Nellesen, Chen, Grepl, and
  Veroy}]{Aretz2020}
\bibinfo{author}{N.~Aretz-Nellesen}, \bibinfo{author}{P.~Chen},
  \bibinfo{author}{M.~A. Grepl}, \bibinfo{author}{K.~Veroy},
\newblock \bibinfo{title}{{A sequential sensor selection strategy for
  hyper-parameterized linear Bayesian inverse problems}},
\newblock in: \bibinfo{booktitle}{Numerical Mathematics and Advanced
  Applications ENUMATH 2019}, \bibinfo{publisher}{Springer},
  \bibinfo{year}{2021}, pp. \bibinfo{pages}{489--497}.
\bibitem[{Aretz-Nellesen et~al.(2019)Aretz-Nellesen, Grepl, and
  Veroy}]{Aretz-Nellesen2019a}
\bibinfo{author}{N.~Aretz-Nellesen}, \bibinfo{author}{M.~A.~M. Grepl},
  \bibinfo{author}{K.~Veroy},
\newblock \bibinfo{title}{{3D-VAR for parameterized partial differential
  equations: a certified reduced basis approach}},
\newblock \bibinfo{journal}{Advances in Computational Mathematics}
  \bibinfo{volume}{45} (\bibinfo{year}{2019}) \bibinfo{pages}{2369--2400}.
\bibitem[{Binev et~al.(2018)Binev, Cohen, Mula, and Nichols}]{Binev2018b}
\bibinfo{author}{P.~Binev}, \bibinfo{author}{A.~Cohen},
  \bibinfo{author}{O.~Mula}, \bibinfo{author}{J.~Nichols},
\newblock \bibinfo{title}{{Greedy algorithms for optimal measurements selection
  in state estimation using reduced models}},
\newblock \bibinfo{journal}{SIAM/ASA Journal on Uncertainty Quantification}
  \bibinfo{volume}{6} (\bibinfo{year}{2018}) \bibinfo{pages}{1101--1126}.
\bibitem[{Maday et~al.(2015)Maday, Patera, Penn, and Yano}]{Maday2015e}
\bibinfo{author}{Y.~Maday}, \bibinfo{author}{A.~T. Patera},
  \bibinfo{author}{J.~D. Penn}, \bibinfo{author}{M.~Yano},
\newblock \bibinfo{title}{{A parameterized‐background data‐weak approach to
  variational data assimilation: formulation, analysis, and application to
  acoustics}},
\newblock \bibinfo{journal}{International Journal for Numerical Methods in
  Engineering} \bibinfo{volume}{102} (\bibinfo{year}{2015})
  \bibinfo{pages}{933--965}.
\bibitem[{Barrault et~al.(2004)Barrault, Maday, Nguyen, and
  Patera}]{Barrault2004a}
\bibinfo{author}{M.~Barrault}, \bibinfo{author}{Y.~Maday},
  \bibinfo{author}{N.~C. Nguyen}, \bibinfo{author}{A.~T. Patera},
\newblock \bibinfo{title}{{An ‘empirical interpolation'method: application to
  efficient reduced-basis discretization of partial differential equations}},
\newblock \bibinfo{journal}{Comptes Rendus Mathematique} \bibinfo{volume}{339}
  (\bibinfo{year}{2004}) \bibinfo{pages}{667--672}.
\bibitem[{Maday and Mula(2013)}]{Maday2013a}
\bibinfo{author}{Y.~Maday}, \bibinfo{author}{O.~Mula},
\newblock \bibinfo{title}{{A generalized empirical interpolation method:
  application of reduced basis techniques to data assimilation}},
\newblock in: \bibinfo{booktitle}{Analysis and numerics of partial differential
  equations}, \bibinfo{publisher}{Springer}, \bibinfo{year}{2013}, pp.
  \bibinfo{pages}{221--235}.
\bibitem[{Alexanderian et~al.(2016)Alexanderian, Gloor, and
  Ghattas}]{Alexanderian2016}
\bibinfo{author}{A.~Alexanderian}, \bibinfo{author}{P.~J. Gloor},
  \bibinfo{author}{O.~Ghattas},
\newblock \bibinfo{title}{{On Bayesian A-and D-optimal experimental designs in
  infinite dimensions}},
\newblock \bibinfo{journal}{Bayesian Analysis} \bibinfo{volume}{11}
  (\bibinfo{year}{2016}) \bibinfo{pages}{671--695}.
\bibitem[{Alexanderian et~al.(2014)Alexanderian, Petra, Stadler, and
  Ghattas}]{Alexanderian2014a}
\bibinfo{author}{A.~Alexanderian}, \bibinfo{author}{N.~Petra},
  \bibinfo{author}{G.~Stadler}, \bibinfo{author}{O.~Ghattas},
\newblock \bibinfo{title}{{A-optimal design of experiments for
  infinite-dimensional Bayesian linear inverse problems with regularized
  $\backslash$ell{\_}0-sparsification}},
\newblock \bibinfo{journal}{SIAM Journal on Scientific Computing}
  \bibinfo{volume}{36} (\bibinfo{year}{2014}) \bibinfo{pages}{A2122--A2148}.
\bibitem[{Attia et~al.(2018)Attia, Alexanderian, and Saibaba}]{Attia2018}
\bibinfo{author}{A.~Attia}, \bibinfo{author}{A.~Alexanderian},
  \bibinfo{author}{A.~K. Saibaba},
\newblock \bibinfo{title}{{Goal-oriented optimal design of experiments for
  large-scale Bayesian linear inverse problems}},
\newblock \bibinfo{journal}{Inverse Problems} \bibinfo{volume}{34}
  (\bibinfo{year}{2018}) \bibinfo{pages}{095009}.
\bibitem[{Alexanderian and Saibaba(2018)}]{Alexanderian2018}
\bibinfo{author}{A.~Alexanderian}, \bibinfo{author}{A.~K. Saibaba},
\newblock \bibinfo{title}{{Efficient D-optimal design of experiments for
  infinite-dimensional Bayesian linear inverse problems}},
\newblock \bibinfo{journal}{SIAM Journal on Scientific Computing}
  \bibinfo{volume}{40} (\bibinfo{year}{2018}) \bibinfo{pages}{A2956--A2985}.
\bibitem[{Alexanderian et~al.(2021)Alexanderian, Petra, Stadler, and
  Sunseri}]{Alexanderian2021}
\bibinfo{author}{A.~Alexanderian}, \bibinfo{author}{N.~Petra},
  \bibinfo{author}{G.~Stadler}, \bibinfo{author}{I.~Sunseri},
\newblock \bibinfo{title}{{Optimal design of large-scale Bayesian linear
  inverse problems under reducible model uncertainty: good to know what you
  don't know}},
\newblock \bibinfo{journal}{SIAM/ASA Journal on Uncertainty Quantification}
  \bibinfo{volume}{9} (\bibinfo{year}{2021}) \bibinfo{pages}{163--184}.
\bibitem[{Wu et~al.(2022)Wu, Chen, and Ghattas}]{WuChenGhattas21}
\bibinfo{author}{K.~Wu}, \bibinfo{author}{P.~Chen},
  \bibinfo{author}{O.~Ghattas},
\newblock \bibinfo{title}{An efficient method for goal-oriented linear bayesian
  optimal experimental design: Application to optimal sensor placement},
\newblock \bibinfo{journal}{arXiv preprint arXiv:2102.06627, to appear in
  SIAM/AMS Journal on Uncertainty Quantification}  (\bibinfo{year}{2022}).
\bibitem[{Alexanderian(2021)}]{Alexanderian2021a}
\bibinfo{author}{A.~Alexanderian},
\newblock \bibinfo{title}{{Optimal experimental design for infinite-dimensional
  Bayesian inverse problems governed by PDEs: a review}},
\newblock \bibinfo{journal}{Inverse Problems}  (\bibinfo{year}{2021}).
\bibitem[{Alexanderian et~al.(2016)Alexanderian, Petra, Stadler, and
  Ghattas}]{Alexanderian2016a}
\bibinfo{author}{A.~Alexanderian}, \bibinfo{author}{N.~Petra},
  \bibinfo{author}{G.~Stadler}, \bibinfo{author}{O.~Ghattas},
\newblock \bibinfo{title}{{A fast and scalable method for A-optimal design of
  experiments for infinite-dimensional Bayesian nonlinear inverse problems}},
\newblock \bibinfo{journal}{SIAM Journal on Scientific Computing}
  \bibinfo{volume}{38} (\bibinfo{year}{2016}) \bibinfo{pages}{A243--A272}.
\bibitem[{Huan and Marzouk(2013)}]{Huan2013}
\bibinfo{author}{X.~Huan}, \bibinfo{author}{Y.~M. Marzouk},
\newblock \bibinfo{title}{{Simulation-based optimal Bayesian experimental
  design for nonlinear systems}},
\newblock \bibinfo{journal}{Journal of Computational Physics}
  \bibinfo{volume}{232} (\bibinfo{year}{2013}) \bibinfo{pages}{288--317}.
\bibitem[{Wu et~al.(2022{\natexlab{a}})Wu, Chen, and Ghattas}]{WuChenGhattas20}
\bibinfo{author}{K.~Wu}, \bibinfo{author}{P.~Chen},
  \bibinfo{author}{O.~Ghattas},
\newblock \bibinfo{title}{A fast and scalable computational framework for
  large-scale and high-dimensional {B}ayesian optimal experimental design},
\newblock \bibinfo{journal}{arXiv preprint arXiv:2010.15196, to appear in SIAM
  Journal on Scientific Computing}  (\bibinfo{year}{2022}{\natexlab{a}}).
\bibitem[{Wu et~al.(2022{\natexlab{b}})Wu, O'Leary-Roseberry, Chen, and
  Ghattas}]{WuOLearyRoseberryChenEtAl22}
\bibinfo{author}{K.~Wu}, \bibinfo{author}{T.~O'Leary-Roseberry},
  \bibinfo{author}{P.~Chen}, \bibinfo{author}{O.~Ghattas},
\newblock \bibinfo{title}{Derivative-informed projected neural network for
  large-scale {B}ayesian optimal experimental design},
\newblock \bibinfo{journal}{arXiv preprint arXiv:2201.07925, to appear in
  Journal of Scientific Computing}  (\bibinfo{year}{2022}{\natexlab{b}}).
\bibitem[{Attia and Constantinescu(2020)}]{Attia2020}
\bibinfo{author}{A.~Attia}, \bibinfo{author}{E.~Constantinescu},
\newblock \bibinfo{title}{{Optimal experimental design for inverse problems in
  the presence of observation correlations}},
\newblock \bibinfo{journal}{arXiv preprint arXiv:2007.14476}
  (\bibinfo{year}{2020}).
\bibitem[{Bui-Thanh et~al.(2013)Bui-Thanh, Ghattas, Martin, and
  Stadler}]{Bui-ThanhGhattasMartinEtAl13}
\bibinfo{author}{T.~Bui-Thanh}, \bibinfo{author}{O.~Ghattas},
  \bibinfo{author}{J.~Martin}, \bibinfo{author}{G.~Stadler},
\newblock \bibinfo{title}{A computational framework for infinite-dimensional
  {B}ayesian inverse problems {P}art {I}: {T}he linearized case, with
  application to global seismic inversion},
\newblock \bibinfo{journal}{SIAM Journal on Scientific Computing}
  \bibinfo{volume}{35} (\bibinfo{year}{2013}) \bibinfo{pages}{A2494--A2523}.
\bibitem[{Cui et~al.(2016)Cui, Marzouk, and Willcox}]{Cui2016}
\bibinfo{author}{T.~Cui}, \bibinfo{author}{Y.~Marzouk},
  \bibinfo{author}{K.~Willcox},
\newblock \bibinfo{title}{{Scalable posterior approximations for large-scale
  Bayesian inverse problems via likelihood-informed parameter and state
  reduction}},
\newblock \bibinfo{journal}{Journal of Computational Physics}
  \bibinfo{volume}{315} (\bibinfo{year}{2016}) \bibinfo{pages}{363--387}.
\bibitem[{Parente et~al.(2020)Parente, Wallin, and Wohlmuth}]{Parente2020a}
\bibinfo{author}{M.~T. Parente}, \bibinfo{author}{J.~Wallin},
  \bibinfo{author}{B.~Wohlmuth},
\newblock \bibinfo{title}{{Generalized bounds for active subspaces}},
\newblock \bibinfo{journal}{Electronic Journal of Statistics}
  \bibinfo{volume}{14} (\bibinfo{year}{2020}) \bibinfo{pages}{917--943}.
\bibitem[{Lieberman et~al.(2010)Lieberman, Willcox, and
  Ghattas}]{Lieberman2010}
\bibinfo{author}{C.~Lieberman}, \bibinfo{author}{K.~Willcox},
  \bibinfo{author}{O.~Ghattas},
\newblock \bibinfo{title}{{Parameter and state model reduction for large-scale
  statistical inverse problems}},
\newblock \bibinfo{journal}{SIAM Journal on Scientific Computing}
  \bibinfo{volume}{32} (\bibinfo{year}{2010}) \bibinfo{pages}{2523--2542}.
\bibitem[{Chen and Ghattas(2019)}]{Chen2019a}
\bibinfo{author}{P.~Chen}, \bibinfo{author}{O.~Ghattas},
\newblock \bibinfo{title}{{Hessian-based sampling for high-dimensional model
  reduction}},
\newblock \bibinfo{journal}{International Journal for Uncertainty
  Quantification} \bibinfo{volume}{9} (\bibinfo{year}{2019}).
\bibitem[{Chen et~al.(2019)Chen, Wu, Chen, O'Leary-Roseberry, and
  Ghattas}]{ChenWuChenEtAl19b}
\bibinfo{author}{P.~Chen}, \bibinfo{author}{K.~Wu}, \bibinfo{author}{J.~Chen},
  \bibinfo{author}{T.~O'Leary-Roseberry}, \bibinfo{author}{O.~Ghattas},
\newblock \bibinfo{title}{Projected {S}tein variational {N}ewton: {A} fast and
  scalable {B}ayesian inference method in high dimensions},
\newblock \bibinfo{journal}{NeurIPS}  (\bibinfo{year}{2019}).
  \bibinfo{note}{Https://arxiv.org/abs/1901.08659}.
\bibitem[{Chen and Ghattas(2020)}]{ChenGhattas20}
\bibinfo{author}{P.~Chen}, \bibinfo{author}{O.~Ghattas},
\newblock \bibinfo{title}{Projected {S}tein variational gradient descent},
\newblock in: \bibinfo{booktitle}{Advances in Neural Information Processing
  Systems}, \bibinfo{year}{2020}.
\bibitem[{Zahm et~al.(2022)Zahm, Cui, Law, Spantini, and
  Marzouk}]{zahm2022certified}
\bibinfo{author}{O.~Zahm}, \bibinfo{author}{T.~Cui}, \bibinfo{author}{K.~Law},
  \bibinfo{author}{A.~Spantini}, \bibinfo{author}{Y.~Marzouk},
\newblock \bibinfo{title}{Certified dimension reduction in nonlinear {B}ayesian
  inverse problems},
\newblock \bibinfo{journal}{Mathematics of Computation} \bibinfo{volume}{91}
  (\bibinfo{year}{2022}) \bibinfo{pages}{1789--1835}.
\bibitem[{Qian et~al.(2017)Qian, Grepl, Veroy, and Willcox}]{Qian2017}
\bibinfo{author}{E.~Qian}, \bibinfo{author}{M.~Grepl},
  \bibinfo{author}{K.~Veroy}, \bibinfo{author}{K.~Willcox},
\newblock \bibinfo{title}{{A certified trust region reduced basis approach to
  PDE-constrained optimization}},
\newblock \bibinfo{journal}{SIAM Journal on Scientific Computing}
  \bibinfo{volume}{39} (\bibinfo{year}{2017}) \bibinfo{pages}{S434--S460}.
\bibitem[{Chen(2014)}]{Chen2014a}
\bibinfo{author}{P.~Chen}, \bibinfo{title}{{Model order reduction techniques
  for uncertainty quantification problems}}, \bibinfo{type}{Technical Report},
  \bibinfo{year}{2014}.
\bibitem[{Chen et~al.(2017)Chen, Quarteroni, and Rozza}]{Chen2017}
\bibinfo{author}{P.~Chen}, \bibinfo{author}{A.~Quarteroni},
  \bibinfo{author}{G.~Rozza},
\newblock \bibinfo{title}{{Reduced basis methods for uncertainty
  quantification}},
\newblock \bibinfo{journal}{SIAM/ASA Journal on Uncertainty Quantification}
  \bibinfo{volume}{5} (\bibinfo{year}{2017}) \bibinfo{pages}{813--869}.
\bibitem[{O’Leary-Roseberry et~al.(2022)O’Leary-Roseberry, Villa, Chen, and
  Ghattas}]{OLearyRoseberryVillaChenEtAl22}
\bibinfo{author}{T.~O’Leary-Roseberry}, \bibinfo{author}{U.~Villa},
  \bibinfo{author}{P.~Chen}, \bibinfo{author}{O.~Ghattas},
\newblock \bibinfo{title}{Derivative-informed projected neural networks for
  high-dimensional parametric maps governed by {PDE}s},
\newblock \bibinfo{journal}{Computer Methods in Applied Mechanics and
  Engineering} \bibinfo{volume}{388} (\bibinfo{year}{2022})
  \bibinfo{pages}{114199}.
\bibitem[{O'Leary-Roseberry et~al.(2022)O'Leary-Roseberry, Chen, Villa, and
  Ghattas}]{OLearyRoseberryChenVillaEtAl22}
\bibinfo{author}{T.~O'Leary-Roseberry}, \bibinfo{author}{P.~Chen},
  \bibinfo{author}{U.~Villa}, \bibinfo{author}{O.~Ghattas},
\newblock \bibinfo{title}{Derivate informed neural operator: {A}n efficient
  framework for high-dimensional parametric derivative learning},
\newblock \bibinfo{journal}{arXiv:2206.10745}  (\bibinfo{year}{2022}).
\bibitem[{{Da Prato}(2006)}]{DaPrato2006}
\bibinfo{author}{G.~{Da Prato}}, \bibinfo{title}{{An introduction to
  infinite-dimensional analysis}}, \bibinfo{publisher}{Springer Science {\&}
  Business Media}, \bibinfo{year}{2006}.
\bibitem[{Schwab and Stevenson(2009)}]{Schwab2009a}
\bibinfo{author}{C.~Schwab}, \bibinfo{author}{R.~Stevenson},
\newblock \bibinfo{title}{{Space-time adaptive wavelet methods for parabolic
  evolution problems}},
\newblock \bibinfo{journal}{Mathematics of Computation} \bibinfo{volume}{78}
  (\bibinfo{year}{2009}) \bibinfo{pages}{1293--1318}.
\bibitem[{Long et~al.(2013)Long, Scavino, Tempone, and Wang}]{Long2013}
\bibinfo{author}{Q.~Long}, \bibinfo{author}{M.~Scavino},
  \bibinfo{author}{R.~Tempone}, \bibinfo{author}{S.~Wang},
\newblock \bibinfo{title}{{Fast estimation of expected information gains for
  Bayesian experimental designs based on Laplace approximations}},
\newblock \bibinfo{journal}{Computer Methods in Applied Mechanics and
  Engineering} \bibinfo{volume}{259} (\bibinfo{year}{2013})
  \bibinfo{pages}{24--39}.
\bibitem[{Aretz(2022)}]{Aretz2022}
\bibinfo{author}{N.~Aretz}, \bibinfo{title}{{Data Assimilation and Sensor
  selection for Configurable Forward Models: Challenges and Opportunities for
  Model Order Reduction Methods}}, Ph.D. thesis, RWTH Aachen University,
  \bibinfo{year}{2022}.
\bibitem[{Cui et~al.(2015)Cui, Marzouk, and Willcox}]{Cui2015}
\bibinfo{author}{T.~Cui}, \bibinfo{author}{Y.~M. Marzouk},
  \bibinfo{author}{K.~E. Willcox},
\newblock \bibinfo{title}{{Data‐driven model reduction for the Bayesian
  solution of inverse problems}},
\newblock \bibinfo{journal}{International Journal for Numerical Methods in
  Engineering} \bibinfo{volume}{102} (\bibinfo{year}{2015})
  \bibinfo{pages}{966--990}.
\bibitem[{Bui-Thanh et~al.(2008)Bui-Thanh, Willcox, and
  Ghattas}]{Bui-Thanh2008}
\bibinfo{author}{T.~Bui-Thanh}, \bibinfo{author}{K.~Willcox},
  \bibinfo{author}{O.~Ghattas},
\newblock \bibinfo{title}{{Model reduction for large-scale systems with
  high-dimensional parametric input space}},
\newblock \bibinfo{journal}{SIAM Journal on Scientific Computing}
  \bibinfo{volume}{30} (\bibinfo{year}{2008}) \bibinfo{pages}{3270--3288}.
\bibitem[{Benner et~al.(2015)Benner, Gugercin, and Willcox}]{Benner2015}
\bibinfo{author}{P.~Benner}, \bibinfo{author}{S.~Gugercin},
  \bibinfo{author}{K.~Willcox},
\newblock \bibinfo{title}{{A survey of projection-based model reduction methods
  for parametric dynamical systems}},
\newblock \bibinfo{journal}{SIAM review} \bibinfo{volume}{57}
  (\bibinfo{year}{2015}) \bibinfo{pages}{483--531}.
\bibitem[{Schilders et~al.(2008)Schilders, {Van der Vorst}, and
  Rommes}]{Schilders2008}
\bibinfo{author}{W.~H.~A. Schilders}, \bibinfo{author}{H.~A. {Van der Vorst}},
  \bibinfo{author}{J.~Rommes}, \bibinfo{title}{{Model order reduction: theory,
  research aspects and applications}}, volume~\bibinfo{volume}{13},
  \bibinfo{publisher}{Springer}, \bibinfo{year}{2008}.
\bibitem[{Hesthaven et~al.(2016)Hesthaven, Rozza, and Stamm}]{Hesthaven2016}
\bibinfo{author}{J.~S. Hesthaven}, \bibinfo{author}{G.~Rozza},
  \bibinfo{author}{B.~Stamm}, \bibinfo{title}{{Certified reduced basis methods
  for parametrized partial differential equations}}, volume
  \bibinfo{volume}{590}, \bibinfo{publisher}{Springer}, \bibinfo{year}{2016}.
\bibitem[{Quarteroni et~al.(2015)Quarteroni, Manzoni, and
  Negri}]{Quarteroni2015}
\bibinfo{author}{A.~Quarteroni}, \bibinfo{author}{A.~Manzoni},
  \bibinfo{author}{F.~Negri}, \bibinfo{title}{{Reduced basis methods for
  partial differential equations: an introduction}},
  volume~\bibinfo{volume}{92}, \bibinfo{publisher}{Springer},
  \bibinfo{year}{2015}.
\bibitem[{Haasdonk(2017)}]{Haasdonk2017}
\bibinfo{author}{B.~Haasdonk},
\newblock \bibinfo{title}{{Reduced basis methods for parametrized PDEs–a
  tutorial introduction for stationary and instationary problems}},
\newblock \bibinfo{journal}{Model reduction and approximation: theory and
  algorithms} \bibinfo{volume}{15} (\bibinfo{year}{2017}) \bibinfo{pages}{65}.
\bibitem[{Golub and {Van Loan}(2013)}]{Golub2013}
\bibinfo{author}{G.~H. Golub}, \bibinfo{author}{C.~F. {Van Loan}},
  \bibinfo{title}{{Matrix computations}}, volume~\bibinfo{volume}{3},
  \bibinfo{publisher}{JHU press}, \bibinfo{year}{2013}.
\bibitem[{Binev et~al.(2011)Binev, Cohen, Dahmen, DeVore, Petrova, and
  Wojtaszczyk}]{Binev2011a}
\bibinfo{author}{P.~Binev}, \bibinfo{author}{A.~Cohen},
  \bibinfo{author}{W.~Dahmen}, \bibinfo{author}{R.~DeVore},
  \bibinfo{author}{G.~Petrova}, \bibinfo{author}{P.~Wojtaszczyk},
\newblock \bibinfo{title}{{Convergence rates for greedy algorithms in reduced
  basis methods}},
\newblock \bibinfo{journal}{SIAM journal on mathematical analysis}
  \bibinfo{volume}{43} (\bibinfo{year}{2011}) \bibinfo{pages}{1457--1472}.
\bibitem[{Cohen et~al.(2020)Cohen, Dahmen, DeVore, Fadili, Mula, and
  Nichols}]{Cohen2020}
\bibinfo{author}{A.~Cohen}, \bibinfo{author}{W.~Dahmen},
  \bibinfo{author}{R.~DeVore}, \bibinfo{author}{J.~Fadili},
  \bibinfo{author}{O.~Mula}, \bibinfo{author}{J.~Nichols},
\newblock \bibinfo{title}{{Optimal reduced model algorithms for data-based
  state estimation}},
\newblock \bibinfo{journal}{SIAM Journal on Numerical Analysis}
  \bibinfo{volume}{58} (\bibinfo{year}{2020}) \bibinfo{pages}{3355--3381}.
\bibitem[{Buffa et~al.(2012)Buffa, Maday, Patera, Prud'homme, and
  Turinici}]{Buffa2012}
\bibinfo{author}{A.~Buffa}, \bibinfo{author}{Y.~Maday}, \bibinfo{author}{A.~T.
  Patera}, \bibinfo{author}{C.~Prud'homme}, \bibinfo{author}{G.~Turinici},
\newblock \bibinfo{title}{{A priori convergence of the greedy algorithm for the
  parametrized reduced basis method}},
\newblock \bibinfo{journal}{ESAIM: Mathematical Modelling and Numerical
  Analysis-Mod{\'{e}}lisation Math{\'{e}}matique et Analyse Num{\'{e}}rique}
  \bibinfo{volume}{46} (\bibinfo{year}{2012}) \bibinfo{pages}{595--603}.
\bibitem[{Jagalur-Mohan and Marzouk(2021)}]{Jagalur-Mohan2021}
\bibinfo{author}{J.~Jagalur-Mohan}, \bibinfo{author}{Y.~M. Marzouk},
\newblock \bibinfo{title}{{Batch greedy maximization of non-submodular
  functions: Guarantees and applications to experimental design.}},
\newblock \bibinfo{journal}{J. Mach. Learn. Res.} \bibinfo{volume}{22}
  (\bibinfo{year}{2021}) \bibinfo{pages}{251--252}.
\bibitem[{Eftang et~al.(2010)Eftang, Patera, and R{\o}nquist}]{Eftang2010}
\bibinfo{author}{J.~L. Eftang}, \bibinfo{author}{A.~T. Patera},
  \bibinfo{author}{E.~M. R{\o}nquist},
\newblock \bibinfo{title}{{An" hp" certified reduced basis method for
  parametrized elliptic partial differential equations}},
\newblock \bibinfo{journal}{SIAM Journal on Scientific Computing}
  \bibinfo{volume}{32} (\bibinfo{year}{2010}) \bibinfo{pages}{3170--3200}.
\bibitem[{Eftang et~al.(2011)Eftang, Knezevic, and Patera}]{Eftang2011}
\bibinfo{author}{J.~L. Eftang}, \bibinfo{author}{D.~J. Knezevic},
  \bibinfo{author}{A.~T. Patera},
\newblock \bibinfo{title}{{An hp certified reduced basis method for
  parametrized parabolic partial differential equations}},
\newblock \bibinfo{journal}{Mathematical and Computer Modelling of Dynamical
  Systems} \bibinfo{volume}{17} (\bibinfo{year}{2011})
  \bibinfo{pages}{395--422}.
\bibitem[{Regenauer-Lieb and Horowitz(2007)}]{Regenauer-Lieb2007}
\bibinfo{author}{K.~Regenauer-Lieb}, \bibinfo{author}{F.~Horowitz},
\newblock \bibinfo{title}{{The Perth Basin geothermal opportunity}},
\newblock \bibinfo{journal}{Petroleum in Western Australia} \bibinfo{volume}{3}
  (\bibinfo{year}{2007}).
\bibitem[{Corbel et~al.(2012)Corbel, Schilling, Horowitz, Reid, Sheldon, Timms,
  and Wilkes}]{Corbel2012}
\bibinfo{author}{S.~Corbel}, \bibinfo{author}{O.~Schilling},
  \bibinfo{author}{F.~G. Horowitz}, \bibinfo{author}{L.~B. Reid},
  \bibinfo{author}{H.~A. Sheldon}, \bibinfo{author}{N.~E. Timms},
  \bibinfo{author}{P.~Wilkes},
\newblock \bibinfo{title}{{Identification and geothermal influence of faults in
  the Perth metropolitan area, Australia}},
\newblock in: \bibinfo{booktitle}{Thirty-seventh workshop on geothermal
  reservoir engineering, Stanford, CA}, \bibinfo{year}{2012}.
\bibitem[{Sheldon et~al.(2012)Sheldon, Florio, Trefry, Reid, Ricard, and
  Ghori}]{Sheldon2012}
\bibinfo{author}{H.~A. Sheldon}, \bibinfo{author}{B.~Florio},
  \bibinfo{author}{M.~G. Trefry}, \bibinfo{author}{L.~B. Reid},
  \bibinfo{author}{L.~P. Ricard}, \bibinfo{author}{K.~A.~R. Ghori},
\newblock \bibinfo{title}{{The potential for convection and implications for
  geothermal energy in the Perth Basin, Western Australia}},
\newblock \bibinfo{journal}{Hydrogeology Journal} \bibinfo{volume}{20}
  (\bibinfo{year}{2012}) \bibinfo{pages}{1251--1268}.
\bibitem[{Schilling et~al.(2013)Schilling, Sheldon, Reid, and
  Corbel}]{Schilling2013}
\bibinfo{author}{O.~Schilling}, \bibinfo{author}{H.~A. Sheldon},
  \bibinfo{author}{L.~B. Reid}, \bibinfo{author}{S.~Corbel},
\newblock \bibinfo{title}{{Hydrothermal models of the Perth metropolitan area,
  Western Australia: implications for geothermal energy}},
\newblock \bibinfo{journal}{Hydrogeology Journal} \bibinfo{volume}{21}
  (\bibinfo{year}{2013}) \bibinfo{pages}{605--621}.
\bibitem[{Pujol et~al.(2015)Pujol, Ricard, and Bolton}]{Pujol2015}
\bibinfo{author}{M.~Pujol}, \bibinfo{author}{L.~P. Ricard},
  \bibinfo{author}{G.~Bolton},
\newblock \bibinfo{title}{{20 years of exploitation of the Yarragadee aquifer
  in the Perth Basin of Western Australia for direct-use of geothermal heat}},
\newblock \bibinfo{journal}{Geothermics} \bibinfo{volume}{57}
  (\bibinfo{year}{2015}) \bibinfo{pages}{39--55}.
\bibitem[{Wellmann and Reid(2014)}]{Wellmann2014}
\bibinfo{author}{J.~F. Wellmann}, \bibinfo{author}{L.~B. Reid},
\newblock \bibinfo{title}{{Basin-scale geothermal model calibration: Experience
  from the Perth Basin, Australia}},
\newblock \bibinfo{journal}{Energy Procedia} \bibinfo{volume}{59}
  (\bibinfo{year}{2014}) \bibinfo{pages}{382--389}.
\bibitem[{Degen et~al.(2020)Degen, Veroy, and Wellmann}]{Degen2020b}
\bibinfo{author}{D.~Degen}, \bibinfo{author}{K.~Veroy},
  \bibinfo{author}{F.~Wellmann},
\newblock \bibinfo{title}{{Certified reduced basis method in geosciences}},
\newblock \bibinfo{journal}{Computational Geosciences} \bibinfo{volume}{24}
  (\bibinfo{year}{2020}) \bibinfo{pages}{241--259}.
\bibitem[{Degen(2020)}]{Degen2020Phd}
\bibinfo{author}{D.~M. Degen}, \bibinfo{title}{{Application of the reduced
  basis method in geophysical simulations: concepts, implementation,
  advantages, and limitations}}, \bibinfo{type}{Dissertation}, RWTH Aachen
  University, \bibinfo{year}{2020}. \DOIprefix\doi{10.18154/RWTH-2020-12042}.
\bibitem[{Bauer et~al.(2014)Bauer, Freeden, Jacobi, and Neu}]{Bauer2014}
\bibinfo{author}{M.~Bauer}, \bibinfo{author}{W.~Freeden},
  \bibinfo{author}{H.~Jacobi}, \bibinfo{author}{T.~Neu},
  \bibinfo{title}{{Handbuch Tiefe Geothermie}}, \bibinfo{publisher}{Springer},
  \bibinfo{year}{2014}.
\bibitem[{de~la Varga et~al.(2019)de~la Varga, Schaaf, and
  Wellmann}]{Varga2019}
\bibinfo{author}{M.~de~la Varga}, \bibinfo{author}{A.~Schaaf},
  \bibinfo{author}{F.~Wellmann},
\newblock \bibinfo{title}{{GemPy 1.0: open-source stochastic geological
  modeling and inversion}},
\newblock \bibinfo{journal}{Geoscientific Model Development}
  \bibinfo{volume}{12} (\bibinfo{year}{2019}) \bibinfo{pages}{1--32}.
\bibitem[{Permann et~al.(2020)Permann, Gaston, Andr{\v{s}}, Carlsen, Kong,
  Lindsay, Miller, Peterson, Slaughter, and Stogner}]{Permann2020}
\bibinfo{author}{C.~J. Permann}, \bibinfo{author}{D.~R. Gaston},
  \bibinfo{author}{D.~Andr{\v{s}}}, \bibinfo{author}{R.~W. Carlsen},
  \bibinfo{author}{F.~Kong}, \bibinfo{author}{A.~D. Lindsay},
  \bibinfo{author}{J.~M. Miller}, \bibinfo{author}{J.~W. Peterson},
  \bibinfo{author}{A.~E. Slaughter}, \bibinfo{author}{R.~H. Stogner},
\newblock \bibinfo{title}{{MOOSE: Enabling massively parallel multiphysics
  simulation}},
\newblock \bibinfo{journal}{SoftwareX} \bibinfo{volume}{11}
  (\bibinfo{year}{2020}) \bibinfo{pages}{100430}.
\bibitem[{Dahmen et~al.(2014)Dahmen, Plesken, and Welper}]{Dahmen2014}
\bibinfo{author}{W.~Dahmen}, \bibinfo{author}{C.~Plesken},
  \bibinfo{author}{G.~Welper},
\newblock \bibinfo{title}{{Double greedy algorithms: reduced basis methods for
  transport dominated problems}},
\newblock \bibinfo{journal}{ESAIM: Mathematical Modelling and Numerical
  Analysis-Mod{\'{e}}lisation Math{\'{e}}matique et Analyse Num{\'{e}}rique}
  \bibinfo{volume}{48} (\bibinfo{year}{2014}) \bibinfo{pages}{623--663}.
\bibitem[{Cressie(1990)}]{Cressie1990}
\bibinfo{author}{N.~Cressie},
\newblock \bibinfo{title}{{The origins of kriging}},
\newblock \bibinfo{journal}{Mathematical geology} \bibinfo{volume}{22}
  (\bibinfo{year}{1990}) \bibinfo{pages}{239--252}.
\bibitem[{Holgate and Gerner(2010)}]{Holgate2010}
\bibinfo{author}{F.~L. Holgate}, \bibinfo{author}{E.~J. Gerner},
\newblock \bibinfo{title}{{OzTemp Well temperature data}},
\newblock \bibinfo{journal}{Geoscience Australia http://www. ga. gov. au
  Catalogue}  (\bibinfo{year}{2010}).
\bibitem[{Maday et~al.(2015)Maday, Anthony, Penn, and Yano}]{Maday2015f}
\bibinfo{author}{Y.~Maday}, \bibinfo{author}{T.~Anthony},
  \bibinfo{author}{J.~D. Penn}, \bibinfo{author}{M.~Yano},
\newblock \bibinfo{title}{{PBDW state estimation: Noisy observations;
  configuration-adaptive background spaces; physical interpretations}},
\newblock \bibinfo{journal}{ESAIM: Proceedings and Surveys}
  \bibinfo{volume}{50} (\bibinfo{year}{2015}) \bibinfo{pages}{144--168}.
\bibitem[{Taddei(2017)}]{Taddei2017d}
\bibinfo{author}{T.~Taddei}, \bibinfo{title}{{Model order reduction methods for
  data assimilation: state estimation and structural health monitoring}}, Ph.D.
  thesis, Massachusetts Institute of Technology, \bibinfo{year}{2017}.

\end{thebibliography}
